\documentclass[12pt,a4paper]{amsart}
\usepackage{amsmath,amsthm,amsfonts,amssymb,mathrsfs}
\usepackage{subfigure} 
\usepackage{graphicx}
\usepackage{color}
\usepackage{caption}

\newcommand\shorter[1]{}

\newtheorem{theorem}{Theorem}

\newtheorem{lemma}[theorem]{Lemma}
\newtheorem{prop}[theorem]{Proposition}

\newtheorem{definition}[theorem]{Definition}

\numberwithin{equation}{section}
\numberwithin{theorem}{section}
\numberwithin{figure}{section}

\theoremstyle{remark}

\usepackage{delarray}
\usepackage{enumerate}

\begin{document}

\title[Spike patterns]{Spike patterns in a reaction-diffusion-ode model with Turing instability}

\author{Steffen H\"arting}
\author{Anna Marciniak-Czochra}
\address[Steffen H\"arting and Anna Marciniak-Czochra]{
Institute of Applied Mathematics,  Interdisciplinary Center for Scientific Computing (IWR) and BIOQUANT, University of Heidelberg, 69120 Heidelberg, Germany}
\email{anna.marciniak@iwr.uni-heidelberg.de}
\urladdr {http://www.biostruct.uni-hd.de/}

\date{\today}

\begin{abstract}
We explore
a mechanism of pattern formation arising in processes described by a system of a single
reaction-diffusion equation coupled with ordinary differential equations.
Such
systems of equations arise from the modeling of interactions between cellular processes
and diffusing growth factors. We focused on the model of early carcinogenesis proposed
by Marciniak-Czochra and Kimmel, which is an example of a wider class of pattern formation models with an autocatalytic non-diffusing
component.  We present a numerical study showing emergence of periodic and irregular spike patterns due to diffusion-driven instability.
To control the accuracy of simulations, we develop a numerical code based on the finite
element method and adaptive mesh. 
Simulations, supplemented by numerical analysis, indicate a novel
pattern formation phenomenon based on the emergence of nonstationary structures tending
asymptotically to the sum of Dirac deltas.


\end{abstract}
\maketitle

\section{Introduction}

Classical mathematical models of biological or chemical pattern formation
have been developed using reaction-diffusion
equations, see eg. \cite{G-M,Epstein,MC12r,Murray,Epstein2} and references therein.  In that framework there exist essentially two mechanisms of formation of stable spatially heterogeneous structures,
\begin{itemize}
\item diffusion-driven instability (DDI) which leads to destabilization of a spatially homogeneous steady state and emergence of Turing patterns,
\item a mechanism based on the multistability and hysteresis in the kinetic system which allows for the formation of transition layer patterns far from equilibrium.
\end{itemize} 
Both mechanisms can also coexist yielding a complex dynamics of the system as, for example, in the Lengyel-Epstein model of chemical reactions \cite{Epstein,Epstein2}.

The Turing phenomenon is related to a local behavior of solutions of a reaction-diffusion system in the neighborhood of  a constant solution that is destabilized via diffusion. Patterns arising through a bifurcation can be spatially monotone or spatially periodic. The mechanism responsible for such behavior of
model solutions is called a diffusion-driven instability (Turing-type instability), which can be formulated in the following way.
  \begin{definition}[Diffusion-driven instability (DDI)]
  A system of reaction-diffusion equations exhibits DDI (Turing instability)  if and only if  there exists a constant stationary solution which is stable to spatially homogeneous perturbations, but unstable to spatially heterogeneous perturbations.
  \end{definition} 
The original idea was presented by Turing on the example of two linear reaction-diffusion equations \cite{Turing}. Due to the local character of Turing instability, the notion has been  extended in a natural way to the nonlinear equations using linearization around a constant positive steady state. However, nonlinear systems  may have multiple constant steady states yielding existence of heterogeneous structures far from the equilibrium. In such cases, global behavior of the solutions cannot be predicted by the properties of the linearized system and a variety of possible dynamics depending on the type of nonlinearities can be observed. On the other hand, Turing instability can be exhibited also in degenerated systems such as reaction-diffusion-odes models or integro-differential equations, for example shadow systems obtained through reduction of the reaction-diffusion model, \cite{Keener}, \cite{Nishiura82}.
 Following all these observations and the character of Turing's original system, we define Turing patterns in the following way.
\begin{definition}
By Turing patterns we call the solutions of reaction-diffusion equations that are
\begin{itemize}
\item stable,
\item stationary,
\item continuous,
\item spatially heterogeneous and
\item arise due to the Turing instability (DDI) of a constant steady state.  
\end{itemize}
\end{definition}

Recently, it has been shown that if DDI property is exhibited  by a system of a  single reaction-diffusion equation coupled to an ordinary differential equation with autocatalysis of non-diffusing component. Then, it does not lead to Turing patterns, namely all continuous patterns are unstable \cite{MCKS13}. As a consequence the question for the long-term behavior of solutions arises.
It has been previously shown that a diffusion-driven blow-up in systems of reaction-diffusion equations can occur in finite time, \cite{Yanagida}. Even more, blow-up in finite time in $L^{\infty}$, but global existence of weaker solutions has been shown, leading to so called "incomplete blow-up", see e.g. \cite{Pierre} for uniform boundedness in $L^1$. 

In the current paper we present a phenomenon of diffusion-driven unbounded growth and formation of dynamic spike pattern converging asymptotically to a sum of Dirac deltas. For a reaction-diffusion-ode model arising from applications in biology, we show that introducing diffusion in the ODEs system not only destabilizes the constant steady state, but also leads to an unbounded growth of model solutions. Since the solutions of the system with zero diffusion are uniformly bounded, we call the observed phenomenon the {\it diffusion-driven unbounded growth}.  The total mass ($L^1$ norm) of the solutions is uniformly bounded but it concentrates in isolated points for time tending to infinity. Using numerical simulations, we investigate how the shape of emerging patterns depends on initial conditions and the scaling coefficient (size of diffusion versus domain size). Interestingly, we find out that the shape of observed patterns are superposition of a near-equilibrium effect of diffusion-driven instability and a far-from-equilibrium effect of multistability exhibited by the model. 
\section{Problem formulation}
We study a reaction-diffusion-odes model of the diffusion-regulated growth of cell population, which has the form
of two ordinary-partial differential equations 
\begin{align}
u_t&=\Big(a_1\frac{uw}{1+uw} -d_1\Big) u &\text{for}\ x\in [0,1], \; t>0, \label{eq1-qs}\\
w_t &= D_w w_{xx} - w - u^2 w +\kappa_1  &\text{for}\ x\in (0,1), \; t>0, \label{eq2-qs}
\end{align}
supplemented with homogeneous Neumann (zero flux) boundary conditions for the function $w=w(x,t)$
\begin{equation}\label{N1}
w_x(0,t)=w_x(1,t)=0 \quad \mbox{for all} \quad t>0,
\end{equation}
and with nonnegative initial conditions
\begin{equation}
u(x,0)=u_0(x), \quad w(x,0)=w_0(x).\label{ini-qs}
\end{equation}
$a_1, d_1, D_w, \kappa_1$  denote positive constants.\\

In this paper we focus on one-dimensional domain $[0,1]$ for a clarity of presentation. The results can be obtained also for a model defined on two-dimensional space domain. Obviously, in such case a structure of spatial patterns is richer. Nevertheless, the main aspect of the pattern formation phenomenon exhibited by this model, i.e. evolution of spike patterns of mass concentration, is preserved independent on the dimension of the spatial domain. 
Model \eqref{eq1-qs}-\eqref{ini-qs} is a rescaled reduction of the model
\begin{align}
u_t&=\Big(a\frac{v}{u+v} -d_{c}\Big) u,                                  & \text{for}\ x\in [0,1], \; t>0, \label{eq1}\\
v_t &=-d_b v +\alpha u^2 w -d v,                                         & \text{for}\ x\in [0,1], \; t>0, \label{eq2}\\
w_t &= \frac{1}{\gamma} w_{xx} -d_g w -\alpha u^2 w +d v +\kappa,  \quad &\text{for}\ x\in (0,1), \; t>0 \label{eq3},
\end{align}
supplemented with homogeneous Neumann (zero flux) boundary conditions for the function $w=w(x,t)$.\\
Model \eqref{eq1}-\eqref{eq3} was proposed in \cite{MCK06} as a receptor-based model of spatially distributed growth of a clonal population of pre-cancerous cells and its extensions and modifications were studied in \cite{MCK07,MCK08}. 
The reduction was proposed in \cite{H2011}, but without further numerical or analytical investigation.\\
In case of the spatial domain being the unit square, approximation of solutions of model \eqref{eq1}-\eqref{eq3} have been performed in \cite{H2011}.
Numerical simulations of the models showed qualitatively new patterns of behavior of solutions, including, in some cases, a strong dependence of the emerging pattern on initial conditions and quasi-stability followed by rapid growth of solutions.  However, recently it has been shown using linear stability analysis of nonconstant steady states that all stationary solutions of this model, both continuous and discontinuous, are unstable \cite{MCKS12}. A question arises if the model exhibits a formation of any pattern, which persist for long times. Our present research is focused on understating these phenomena and answering questions on pattern formation in such class of models.
\section{Analytical results}
In the remainder of this paper, we consider the system \eqref{eq1-qs}-\eqref{ini-qs}. It has been obtained using a quasi-stationary approximation assuming that the dynamics of  $v$ variable is faster than the dynamics of other variables. In the present paper, we focus on the reduced model, since it is the simplest reaction-diffusion-ode model exhibiting the spike pattern formation mechanism. A rigorous link between the solutions of the original model \eqref{eq1}-\eqref{eq3} and its two-equations approximation has been recently shown in \cite{MCKS13}.
\subsection{Existence of solutions}
Existence of global, classical solutions can be proven within the framework of ordinary differential equations and the theory of linear semigroups, see e.g. \cite{R84,Smoller}. Moreover, it can be shown using maximum principle that the solutions remain positive for positive initial conditions.\\
\subsection{Existence of steady states}
The analytical results concerning existence of regular stationary patterns of \eqref{eq1-qs}-\eqref{eq2-qs} can be summarized in the following theorem:\\
\begin{theorem}\label{thm:C2}
Under assumptions $a_1>d_1$ and $\kappa_1>2\frac{d_1}{a_1-d_1}$, system \eqref{eq1-qs}-\eqref{eq2-qs} has the following smooth stationary solutions
\begin{itemize}
\item constant steady states $(\overline{u}_0,\overline{w}_0)=(0,\kappa_1), (\overline{u}_+,\overline{w}_+)=(\frac{d_1}{a_1-d_1} \frac{1}{\overline{w}_+}, \frac{\kappa_1}{2} + \sqrt{(\frac{\kappa_1}{2})^2-(\frac{d_1}{a_1-d_1})^2})$ and 
$(\overline{u}_-,\overline{w}_-)=(\frac{d_1}{a_1-d_1} \frac{1}{\overline{w}_-}, \frac{\kappa_1}{2} - \sqrt{(\frac{\kappa_1}{2})^2-(\frac{d_1}{a_1-d_1})^2})$  being stationary solutions of the kinetic system. 
\item a unique strictly increasing solution $W$ and a unique strictly decreasing solution $W$; $U$ is defined by $U=\frac{d_1}{a_1-d_1} \frac{1}{W}$.
\item a periodic solution $W$ with $n$ modes, increasing on intervals $[0, \frac{1}{n}]$ and its symmetric counterpart $\widetilde W(x)\equiv W_n(1-x)$, where $n \in \mathbb{N}$ depends on the diffusion coefficient; and the periodic function  $W\in C([0,1])$ is defined in the following
\end{itemize}
$$
W(x) =
\left\{
\begin{array}{ccc}
W\left(x-\frac{2j}{n}\right)& \; \text{dla}\;  & x \in \left[\frac{2j}{n}, \frac{2j+1}{n}\right]\vspace{0.3cm}\\
W\left(\frac{2j+2}{n}-x \right)& \; \text{dla}\;  & x \in \left[\frac{2j+1}{n}, \frac{2j+2}{n}\right]
\end{array}
\right.
$$
for every $j\in \{0,1,2,3,...\}$ such that $2j+2\leq n$. $\widetilde{U}$ is defined by $U=\frac{d_1}{a_1-d_1} \frac{1}{W}$.\\
\end{theorem}
The proof of this statement is deferred to the Appendix.
\subsection{Stability of steady states}
We investigate stability of the solutions described in Theorem \ref{thm:C2}, item (i):\\
The operator resulting from linearization of \eqref{eq1-qs}-\eqref{eq2-qs} around $(\frac{d_1}{a_1-d_1} \frac{1}{w}, w)$ reads in the matrix-form:
\begin{equation}
J:=
\begin{bmatrix}
a_{11} & a_{12} \\
a_{21} & a_{22}+D_w \Delta 
\end{bmatrix}
=\begin{bmatrix}
\frac{a_1-d_1}{a_1}d_1 & \frac{d_1^2}{a_1 w^2} \\
-2\frac{d1}{a_1-d_1}   & -\big(1+\big(\frac{d_1}{(a_1-d_1)w}\big)^2\big) + D_w \Delta
\end{bmatrix}
\end{equation}
Assume that a solution of $\frac{d}{dt} \phi=J\phi$ with homogeneous Neumann boundary conditions is of the form $\phi=\phi_k$, where $\phi_k$ denotes the eigenvector of the Laplace operator associated to the $k$th eigenvalue. Then, the {\it dispersion relation}, i.e. the dependence of eigenvalues of the  problem  linearized at a constant steady states with the eigenvalues of the  Laplace operator, see Fig. \ref{fig:DR-0}, is defined by
\begin{equation}
\text{disp}(\lambda,k)=\operatorname{det}\left( \begin{bmatrix}
a_{11}-\lambda & a_{12} \\
a_{21} & a_{22}-D_w (\pi k)^2 - \lambda
\end{bmatrix}\right),
\end{equation}
where $(\pi k)^2$ is the $k$-th eigenvalue of the Laplace operator considered on $C^2(0,1)$.
Therefore, $\lambda$ is an element of the point spectrum of $J$ if $\text{disp}(\lambda,k)=0$ for $\lambda \neq a_{11},0$. 
\begin{prop}\label{thm:DDI}
Under assumptions $a_1>d_1$ and $\kappa_1>2 \frac{d_1}{a_1-d_1}$, the following holds
\begin{itemize}
 \item $(\overline{u}_0,\overline{w}_0)$ is a stable stationary solution of \eqref{eq1-qs}-\eqref{eq2-qs} and its kinetic system.
 \item $(\overline{u}_+,\overline{w}_+)$ is an unstable stationary solution of \eqref{eq1-qs}-\eqref{eq2-qs} and its kinetic system.
 \item $(\overline{u}_-,\overline{w}_-)$ is an unstable stationary solution of \eqref{eq1-qs}-\eqref{eq2-qs}.
 \item $(\overline{u}_-,\overline{w}_-)$ is a stable stationary solution of the kinetic system of \eqref{eq1-qs}-\eqref{eq2-qs} if and only if at least one of the following conditions is satisfied: \\
 1) $\kappa_1^2>2 \frac{d_1^3}{a_1(a_1-d_1)}$ \\
 2) $a_1>\frac{d_1^2}{d_1-1}$ and $\kappa_1^2 > \frac{d_1^4}{a_1} \frac{1}{a_1-d_1(a_1-d_1)}$.
\end{itemize}
Additionally, there exist infinitely many positive eigenvalues of the operator resulting from a linearization of \eqref{eq1-qs}-\eqref{eq2-qs} at $(\overline{u}_-,\overline{w}_-)$ and $\frac{a_1-d_1}{a_1}d_1$ and $-\infty$ are their only limit points.
\end{prop}
The proof of this proposition can be found in the Appendix.\\
A solution with initial conditions close to $(\overline{u}_0,\overline{w}_0)$ is shown in Figure \ref{fig:trivstab}.\\
\begin{minipage}[t]{40em}
\begin{center}
\begin{tabular}{p{15em}p{15em}}
   \centering
  \makeatletter{}\begingroup
  \makeatletter
  \providecommand\color[2][]{    \GenericError{(gnuplot) \space\space\space\@spaces}{      Package color not loaded in conjunction with
      terminal option `colourtext'    }{See the gnuplot documentation for explanation.    }{Either use 'blacktext' in gnuplot or load the package
      color.sty in LaTeX.}    \renewcommand\color[2][]{}  }  \providecommand\includegraphics[2][]{    \GenericError{(gnuplot) \space\space\space\@spaces}{      Package graphicx or graphics not loaded    }{See the gnuplot documentation for explanation.    }{The gnuplot epslatex terminal needs graphicx.sty or graphics.sty.}    \renewcommand\includegraphics[2][]{}  }  \providecommand\rotatebox[2]{#2}  \@ifundefined{ifGPcolor}{    \newif\ifGPcolor
    \GPcolorfalse
  }{}  \@ifundefined{ifGPblacktext}{    \newif\ifGPblacktext
    \GPblacktexttrue
  }{}    \let\gplgaddtomacro\g@addto@macro
    \gdef\gplbacktext{}  \gdef\gplfronttext{}  \makeatother
  \ifGPblacktext
        \def\colorrgb#1{}    \def\colorgray#1{}  \else
        \ifGPcolor
      \def\colorrgb#1{\color[rgb]{#1}}      \def\colorgray#1{\color[gray]{#1}}      \expandafter\def\csname LTw\endcsname{\color{white}}      \expandafter\def\csname LTb\endcsname{\color{black}}      \expandafter\def\csname LTa\endcsname{\color{black}}      \expandafter\def\csname LT0\endcsname{\color[rgb]{1,0,0}}      \expandafter\def\csname LT1\endcsname{\color[rgb]{0,1,0}}      \expandafter\def\csname LT2\endcsname{\color[rgb]{0,0,1}}      \expandafter\def\csname LT3\endcsname{\color[rgb]{1,0,1}}      \expandafter\def\csname LT4\endcsname{\color[rgb]{0,1,1}}      \expandafter\def\csname LT5\endcsname{\color[rgb]{1,1,0}}      \expandafter\def\csname LT6\endcsname{\color[rgb]{0,0,0}}      \expandafter\def\csname LT7\endcsname{\color[rgb]{1,0.3,0}}      \expandafter\def\csname LT8\endcsname{\color[rgb]{0.5,0.5,0.5}}    \else
            \def\colorrgb#1{\color{black}}      \def\colorgray#1{\color[gray]{#1}}      \expandafter\def\csname LTw\endcsname{\color{white}}      \expandafter\def\csname LTb\endcsname{\color{black}}      \expandafter\def\csname LTa\endcsname{\color{black}}      \expandafter\def\csname LT0\endcsname{\color{black}}      \expandafter\def\csname LT1\endcsname{\color{black}}      \expandafter\def\csname LT2\endcsname{\color{black}}      \expandafter\def\csname LT3\endcsname{\color{black}}      \expandafter\def\csname LT4\endcsname{\color{black}}      \expandafter\def\csname LT5\endcsname{\color{black}}      \expandafter\def\csname LT6\endcsname{\color{black}}      \expandafter\def\csname LT7\endcsname{\color{black}}      \expandafter\def\csname LT8\endcsname{\color{black}}    \fi
  \fi
  \setlength{\unitlength}{0.0500bp}  \begin{picture}(3600.00,3528.00)    \gplgaddtomacro\gplfronttext{      \csname LTb\endcsname      \put(726,704){\makebox(0,0)[r]{\strut{}-0.5}}      \put(726,920){\makebox(0,0)[r]{\strut{}-0.4}}      \put(726,1137){\makebox(0,0)[r]{\strut{}-0.3}}      \put(726,1353){\makebox(0,0)[r]{\strut{}-0.2}}      \put(726,1569){\makebox(0,0)[r]{\strut{}-0.1}}      \put(726,1786){\makebox(0,0)[r]{\strut{} 0}}      \put(726,2002){\makebox(0,0)[r]{\strut{} 0.1}}      \put(726,2218){\makebox(0,0)[r]{\strut{} 0.2}}      \put(726,2434){\makebox(0,0)[r]{\strut{} 0.3}}      \put(726,2651){\makebox(0,0)[r]{\strut{} 0.4}}      \put(726,2867){\makebox(0,0)[r]{\strut{} 0.5}}      \put(858,484){\makebox(0,0){\strut{} 0}}      \put(1327,484){\makebox(0,0){\strut{} 1}}      \put(1796,484){\makebox(0,0){\strut{} 2}}      \put(2265,484){\makebox(0,0){\strut{} 3}}      \put(2734,484){\makebox(0,0){\strut{} 4}}      \put(3203,484){\makebox(0,0){\strut{} 5}}      \put(2030,154){\makebox(0,0){\strut{}$k$}}      \put(2030,2607){\makebox(0,0){\strut{}$\lambda_+$}}    }    \gplgaddtomacro\gplfronttext{    }    \gplbacktext
    \put(0,0){\includegraphics{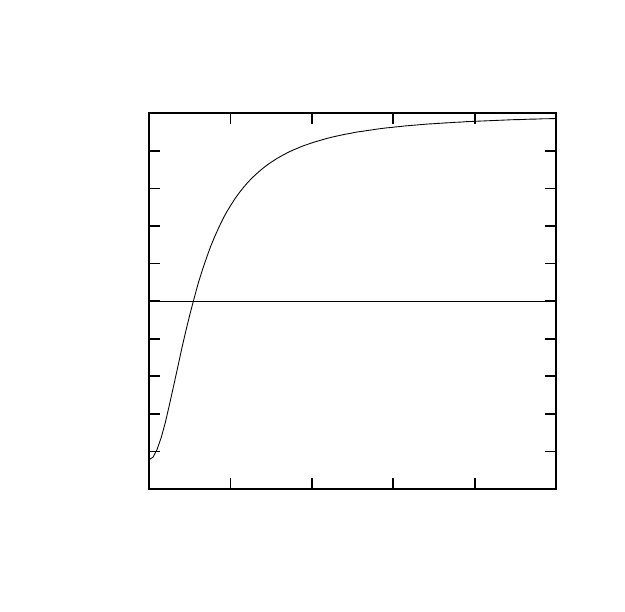}}    \gplfronttext
  \end{picture}\endgroup
 
&
   \centering
    \makeatletter{}\begingroup
  \makeatletter
  \providecommand\color[2][]{    \GenericError{(gnuplot) \space\space\space\@spaces}{      Package color not loaded in conjunction with
      terminal option `colourtext'    }{See the gnuplot documentation for explanation.    }{Either use 'blacktext' in gnuplot or load the package
      color.sty in LaTeX.}    \renewcommand\color[2][]{}  }  \providecommand\includegraphics[2][]{    \GenericError{(gnuplot) \space\space\space\@spaces}{      Package graphicx or graphics not loaded    }{See the gnuplot documentation for explanation.    }{The gnuplot epslatex terminal needs graphicx.sty or graphics.sty.}    \renewcommand\includegraphics[2][]{}  }  \providecommand\rotatebox[2]{#2}  \@ifundefined{ifGPcolor}{    \newif\ifGPcolor
    \GPcolorfalse
  }{}  \@ifundefined{ifGPblacktext}{    \newif\ifGPblacktext
    \GPblacktexttrue
  }{}    \let\gplgaddtomacro\g@addto@macro
    \gdef\gplbacktext{}  \gdef\gplfronttext{}  \makeatother
  \ifGPblacktext
        \def\colorrgb#1{}    \def\colorgray#1{}  \else
        \ifGPcolor
      \def\colorrgb#1{\color[rgb]{#1}}      \def\colorgray#1{\color[gray]{#1}}      \expandafter\def\csname LTw\endcsname{\color{white}}      \expandafter\def\csname LTb\endcsname{\color{black}}      \expandafter\def\csname LTa\endcsname{\color{black}}      \expandafter\def\csname LT0\endcsname{\color[rgb]{1,0,0}}      \expandafter\def\csname LT1\endcsname{\color[rgb]{0,1,0}}      \expandafter\def\csname LT2\endcsname{\color[rgb]{0,0,1}}      \expandafter\def\csname LT3\endcsname{\color[rgb]{1,0,1}}      \expandafter\def\csname LT4\endcsname{\color[rgb]{0,1,1}}      \expandafter\def\csname LT5\endcsname{\color[rgb]{1,1,0}}      \expandafter\def\csname LT6\endcsname{\color[rgb]{0,0,0}}      \expandafter\def\csname LT7\endcsname{\color[rgb]{1,0.3,0}}      \expandafter\def\csname LT8\endcsname{\color[rgb]{0.5,0.5,0.5}}    \else
            \def\colorrgb#1{\color{black}}      \def\colorgray#1{\color[gray]{#1}}      \expandafter\def\csname LTw\endcsname{\color{white}}      \expandafter\def\csname LTb\endcsname{\color{black}}      \expandafter\def\csname LTa\endcsname{\color{black}}      \expandafter\def\csname LT0\endcsname{\color{black}}      \expandafter\def\csname LT1\endcsname{\color{black}}      \expandafter\def\csname LT2\endcsname{\color{black}}      \expandafter\def\csname LT3\endcsname{\color{black}}      \expandafter\def\csname LT4\endcsname{\color{black}}      \expandafter\def\csname LT5\endcsname{\color{black}}      \expandafter\def\csname LT6\endcsname{\color{black}}      \expandafter\def\csname LT7\endcsname{\color{black}}      \expandafter\def\csname LT8\endcsname{\color{black}}    \fi
  \fi
  \setlength{\unitlength}{0.0500bp}  \begin{picture}(3600.00,3528.00)    \gplgaddtomacro\gplfronttext{      \csname LTb\endcsname      \put(726,704){\makebox(0,0)[r]{\strut{}-600}}      \put(726,1065){\makebox(0,0)[r]{\strut{}-500}}      \put(726,1425){\makebox(0,0)[r]{\strut{}-400}}      \put(726,1786){\makebox(0,0)[r]{\strut{}-300}}      \put(726,2146){\makebox(0,0)[r]{\strut{}-200}}      \put(726,2507){\makebox(0,0)[r]{\strut{}-100}}      \put(726,2867){\makebox(0,0)[r]{\strut{} 0}}      \put(858,484){\makebox(0,0){\strut{} 0}}      \put(1327,484){\makebox(0,0){\strut{} 1}}      \put(1796,484){\makebox(0,0){\strut{} 2}}      \put(2265,484){\makebox(0,0){\strut{} 3}}      \put(2734,484){\makebox(0,0){\strut{} 4}}      \put(3203,484){\makebox(0,0){\strut{} 5}}      \put(2030,154){\makebox(0,0){\strut{}$k$}}      \put(2730,2007){\makebox(0,0){\strut{}$\lambda_-$}}    }    \gplgaddtomacro\gplfronttext{    }    \gplbacktext
    \put(0,0){\includegraphics{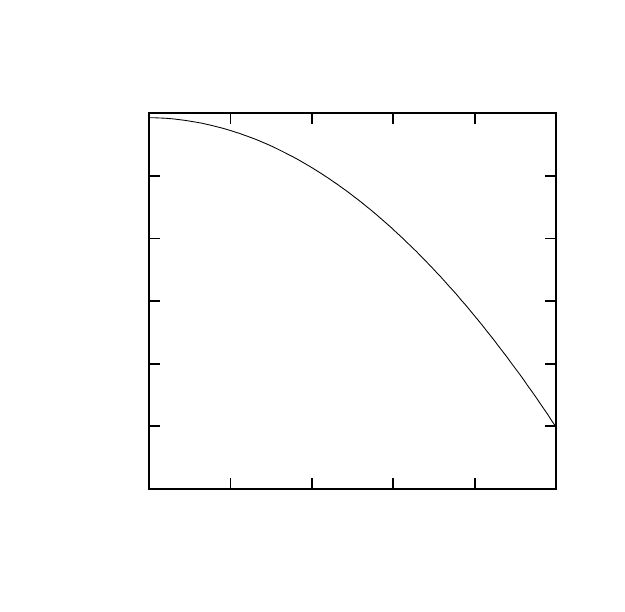}}    \gplfronttext
  \end{picture}\endgroup
 
\end{tabular}
\end{center}
\captionof{figure}{ Roots of the dispersion relation for the operator resulting from a linearization of \eqref{eq1-qs}-\eqref{eq2-qs} around $(\overline{u}_-,\overline{w}_-)$. The parameters are $a_1=2,d_1=1,\kappa_1=3, D_w=2$. left: $\lambda_+$. right: $\lambda_-$. We see that there exist infinitely many positive eigenvalues.}
  \label{fig:DR-0}
\end{minipage}\\ $\,$ \\
Moreover, all steady states except $(\overline{u}_0,\overline{w}_0)$ are linearly unstable. The latter results directly from Theorem 2.1 and Corollary 2.7  in \cite{MCKS13} due to autocatalysis of $u$ for all $u,w\geq 0$.
\begin{theorem}\label{thm:instab}
Under assumptions $a_1>d_1$ and $\kappa_1>2 \frac{d_1}{a_1-d_1}$, 
all steady states $(U,W) \in L^{\infty}(0,1) \times H^1(0,1)$ of \eqref{eq1-qs}-\eqref{eq2-qs} with $U>0$ on a non-zero-measure set are linearly unstable.
\end{theorem}
\subsection{Boundedness properties}

Numerical solutions of system \eqref{eq1-qs}-\eqref{eq2-qs} presented in the following chapter show unbounded growth of spikes. To understand the underlying phenomenon, we summarize here results on the boundedness of solutions of the model with and without diffusion.  It can be easily shown that the mass of solutions of system \eqref{eq1-qs}-\eqref{eq2-qs} is uniformly bounded in time, see Lemma \ref{thm:mass}. Therefore, unbounded growth of the solutions may happen at most  in isolated points of the spatial domain.  Furthermore, to exclude blow-up induced by unbounded solutions of the kinetic system, such as shown in \cite{Ball}, we check boundedness properties of the kinetic system, see Lemma \ref{thm:bounded}.
\begin{lemma}\label{thm:mass}
Let $u(x,t),w(x,t)$ denote a solution of \eqref{eq1-qs}-\eqref{eq2-qs} for positive initial conditions. Then, it holds
\begin{align}
 \limsup_{t \rightarrow \infty}\big( \frac{1}{a_1}\left\|u(t)\right\|_{L^1}+\left\|w(t)\right\|_{L^1}\big) &\leq \frac{\kappa_1}{\min(d_1,1)},\\
 \limsup_{t \rightarrow \infty} \left\|u(t)\right\|_{L^1} &\leq \frac{a_1}{\min(d_1,1)} \kappa_1,\\
 \limsup_{t \rightarrow \infty} \left\|w(t)\right\|_{L^1} &\leq \kappa_1.
\end{align}
\end{lemma}
Moreover, the solution of the kinetic system of \eqref{eq1-qs}-\eqref{eq2-qs} is uniformly bounded in time,
\begin{lemma}\label{thm:bounded}
Let $u(x,t),w(x,t)$ denote a solution of the kinetic system of \eqref{eq1-qs}-\eqref{eq2-qs} for positive initial conditions. Then holds
\begin{align}
 \limsup_{t \rightarrow \infty} u(t)&\leq \frac{a_1}{\min(d_1,1)} \kappa_1,\\
 \limsup_{t \rightarrow \infty} w(t)&\leq \kappa_1.
\end{align}
\end{lemma}
Both lemmas can be proven similarly as it was shown in \cite{MCKS12} for the three equation model. More details are deferred to the Appendix.

We conclude from Lemma \ref{thm:bounded} that the mass concentration observed in numerical simulations  does not result from a blow-up of the solution of the kinetic system.

\section{Numerical approach}
Numerical approximations of solutions to system \eqref{eq1-qs}-\eqref{eq2-qs} presented in this paper are obtained using the program library \textit{deal.ii}, \cite{deal.ii}.\\
Simulations using adaptive grid refinement based on cell-wise evaluation of the proposed error indicators in \cite{ELS2000} show a growth of spikes, see Figure \ref{fig:1s}.\\
The question of what is seen in numerical simulations motivated us to undertake a numerical study of the pattern formation phenomenon.
To allow a rigorous argumentation using classical finite-element analysis, we investigate the asymptotic behavior of the numerical solution for spatially homogeneous meshes.\\
For space discretization, we use a finite-element scheme with piecewise linear, globally continuous ansatzfunctions. The time discretization is performed using the implicit Euler scheme or the Crank-Nicholson scheme.\\
Convergence for such scheme for solutions of systems of type \eqref{eq1-qs}-\eqref{eq2-qs} is well known, see \cite{ELS2000}.\\
The low order of the space discretization is due to the fact that preliminary simulations already showed emergence of spikes, corresponding to a large second derivative in space.

\section{Numerical analysis of the pattern formation phenomenon}
We choose parameters 
\begin{equation}
a_1=2, d_1=1, \kappa_1=3
\end{equation}
and diffusion coefficient $D_w=6$.\\
A numerically obtained solution for different parameters, $a_1=2.5, d_1=1.5, \kappa_1=4$ is shown in the Appendix in figure \ref{fig:np}. It shows qualitatively the same behavior.
We recall that system \eqref{eq1-qs}-\eqref{eq2-qs} exhibits Turing type diffusion driven instability, but all positive steady states are linearly unstable.\\
\subsection{Unbounded growth and spike formation}
Initial conditions are chosen as a perturbation of the stable stationary solution $(\overline{u},\overline{w})$ of the kinetics system of \eqref{eq1-qs}-\eqref{eq2-qs}:
\begin{equation}\label{IC:sh}
\begin{aligned}
 u_0(x)     &= \overline{u}_- + \epsilon_1 p(x),\\
 w_0(x)     &= \overline{w}_-, \\
\end{aligned}
\end{equation}
where the 'perturbation function' $p(x)$ satisfies the following conditions:
\begin{align}
 p \text{ is a polynomial of degree two on } &(0,s-\epsilon),(s-\epsilon,s+\epsilon),(s+\epsilon,1),\label{ICdef1}\\
 p &\in C^1(0,1),\label{ICdef2}\\
 p'(0)&=p'(1)=0,\label{ICdef3}\\
 p(0)&=p(1)=-1,\label{ICdef4}\\
 p(s)&=1\label{ICdef5},
\end{align}
and is thereby uniquely defined by the pair $(s,\epsilon)$.
The explicit formula for $p$ can be found in the appendix, \eqref{spline}, an illustration can be found in figure \ref{fig:IC}.
\begin{figure}[ht]
  \centering
   \makeatletter{}\begingroup
  \makeatletter
  \providecommand\color[2][]{    \GenericError{(gnuplot) \space\space\space\@spaces}{      Package color not loaded in conjunction with
      terminal option `colourtext'    }{See the gnuplot documentation for explanation.    }{Either use 'blacktext' in gnuplot or load the package
      color.sty in LaTeX.}    \renewcommand\color[2][]{}  }  \providecommand\includegraphics[2][]{    \GenericError{(gnuplot) \space\space\space\@spaces}{      Package graphicx or graphics not loaded    }{See the gnuplot documentation for explanation.    }{The gnuplot epslatex terminal needs graphicx.sty or graphics.sty.}    \renewcommand\includegraphics[2][]{}  }  \providecommand\rotatebox[2]{#2}  \@ifundefined{ifGPcolor}{    \newif\ifGPcolor
    \GPcolorfalse
  }{}  \@ifundefined{ifGPblacktext}{    \newif\ifGPblacktext
    \GPblacktexttrue
  }{}    \let\gplgaddtomacro\g@addto@macro
    \gdef\gplbacktext{}  \gdef\gplfronttext{}  \makeatother
  \ifGPblacktext
        \def\colorrgb#1{}    \def\colorgray#1{}  \else
        \ifGPcolor
      \def\colorrgb#1{\color[rgb]{#1}}      \def\colorgray#1{\color[gray]{#1}}      \expandafter\def\csname LTw\endcsname{\color{white}}      \expandafter\def\csname LTb\endcsname{\color{black}}      \expandafter\def\csname LTa\endcsname{\color{black}}      \expandafter\def\csname LT0\endcsname{\color[rgb]{1,0,0}}      \expandafter\def\csname LT1\endcsname{\color[rgb]{0,1,0}}      \expandafter\def\csname LT2\endcsname{\color[rgb]{0,0,1}}      \expandafter\def\csname LT3\endcsname{\color[rgb]{1,0,1}}      \expandafter\def\csname LT4\endcsname{\color[rgb]{0,1,1}}      \expandafter\def\csname LT5\endcsname{\color[rgb]{1,1,0}}      \expandafter\def\csname LT6\endcsname{\color[rgb]{0,0,0}}      \expandafter\def\csname LT7\endcsname{\color[rgb]{1,0.3,0}}      \expandafter\def\csname LT8\endcsname{\color[rgb]{0.5,0.5,0.5}}    \else
            \def\colorrgb#1{\color{black}}      \def\colorgray#1{\color[gray]{#1}}      \expandafter\def\csname LTw\endcsname{\color{white}}      \expandafter\def\csname LTb\endcsname{\color{black}}      \expandafter\def\csname LTa\endcsname{\color{black}}      \expandafter\def\csname LT0\endcsname{\color{black}}      \expandafter\def\csname LT1\endcsname{\color{black}}      \expandafter\def\csname LT2\endcsname{\color{black}}      \expandafter\def\csname LT3\endcsname{\color{black}}      \expandafter\def\csname LT4\endcsname{\color{black}}      \expandafter\def\csname LT5\endcsname{\color{black}}      \expandafter\def\csname LT6\endcsname{\color{black}}      \expandafter\def\csname LT7\endcsname{\color{black}}      \expandafter\def\csname LT8\endcsname{\color{black}}    \fi
  \fi
  \setlength{\unitlength}{0.0500bp}  \begin{picture}(4320.00,3024.00)    \gplgaddtomacro\gplfronttext{      \csname LTb\endcsname      \put(726,704){\makebox(0,0)[r]{\strut{}-1.5}}      \put(726,1047){\makebox(0,0)[r]{\strut{}-1}}      \put(726,1389){\makebox(0,0)[r]{\strut{}-0.5}}      \put(726,1732){\makebox(0,0)[r]{\strut{} 0}}      \put(726,2074){\makebox(0,0)[r]{\strut{} 0.5}}      \put(726,2417){\makebox(0,0)[r]{\strut{} 1}}      \put(726,2759){\makebox(0,0)[r]{\strut{} 1.5}}      \put(858,484){\makebox(0,0){\strut{} 0}}      \put(1471,484){\makebox(0,0){\strut{} 0.2}}      \put(2084,484){\makebox(0,0){\strut{} 0.4}}      \put(2697,484){\makebox(0,0){\strut{} 0.6}}      \put(3310,484){\makebox(0,0){\strut{} 0.8}}      \put(3923,484){\makebox(0,0){\strut{} 1}}      \put(2390,154){\makebox(0,0){\strut{}x}}    }    \gplgaddtomacro\gplfronttext{      \csname LTb\endcsname      \put(2936,2586){\makebox(0,0)[r]{\strut{}$p(x)$}}    }    \gplbacktext
    \put(0,0){\includegraphics{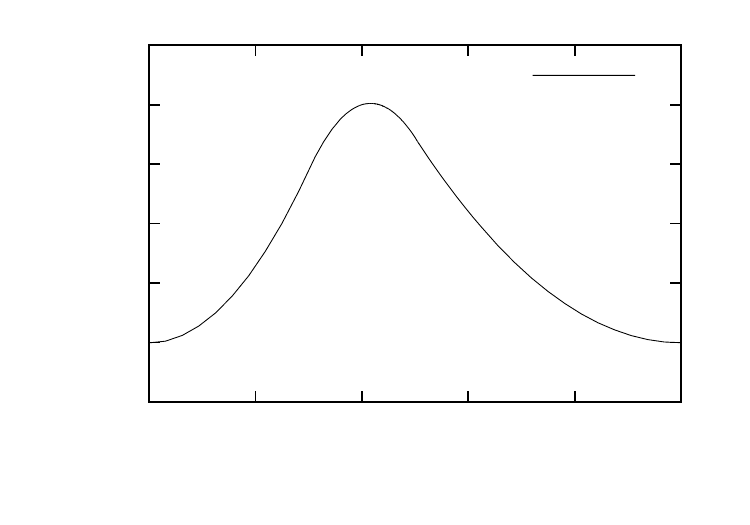}}    \gplfronttext
  \end{picture}\endgroup
 
  \caption{Illustration of the perturbation function $p(x)$, defined by \eqref{ICdef1}-\eqref{ICdef5} for $s=0.4,\epsilon=0.1$. $\max_{x \in \Omega}p(x)$ is always assumed in $(s-\epsilon,s+\epsilon)$.}
  \label{fig:IC}
\end{figure}
In figure \ref{fig:1s}, the solution for initial conditions \eqref{IC:sh} for $s=0.4, \epsilon_1 = 0.05, \epsilon=0.1$ is shown. We observe exponential growth in a single point and decay towards zero otherwise. The maximum value of the numerically obtained solution keeps growing.
\subsection{Spike position and initial conditions}
Simulations performed using the parameters $a_1=2,d_1=1,\kappa_1=3, D_w=6$ and initial conditions \eqref{IC:sh} show emergence of spikes at the maximum of the initial conditions, see Table \ref{tab:s_pos}.
For the linearized problem, this is heuristically reasonable since almost all eigenmodes of the Laplace Operator are unstable with almost the same eigenvalue, see Lemma \ref{lem:EV} or figure \ref{fig:DR-0} for an illustration.
\\
\begin{minipage}[H]{40em}
\begin{center}
\begin{tabular}{p{15em}p{15em}}
   \centering
   \makeatletter{}\begingroup
  \makeatletter
  \providecommand\color[2][]{    \GenericError{(gnuplot) \space\space\space\@spaces}{      Package color not loaded in conjunction with
      terminal option `colourtext'    }{See the gnuplot documentation for explanation.    }{Either use 'blacktext' in gnuplot or load the package
      color.sty in LaTeX.}    \renewcommand\color[2][]{}  }  \providecommand\includegraphics[2][]{    \GenericError{(gnuplot) \space\space\space\@spaces}{      Package graphicx or graphics not loaded    }{See the gnuplot documentation for explanation.    }{The gnuplot epslatex terminal needs graphicx.sty or graphics.sty.}    \renewcommand\includegraphics[2][]{}  }  \providecommand\rotatebox[2]{#2}  \@ifundefined{ifGPcolor}{    \newif\ifGPcolor
    \GPcolorfalse
  }{}  \@ifundefined{ifGPblacktext}{    \newif\ifGPblacktext
    \GPblacktexttrue
  }{}    \let\gplgaddtomacro\g@addto@macro
    \gdef\gplbacktext{}  \gdef\gplfronttext{}  \makeatother
  \ifGPblacktext
        \def\colorrgb#1{}    \def\colorgray#1{}  \else
        \ifGPcolor
      \def\colorrgb#1{\color[rgb]{#1}}      \def\colorgray#1{\color[gray]{#1}}      \expandafter\def\csname LTw\endcsname{\color{white}}      \expandafter\def\csname LTb\endcsname{\color{black}}      \expandafter\def\csname LTa\endcsname{\color{black}}      \expandafter\def\csname LT0\endcsname{\color[rgb]{1,0,0}}      \expandafter\def\csname LT1\endcsname{\color[rgb]{0,1,0}}      \expandafter\def\csname LT2\endcsname{\color[rgb]{0,0,1}}      \expandafter\def\csname LT3\endcsname{\color[rgb]{1,0,1}}      \expandafter\def\csname LT4\endcsname{\color[rgb]{0,1,1}}      \expandafter\def\csname LT5\endcsname{\color[rgb]{1,1,0}}      \expandafter\def\csname LT6\endcsname{\color[rgb]{0,0,0}}      \expandafter\def\csname LT7\endcsname{\color[rgb]{1,0.3,0}}      \expandafter\def\csname LT8\endcsname{\color[rgb]{0.5,0.5,0.5}}    \else
            \def\colorrgb#1{\color{black}}      \def\colorgray#1{\color[gray]{#1}}      \expandafter\def\csname LTw\endcsname{\color{white}}      \expandafter\def\csname LTb\endcsname{\color{black}}      \expandafter\def\csname LTa\endcsname{\color{black}}      \expandafter\def\csname LT0\endcsname{\color{black}}      \expandafter\def\csname LT1\endcsname{\color{black}}      \expandafter\def\csname LT2\endcsname{\color{black}}      \expandafter\def\csname LT3\endcsname{\color{black}}      \expandafter\def\csname LT4\endcsname{\color{black}}      \expandafter\def\csname LT5\endcsname{\color{black}}      \expandafter\def\csname LT6\endcsname{\color{black}}      \expandafter\def\csname LT7\endcsname{\color{black}}      \expandafter\def\csname LT8\endcsname{\color{black}}    \fi
  \fi
  \setlength{\unitlength}{0.0500bp}  \begin{picture}(3600.00,3528.00)    \gplgaddtomacro\gplbacktext{    }    \gplgaddtomacro\gplfronttext{      \csname LTb\endcsname      \put(2135,2804){\makebox(0,0)[r]{\strut{}shape of $u_0$}}      \csname LTb\endcsname      \put(507,1092){\makebox(0,0){\strut{} 0}}      \put(1697,557){\makebox(0,0){\strut{} 25}}      \put(908,769){\makebox(0,0){\strut{}$t$}}      \put(1893,605){\makebox(0,0){\strut{} 0}}      \put(3083,1090){\makebox(0,0){\strut{} 1}}      \put(2692,819){\makebox(0,0){\strut{}$x$}}      \put(484,1656){\makebox(0,0)[r]{\strut{} 0}}      \put(484,2004){\makebox(0,0)[r]{\strut{} 100}}      \put(484,2354){\makebox(0,0)[r]{\strut{} 200}}    }    \gplbacktext
    \put(0,0){\includegraphics{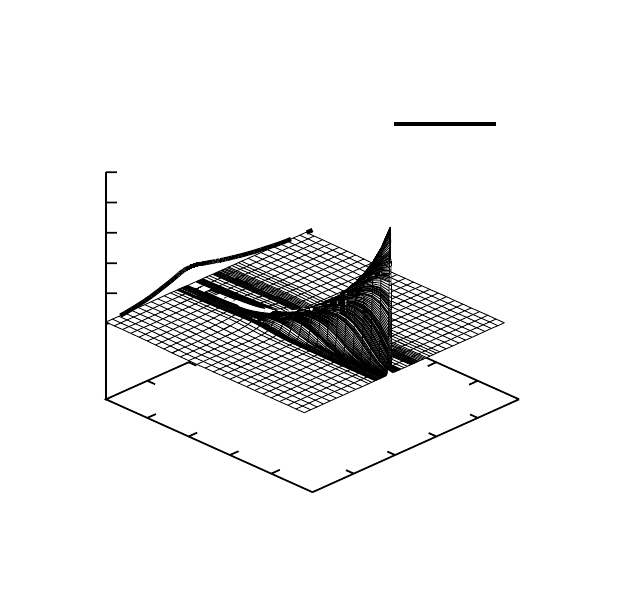}}    \gplfronttext
  \end{picture}\endgroup
 
&
   \centering
   \makeatletter{}\begingroup
  \makeatletter
  \providecommand\color[2][]{    \GenericError{(gnuplot) \space\space\space\@spaces}{      Package color not loaded in conjunction with
      terminal option `colourtext'    }{See the gnuplot documentation for explanation.    }{Either use 'blacktext' in gnuplot or load the package
      color.sty in LaTeX.}    \renewcommand\color[2][]{}  }  \providecommand\includegraphics[2][]{    \GenericError{(gnuplot) \space\space\space\@spaces}{      Package graphicx or graphics not loaded    }{See the gnuplot documentation for explanation.    }{The gnuplot epslatex terminal needs graphicx.sty or graphics.sty.}    \renewcommand\includegraphics[2][]{}  }  \providecommand\rotatebox[2]{#2}  \@ifundefined{ifGPcolor}{    \newif\ifGPcolor
    \GPcolorfalse
  }{}  \@ifundefined{ifGPblacktext}{    \newif\ifGPblacktext
    \GPblacktexttrue
  }{}    \let\gplgaddtomacro\g@addto@macro
    \gdef\gplbacktext{}  \gdef\gplfronttext{}  \makeatother
  \ifGPblacktext
        \def\colorrgb#1{}    \def\colorgray#1{}  \else
        \ifGPcolor
      \def\colorrgb#1{\color[rgb]{#1}}      \def\colorgray#1{\color[gray]{#1}}      \expandafter\def\csname LTw\endcsname{\color{white}}      \expandafter\def\csname LTb\endcsname{\color{black}}      \expandafter\def\csname LTa\endcsname{\color{black}}      \expandafter\def\csname LT0\endcsname{\color[rgb]{1,0,0}}      \expandafter\def\csname LT1\endcsname{\color[rgb]{0,1,0}}      \expandafter\def\csname LT2\endcsname{\color[rgb]{0,0,1}}      \expandafter\def\csname LT3\endcsname{\color[rgb]{1,0,1}}      \expandafter\def\csname LT4\endcsname{\color[rgb]{0,1,1}}      \expandafter\def\csname LT5\endcsname{\color[rgb]{1,1,0}}      \expandafter\def\csname LT6\endcsname{\color[rgb]{0,0,0}}      \expandafter\def\csname LT7\endcsname{\color[rgb]{1,0.3,0}}      \expandafter\def\csname LT8\endcsname{\color[rgb]{0.5,0.5,0.5}}    \else
            \def\colorrgb#1{\color{black}}      \def\colorgray#1{\color[gray]{#1}}      \expandafter\def\csname LTw\endcsname{\color{white}}      \expandafter\def\csname LTb\endcsname{\color{black}}      \expandafter\def\csname LTa\endcsname{\color{black}}      \expandafter\def\csname LT0\endcsname{\color{black}}      \expandafter\def\csname LT1\endcsname{\color{black}}      \expandafter\def\csname LT2\endcsname{\color{black}}      \expandafter\def\csname LT3\endcsname{\color{black}}      \expandafter\def\csname LT4\endcsname{\color{black}}      \expandafter\def\csname LT5\endcsname{\color{black}}      \expandafter\def\csname LT6\endcsname{\color{black}}      \expandafter\def\csname LT7\endcsname{\color{black}}      \expandafter\def\csname LT8\endcsname{\color{black}}    \fi
  \fi
  \setlength{\unitlength}{0.0500bp}  \begin{picture}(3600.00,3528.00)    \gplgaddtomacro\gplbacktext{    }    \gplgaddtomacro\gplfronttext{      \csname LTb\endcsname      \put(507,1092){\makebox(0,0){\strut{} 0}}      \put(1697,557){\makebox(0,0){\strut{} 25}}      \put(908,819){\makebox(0,0){\strut{}$t$}}      \put(1893,615){\makebox(0,0){\strut{} 0}}      \put(3083,1150){\makebox(0,0){\strut{} 1}}      \put(2692,819){\makebox(0,0){\strut{}$x$}}      \put(484,1656){\makebox(0,0)[r]{\strut{} 0}}      \put(484,1874){\makebox(0,0)[r]{\strut{} 0.1}}      \put(484,2092){\makebox(0,0)[r]{\strut{} 0.2}}      \put(484,2310){\makebox(0,0)[r]{\strut{} 0.3}}      \put(484,2528){\makebox(0,0)[r]{\strut{} 0.4}}    }    \gplbacktext
    \put(0,0){\includegraphics{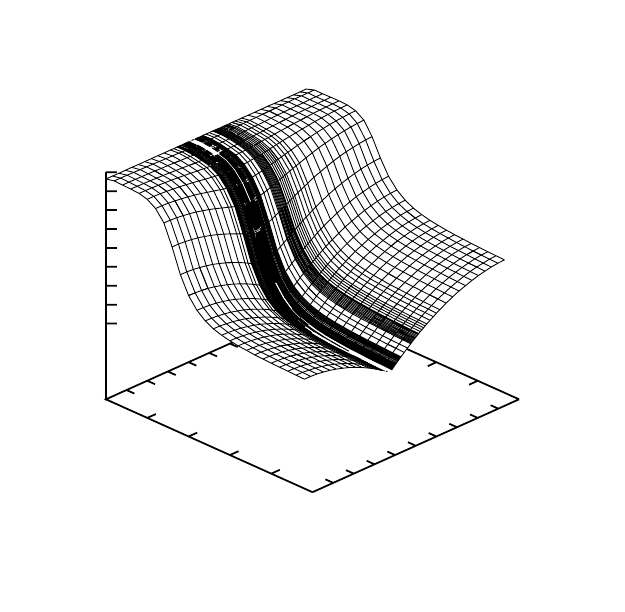}}    \gplfronttext
  \end{picture}\endgroup
 
\end{tabular}
\end{center}
\captionof{figure}{Numerically obtained solution for initial conditions \eqref{IC:sh} with $s=0.4,\epsilon=0.1, \epsilon_1=0.05$ and diffusion coefficient $D_w=6$, $\overline{u}=\frac{3+\sqrt{5}}{2}\approx 2.62,\overline{w}=\frac{3-\sqrt{5}}{2}\approx 0.382$. Left: component $u$. Right: component $w$. We observe formation of a spike at $x \approx 0.43$, which keeps growing exponentially in time.}
\label{fig:1s}
\end{minipage}\\ $\,$ \\ 
\begin{lemma}\label{le:ampl}
Let $J$ denote the operator resulting from the linearization of \eqref{eq1-qs}-\eqref{eq2-qs} around $(\overline{u}_-,\overline{w}_-)$ and consider the initial value problem
\begin{equation}\label{lin}
\frac{d}{dt} \begin{bmatrix} \phi \\ \psi \end{bmatrix} = J \begin{bmatrix} \phi \\ \psi \end{bmatrix},
\end{equation}
with homogeneous Neumann boundary (zero-flux) conditions for $\psi$.\\
As initial conditions take $(\psi_0,\rho_0)=(\phi_k, \frac{a_{21}}{\lambda_+-a_{22}+D_w \pi^2 k^2} \phi_k)$, where $\phi_k$ denotes the eigenfunction of the Laplace-Operator (Neumann) associated to the $k$th eigenvalue and $(a_{ij})$ denotes the Jacobian of the kinetics system at $(\overline{u}_-,\overline{w}_-)$.\\
Then $e^{\lambda_+(k)t}(\psi_0,\rho_0)$ is the solution of \eqref{lin} for homogeneous Neumann boundary conditions.
\end{lemma}
Heuristically speaking, the initial perturbation of $u$ is self-amplifying for large $D_w$ if it is much larger than the perturbation of $w$, because $\lambda_+(k) \rightarrow a_{11}>0$.\\
This heuristic implication leads to the question what happens for more complex initial conditions. In Figure \ref{fig:cos}, we plot the numerical solution for initial conditions
\begin{equation}\label{IC:cos}
\begin{aligned}
 u_0(x)&=\overline{u}_--\epsilon \cos(4 \pi x),\\
 w_0(x)&=\overline{w}_-.
\end{aligned}
\end{equation}
and in figure \ref{fig:cosxx} the numerical solution for initial conditions 
\begin{equation}\label{IC:cosxx}
\begin{aligned}
 u_0(x)&=\overline{u}_--\epsilon \cos(4 \pi x^2),\\
 w_0(x)&=\overline{w}_-.
\end{aligned}
\end{equation}
\begin{minipage}[H]{40em}
\begin{center}
\begin{tabular}{p{15em}p{15em}}
   \centering
   \makeatletter{}\begingroup
  \makeatletter
  \providecommand\color[2][]{    \GenericError{(gnuplot) \space\space\space\@spaces}{      Package color not loaded in conjunction with
      terminal option `colourtext'    }{See the gnuplot documentation for explanation.    }{Either use 'blacktext' in gnuplot or load the package
      color.sty in LaTeX.}    \renewcommand\color[2][]{}  }  \providecommand\includegraphics[2][]{    \GenericError{(gnuplot) \space\space\space\@spaces}{      Package graphicx or graphics not loaded    }{See the gnuplot documentation for explanation.    }{The gnuplot epslatex terminal needs graphicx.sty or graphics.sty.}    \renewcommand\includegraphics[2][]{}  }  \providecommand\rotatebox[2]{#2}  \@ifundefined{ifGPcolor}{    \newif\ifGPcolor
    \GPcolorfalse
  }{}  \@ifundefined{ifGPblacktext}{    \newif\ifGPblacktext
    \GPblacktexttrue
  }{}    \let\gplgaddtomacro\g@addto@macro
    \gdef\gplbacktext{}  \gdef\gplfronttext{}  \makeatother
  \ifGPblacktext
        \def\colorrgb#1{}    \def\colorgray#1{}  \else
        \ifGPcolor
      \def\colorrgb#1{\color[rgb]{#1}}      \def\colorgray#1{\color[gray]{#1}}      \expandafter\def\csname LTw\endcsname{\color{white}}      \expandafter\def\csname LTb\endcsname{\color{black}}      \expandafter\def\csname LTa\endcsname{\color{black}}      \expandafter\def\csname LT0\endcsname{\color[rgb]{1,0,0}}      \expandafter\def\csname LT1\endcsname{\color[rgb]{0,1,0}}      \expandafter\def\csname LT2\endcsname{\color[rgb]{0,0,1}}      \expandafter\def\csname LT3\endcsname{\color[rgb]{1,0,1}}      \expandafter\def\csname LT4\endcsname{\color[rgb]{0,1,1}}      \expandafter\def\csname LT5\endcsname{\color[rgb]{1,1,0}}      \expandafter\def\csname LT6\endcsname{\color[rgb]{0,0,0}}      \expandafter\def\csname LT7\endcsname{\color[rgb]{1,0.3,0}}      \expandafter\def\csname LT8\endcsname{\color[rgb]{0.5,0.5,0.5}}    \else
            \def\colorrgb#1{\color{black}}      \def\colorgray#1{\color[gray]{#1}}      \expandafter\def\csname LTw\endcsname{\color{white}}      \expandafter\def\csname LTb\endcsname{\color{black}}      \expandafter\def\csname LTa\endcsname{\color{black}}      \expandafter\def\csname LT0\endcsname{\color{black}}      \expandafter\def\csname LT1\endcsname{\color{black}}      \expandafter\def\csname LT2\endcsname{\color{black}}      \expandafter\def\csname LT3\endcsname{\color{black}}      \expandafter\def\csname LT4\endcsname{\color{black}}      \expandafter\def\csname LT5\endcsname{\color{black}}      \expandafter\def\csname LT6\endcsname{\color{black}}      \expandafter\def\csname LT7\endcsname{\color{black}}      \expandafter\def\csname LT8\endcsname{\color{black}}    \fi
  \fi
  \setlength{\unitlength}{0.0500bp}  \begin{picture}(3600.00,3528.00)    \gplgaddtomacro\gplbacktext{    }    \gplgaddtomacro\gplfronttext{      \csname LTb\endcsname      \put(2135,2804){\makebox(0,0)[r]{\strut{}shape of $u_0$}}      \csname LTb\endcsname      \put(507,1092){\makebox(0,0){\strut{} 0}}      \put(1697,557){\makebox(0,0){\strut{} 25}}      \put(908,819){\makebox(0,0){\strut{}$t$}}      \put(1893,615){\makebox(0,0){\strut{} 0}}      \put(3083,1150){\makebox(0,0){\strut{} 1}}      \put(2692,819){\makebox(0,0){\strut{}$x$}}      \put(484,1656){\makebox(0,0)[r]{\strut{} 0}}      \put(484,1946){\makebox(0,0)[r]{\strut{} 40}}      \put(484,2237){\makebox(0,0)[r]{\strut{} 80}}      \put(484,2528){\makebox(0,0)[r]{\strut{} 120}}    }    \gplbacktext
    \put(0,0){\includegraphics{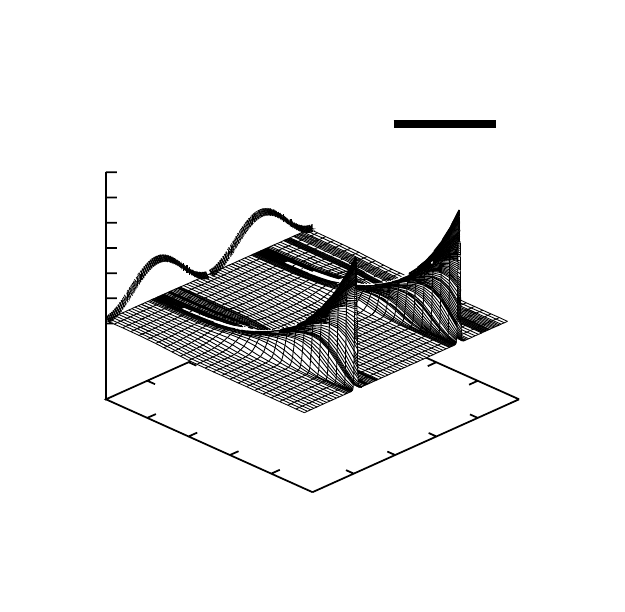}}    \gplfronttext
  \end{picture}\endgroup
 
&
   \centering
   \makeatletter{}\begingroup
  \makeatletter
  \providecommand\color[2][]{    \GenericError{(gnuplot) \space\space\space\@spaces}{      Package color not loaded in conjunction with
      terminal option `colourtext'    }{See the gnuplot documentation for explanation.    }{Either use 'blacktext' in gnuplot or load the package
      color.sty in LaTeX.}    \renewcommand\color[2][]{}  }  \providecommand\includegraphics[2][]{    \GenericError{(gnuplot) \space\space\space\@spaces}{      Package graphicx or graphics not loaded    }{See the gnuplot documentation for explanation.    }{The gnuplot epslatex terminal needs graphicx.sty or graphics.sty.}    \renewcommand\includegraphics[2][]{}  }  \providecommand\rotatebox[2]{#2}  \@ifundefined{ifGPcolor}{    \newif\ifGPcolor
    \GPcolorfalse
  }{}  \@ifundefined{ifGPblacktext}{    \newif\ifGPblacktext
    \GPblacktexttrue
  }{}    \let\gplgaddtomacro\g@addto@macro
    \gdef\gplbacktext{}  \gdef\gplfronttext{}  \makeatother
  \ifGPblacktext
        \def\colorrgb#1{}    \def\colorgray#1{}  \else
        \ifGPcolor
      \def\colorrgb#1{\color[rgb]{#1}}      \def\colorgray#1{\color[gray]{#1}}      \expandafter\def\csname LTw\endcsname{\color{white}}      \expandafter\def\csname LTb\endcsname{\color{black}}      \expandafter\def\csname LTa\endcsname{\color{black}}      \expandafter\def\csname LT0\endcsname{\color[rgb]{1,0,0}}      \expandafter\def\csname LT1\endcsname{\color[rgb]{0,1,0}}      \expandafter\def\csname LT2\endcsname{\color[rgb]{0,0,1}}      \expandafter\def\csname LT3\endcsname{\color[rgb]{1,0,1}}      \expandafter\def\csname LT4\endcsname{\color[rgb]{0,1,1}}      \expandafter\def\csname LT5\endcsname{\color[rgb]{1,1,0}}      \expandafter\def\csname LT6\endcsname{\color[rgb]{0,0,0}}      \expandafter\def\csname LT7\endcsname{\color[rgb]{1,0.3,0}}      \expandafter\def\csname LT8\endcsname{\color[rgb]{0.5,0.5,0.5}}    \else
            \def\colorrgb#1{\color{black}}      \def\colorgray#1{\color[gray]{#1}}      \expandafter\def\csname LTw\endcsname{\color{white}}      \expandafter\def\csname LTb\endcsname{\color{black}}      \expandafter\def\csname LTa\endcsname{\color{black}}      \expandafter\def\csname LT0\endcsname{\color{black}}      \expandafter\def\csname LT1\endcsname{\color{black}}      \expandafter\def\csname LT2\endcsname{\color{black}}      \expandafter\def\csname LT3\endcsname{\color{black}}      \expandafter\def\csname LT4\endcsname{\color{black}}      \expandafter\def\csname LT5\endcsname{\color{black}}      \expandafter\def\csname LT6\endcsname{\color{black}}      \expandafter\def\csname LT7\endcsname{\color{black}}      \expandafter\def\csname LT8\endcsname{\color{black}}    \fi
  \fi
  \setlength{\unitlength}{0.0500bp}  \begin{picture}(3600.00,3528.00)    \gplgaddtomacro\gplbacktext{    }    \gplgaddtomacro\gplfronttext{      \csname LTb\endcsname      \put(507,1092){\makebox(0,0){\strut{} 0}}      \put(1697,557){\makebox(0,0){\strut{} 25}}      \put(908,819){\makebox(0,0){\strut{}$t$}}      \put(1893,615){\makebox(0,0){\strut{} 0}}      \put(3083,1150){\makebox(0,0){\strut{} 1}}      \put(2692,819){\makebox(0,0){\strut{}$x$}}      \put(484,1656){\makebox(0,0)[r]{\strut{} 0}}      \put(484,1874){\makebox(0,0)[r]{\strut{} 0.1}}      \put(484,2092){\makebox(0,0)[r]{\strut{} 0.2}}      \put(484,2310){\makebox(0,0)[r]{\strut{} 0.3}}      \put(484,2528){\makebox(0,0)[r]{\strut{} 0.4}}    }    \gplbacktext
    \put(0,0){\includegraphics{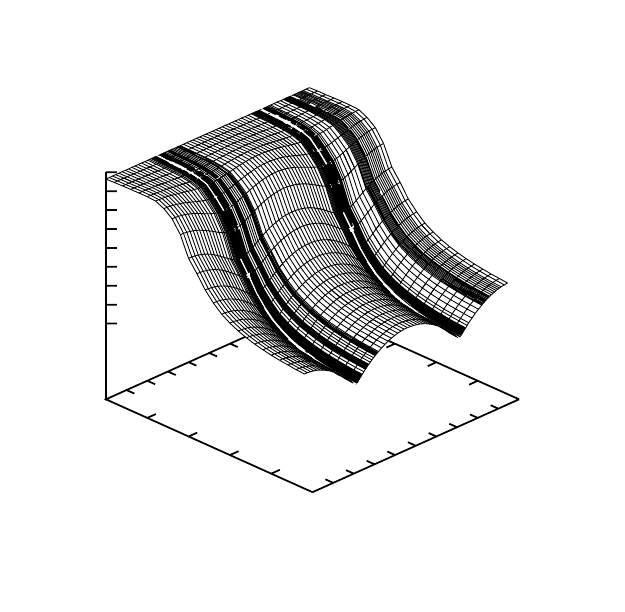}}    \gplfronttext
  \end{picture}\endgroup
 
\end{tabular}
\end{center}
\captionof{figure}{Numerical solution for initial conditions \eqref{IC:cos} with $\epsilon=0.05$ and diffusion coefficient $D_w=2$, $\overline{u}=\frac{3+\sqrt{5}}{2}\approx 2.62,\overline{w}=\frac{3-\sqrt{5}}{2}\approx 0.382$. Left: component $u$. Right: component $w$. We observe formation of spikes at the position of local maxima of the initial conditions.}
\label{fig:cos}
\end{minipage}\\ $\,$ \\ 
\begin{minipage}[H]{40em}
\begin{center}
\begin{tabular}{p{15em}p{15em}}
   \centering
   \makeatletter{}\begingroup
  \makeatletter
  \providecommand\color[2][]{    \GenericError{(gnuplot) \space\space\space\@spaces}{      Package color not loaded in conjunction with
      terminal option `colourtext'    }{See the gnuplot documentation for explanation.    }{Either use 'blacktext' in gnuplot or load the package
      color.sty in LaTeX.}    \renewcommand\color[2][]{}  }  \providecommand\includegraphics[2][]{    \GenericError{(gnuplot) \space\space\space\@spaces}{      Package graphicx or graphics not loaded    }{See the gnuplot documentation for explanation.    }{The gnuplot epslatex terminal needs graphicx.sty or graphics.sty.}    \renewcommand\includegraphics[2][]{}  }  \providecommand\rotatebox[2]{#2}  \@ifundefined{ifGPcolor}{    \newif\ifGPcolor
    \GPcolorfalse
  }{}  \@ifundefined{ifGPblacktext}{    \newif\ifGPblacktext
    \GPblacktexttrue
  }{}    \let\gplgaddtomacro\g@addto@macro
    \gdef\gplbacktext{}  \gdef\gplfronttext{}  \makeatother
  \ifGPblacktext
        \def\colorrgb#1{}    \def\colorgray#1{}  \else
        \ifGPcolor
      \def\colorrgb#1{\color[rgb]{#1}}      \def\colorgray#1{\color[gray]{#1}}      \expandafter\def\csname LTw\endcsname{\color{white}}      \expandafter\def\csname LTb\endcsname{\color{black}}      \expandafter\def\csname LTa\endcsname{\color{black}}      \expandafter\def\csname LT0\endcsname{\color[rgb]{1,0,0}}      \expandafter\def\csname LT1\endcsname{\color[rgb]{0,1,0}}      \expandafter\def\csname LT2\endcsname{\color[rgb]{0,0,1}}      \expandafter\def\csname LT3\endcsname{\color[rgb]{1,0,1}}      \expandafter\def\csname LT4\endcsname{\color[rgb]{0,1,1}}      \expandafter\def\csname LT5\endcsname{\color[rgb]{1,1,0}}      \expandafter\def\csname LT6\endcsname{\color[rgb]{0,0,0}}      \expandafter\def\csname LT7\endcsname{\color[rgb]{1,0.3,0}}      \expandafter\def\csname LT8\endcsname{\color[rgb]{0.5,0.5,0.5}}    \else
            \def\colorrgb#1{\color{black}}      \def\colorgray#1{\color[gray]{#1}}      \expandafter\def\csname LTw\endcsname{\color{white}}      \expandafter\def\csname LTb\endcsname{\color{black}}      \expandafter\def\csname LTa\endcsname{\color{black}}      \expandafter\def\csname LT0\endcsname{\color{black}}      \expandafter\def\csname LT1\endcsname{\color{black}}      \expandafter\def\csname LT2\endcsname{\color{black}}      \expandafter\def\csname LT3\endcsname{\color{black}}      \expandafter\def\csname LT4\endcsname{\color{black}}      \expandafter\def\csname LT5\endcsname{\color{black}}      \expandafter\def\csname LT6\endcsname{\color{black}}      \expandafter\def\csname LT7\endcsname{\color{black}}      \expandafter\def\csname LT8\endcsname{\color{black}}    \fi
  \fi
  \setlength{\unitlength}{0.0500bp}  \begin{picture}(3600.00,3528.00)    \gplgaddtomacro\gplbacktext{    }    \gplgaddtomacro\gplfronttext{      \csname LTb\endcsname      \put(2135,2804){\makebox(0,0)[r]{\strut{}shape of $u_0$}}      \csname LTb\endcsname      \put(507,1092){\makebox(0,0){\strut{} 0}}      \put(1697,557){\makebox(0,0){\strut{} 25}}      \put(908,819){\makebox(0,0){\strut{}$t$}}      \put(1893,615){\makebox(0,0){\strut{} 0}}      \put(3083,1150){\makebox(0,0){\strut{} 1}}      \put(2692,819){\makebox(0,0){\strut{}$x$}}      \put(484,1656){\makebox(0,0)[r]{\strut{} 0}}      \put(484,2004){\makebox(0,0)[r]{\strut{} 100}}      \put(484,2354){\makebox(0,0)[r]{\strut{} 200}}    }    \gplbacktext
    \put(0,0){\includegraphics{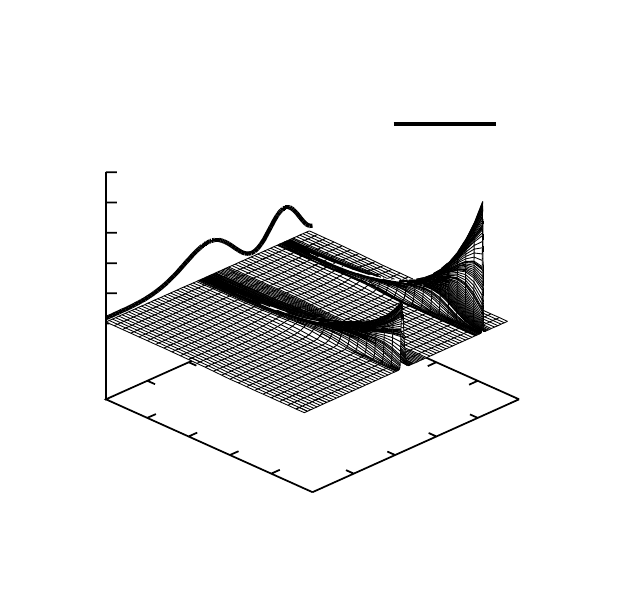}}    \gplfronttext
  \end{picture}\endgroup
 
&
   \centering
   \makeatletter{}\begingroup
  \makeatletter
  \providecommand\color[2][]{    \GenericError{(gnuplot) \space\space\space\@spaces}{      Package color not loaded in conjunction with
      terminal option `colourtext'    }{See the gnuplot documentation for explanation.    }{Either use 'blacktext' in gnuplot or load the package
      color.sty in LaTeX.}    \renewcommand\color[2][]{}  }  \providecommand\includegraphics[2][]{    \GenericError{(gnuplot) \space\space\space\@spaces}{      Package graphicx or graphics not loaded    }{See the gnuplot documentation for explanation.    }{The gnuplot epslatex terminal needs graphicx.sty or graphics.sty.}    \renewcommand\includegraphics[2][]{}  }  \providecommand\rotatebox[2]{#2}  \@ifundefined{ifGPcolor}{    \newif\ifGPcolor
    \GPcolorfalse
  }{}  \@ifundefined{ifGPblacktext}{    \newif\ifGPblacktext
    \GPblacktexttrue
  }{}    \let\gplgaddtomacro\g@addto@macro
    \gdef\gplbacktext{}  \gdef\gplfronttext{}  \makeatother
  \ifGPblacktext
        \def\colorrgb#1{}    \def\colorgray#1{}  \else
        \ifGPcolor
      \def\colorrgb#1{\color[rgb]{#1}}      \def\colorgray#1{\color[gray]{#1}}      \expandafter\def\csname LTw\endcsname{\color{white}}      \expandafter\def\csname LTb\endcsname{\color{black}}      \expandafter\def\csname LTa\endcsname{\color{black}}      \expandafter\def\csname LT0\endcsname{\color[rgb]{1,0,0}}      \expandafter\def\csname LT1\endcsname{\color[rgb]{0,1,0}}      \expandafter\def\csname LT2\endcsname{\color[rgb]{0,0,1}}      \expandafter\def\csname LT3\endcsname{\color[rgb]{1,0,1}}      \expandafter\def\csname LT4\endcsname{\color[rgb]{0,1,1}}      \expandafter\def\csname LT5\endcsname{\color[rgb]{1,1,0}}      \expandafter\def\csname LT6\endcsname{\color[rgb]{0,0,0}}      \expandafter\def\csname LT7\endcsname{\color[rgb]{1,0.3,0}}      \expandafter\def\csname LT8\endcsname{\color[rgb]{0.5,0.5,0.5}}    \else
            \def\colorrgb#1{\color{black}}      \def\colorgray#1{\color[gray]{#1}}      \expandafter\def\csname LTw\endcsname{\color{white}}      \expandafter\def\csname LTb\endcsname{\color{black}}      \expandafter\def\csname LTa\endcsname{\color{black}}      \expandafter\def\csname LT0\endcsname{\color{black}}      \expandafter\def\csname LT1\endcsname{\color{black}}      \expandafter\def\csname LT2\endcsname{\color{black}}      \expandafter\def\csname LT3\endcsname{\color{black}}      \expandafter\def\csname LT4\endcsname{\color{black}}      \expandafter\def\csname LT5\endcsname{\color{black}}      \expandafter\def\csname LT6\endcsname{\color{black}}      \expandafter\def\csname LT7\endcsname{\color{black}}      \expandafter\def\csname LT8\endcsname{\color{black}}    \fi
  \fi
  \setlength{\unitlength}{0.0500bp}  \begin{picture}(3600.00,3528.00)    \gplgaddtomacro\gplbacktext{    }    \gplgaddtomacro\gplfronttext{      \csname LTb\endcsname      \put(507,1092){\makebox(0,0){\strut{} 0}}      \put(1697,557){\makebox(0,0){\strut{} 25}}      \put(908,819){\makebox(0,0){\strut{}$t$}}      \put(1893,615){\makebox(0,0){\strut{} 0}}      \put(3083,1150){\makebox(0,0){\strut{} 1}}      \put(2692,819){\makebox(0,0){\strut{}$x$}}      \put(484,1656){\makebox(0,0)[r]{\strut{} 0}}      \put(484,1850){\makebox(0,0)[r]{\strut{} 0.1}}      \put(484,2043){\makebox(0,0)[r]{\strut{} 0.2}}      \put(484,2237){\makebox(0,0)[r]{\strut{} 0.3}}      \put(484,2431){\makebox(0,0)[r]{\strut{} 0.4}}    }    \gplbacktext
    \put(0,0){\includegraphics{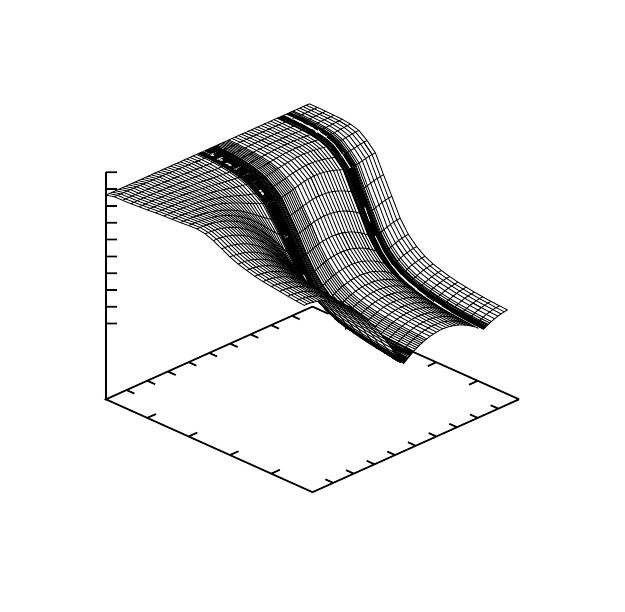}}    \gplfronttext
  \end{picture}\endgroup
 
\end{tabular}
\end{center}
\captionof{figure}{Numerical solution for initial conditions \eqref{IC:cosxx} with $\epsilon=0.05$ and diffusion coefficient $D_w=2$, $\overline{u}=\frac{3+\sqrt{5}}{2}\approx 2.62,\overline{w}=\frac{3-\sqrt{5}}{2}\approx 0.382$. Left: component $u$. Right: component $w$. We observe formation of spikes at the position of local maxima of the initial conditions and faster growth at $x=\frac{\sqrt{3}}{2}$ than at $x=\frac{1}{2}$.}
\label{fig:cosxx}
\end{minipage}\\ $\,$ \\ 
We note that the initial perturbation seems to be indeed self amplifying. This does not explain the long-time behavior, but numerical simulations indicate that spikes grow close to the maxima of the initial conditions. We also note for initial conditions \eqref{IC:cosxx} that the spike for larger $x$ grows faster. Real and imaginary part of the numerically obtained Finite Fourier Transform $\hat{f}(\omega):=\int_0^1 \cos(4 \pi x^2)e^{i \pi \omega x}dx$ and the growth rate of the perturbation at the maxima of $u_0$ are shown in Figure \ref{fig:ff-cosxx}, Figure \ref{fig:ff-cosx} shows the growth rate of perturbation \eqref{IC:cos}.
\subsection{Varying the diffusion coefficient}
The roots of the dispersion relation have the form
\begin{equation}
 \lambda_{\pm}(k^2)=\frac{\operatorname{tr}(A)-(\pi k)^2 D_w}{2} \pm \sqrt{(\frac{\operatorname{tr}(A)-(\pi k)^2 D_w}{2})^2 - |A|+(\pi k)^2 D_w a_{11}}
\end{equation}
We know that $\lambda_-(k^2) \rightarrow -\infty$ and $\lambda_+(k^2) \rightarrow a_{11}>0$ as $k \rightarrow \infty$, see Lemma \ref{lem:LA}, \ref{lem:EV}. Additionally, it holds $\lambda_+(0)<0$ since $(\overline{u}_-,\overline{w}_-)$ is a stable steady state of the kinetic system of \eqref{eq1-qs}-\eqref{eq2-qs}.\\
It follows that there exist stable eigenmodes of the Laplace Operator, because $\lambda_-(k^2)<0$ and
\begin{equation}
                \lambda_+(k^2)<0 \Leftrightarrow k^2<-\frac{|A|}{a_{11}\pi^2 D_w}.
\end{equation}
This implies dampening of the low frequency part of the initial perturbation $(\phi,\psi)$. \\
First, we choose the same initial conditions and parameters as in figure \ref{fig:1s}, but vary the diffusion coefficient $D_w$.\\
\begin{minipage}[H]{40em}
\begin{center}
\begin{tabular}{p{15em}p{15em}}
   \centering
   \makeatletter{}\begingroup
  \makeatletter
  \providecommand\color[2][]{    \GenericError{(gnuplot) \space\space\space\@spaces}{      Package color not loaded in conjunction with
      terminal option `colourtext'    }{See the gnuplot documentation for explanation.    }{Either use 'blacktext' in gnuplot or load the package
      color.sty in LaTeX.}    \renewcommand\color[2][]{}  }  \providecommand\includegraphics[2][]{    \GenericError{(gnuplot) \space\space\space\@spaces}{      Package graphicx or graphics not loaded    }{See the gnuplot documentation for explanation.    }{The gnuplot epslatex terminal needs graphicx.sty or graphics.sty.}    \renewcommand\includegraphics[2][]{}  }  \providecommand\rotatebox[2]{#2}  \@ifundefined{ifGPcolor}{    \newif\ifGPcolor
    \GPcolorfalse
  }{}  \@ifundefined{ifGPblacktext}{    \newif\ifGPblacktext
    \GPblacktexttrue
  }{}    \let\gplgaddtomacro\g@addto@macro
    \gdef\gplbacktext{}  \gdef\gplfronttext{}  \makeatother
  \ifGPblacktext
        \def\colorrgb#1{}    \def\colorgray#1{}  \else
        \ifGPcolor
      \def\colorrgb#1{\color[rgb]{#1}}      \def\colorgray#1{\color[gray]{#1}}      \expandafter\def\csname LTw\endcsname{\color{white}}      \expandafter\def\csname LTb\endcsname{\color{black}}      \expandafter\def\csname LTa\endcsname{\color{black}}      \expandafter\def\csname LT0\endcsname{\color[rgb]{1,0,0}}      \expandafter\def\csname LT1\endcsname{\color[rgb]{0,1,0}}      \expandafter\def\csname LT2\endcsname{\color[rgb]{0,0,1}}      \expandafter\def\csname LT3\endcsname{\color[rgb]{1,0,1}}      \expandafter\def\csname LT4\endcsname{\color[rgb]{0,1,1}}      \expandafter\def\csname LT5\endcsname{\color[rgb]{1,1,0}}      \expandafter\def\csname LT6\endcsname{\color[rgb]{0,0,0}}      \expandafter\def\csname LT7\endcsname{\color[rgb]{1,0.3,0}}      \expandafter\def\csname LT8\endcsname{\color[rgb]{0.5,0.5,0.5}}    \else
            \def\colorrgb#1{\color{black}}      \def\colorgray#1{\color[gray]{#1}}      \expandafter\def\csname LTw\endcsname{\color{white}}      \expandafter\def\csname LTb\endcsname{\color{black}}      \expandafter\def\csname LTa\endcsname{\color{black}}      \expandafter\def\csname LT0\endcsname{\color{black}}      \expandafter\def\csname LT1\endcsname{\color{black}}      \expandafter\def\csname LT2\endcsname{\color{black}}      \expandafter\def\csname LT3\endcsname{\color{black}}      \expandafter\def\csname LT4\endcsname{\color{black}}      \expandafter\def\csname LT5\endcsname{\color{black}}      \expandafter\def\csname LT6\endcsname{\color{black}}      \expandafter\def\csname LT7\endcsname{\color{black}}      \expandafter\def\csname LT8\endcsname{\color{black}}    \fi
  \fi
  \setlength{\unitlength}{0.0500bp}  \begin{picture}(3600.00,3528.00)    \gplgaddtomacro\gplbacktext{    }    \gplgaddtomacro\gplfronttext{      \csname LTb\endcsname      \put(2135,2804){\makebox(0,0)[r]{\strut{}shape of $u_0$}}      \csname LTb\endcsname      \put(507,1092){\makebox(0,0){\strut{} 0}}      \put(1697,557){\makebox(0,0){\strut{} 25}}      \put(908,819){\makebox(0,0){\strut{}$t$}}      \put(1893,615){\makebox(0,0){\strut{} 0}}      \put(3083,1150){\makebox(0,0){\strut{} 1}}      \put(2692,819){\makebox(0,0){\strut{}$x$}}      \put(484,1656){\makebox(0,0)[r]{\strut{} 0}}      \put(484,1946){\makebox(0,0)[r]{\strut{} 40}}      \put(484,2237){\makebox(0,0)[r]{\strut{} 80}}      \put(484,2528){\makebox(0,0)[r]{\strut{} 120}}    }    \gplbacktext
    \put(0,0){\includegraphics{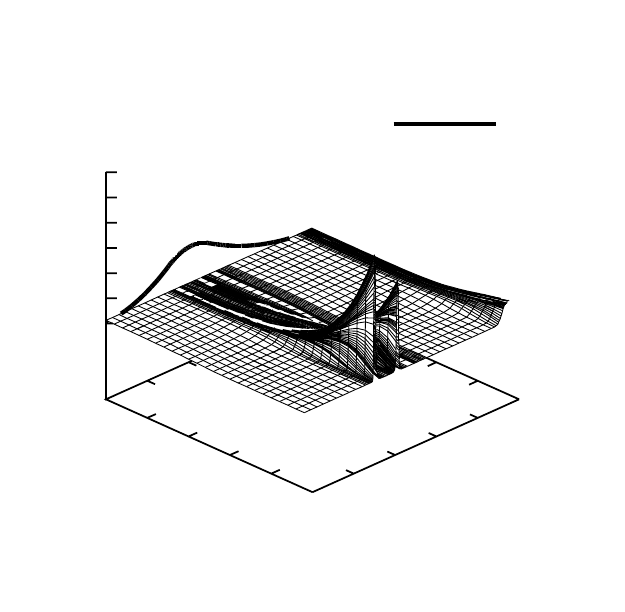}}    \gplfronttext
  \end{picture}\endgroup
 
&
   \centering
   \makeatletter{}\begingroup
  \makeatletter
  \providecommand\color[2][]{    \GenericError{(gnuplot) \space\space\space\@spaces}{      Package color not loaded in conjunction with
      terminal option `colourtext'    }{See the gnuplot documentation for explanation.    }{Either use 'blacktext' in gnuplot or load the package
      color.sty in LaTeX.}    \renewcommand\color[2][]{}  }  \providecommand\includegraphics[2][]{    \GenericError{(gnuplot) \space\space\space\@spaces}{      Package graphicx or graphics not loaded    }{See the gnuplot documentation for explanation.    }{The gnuplot epslatex terminal needs graphicx.sty or graphics.sty.}    \renewcommand\includegraphics[2][]{}  }  \providecommand\rotatebox[2]{#2}  \@ifundefined{ifGPcolor}{    \newif\ifGPcolor
    \GPcolorfalse
  }{}  \@ifundefined{ifGPblacktext}{    \newif\ifGPblacktext
    \GPblacktexttrue
  }{}    \let\gplgaddtomacro\g@addto@macro
    \gdef\gplbacktext{}  \gdef\gplfronttext{}  \makeatother
  \ifGPblacktext
        \def\colorrgb#1{}    \def\colorgray#1{}  \else
        \ifGPcolor
      \def\colorrgb#1{\color[rgb]{#1}}      \def\colorgray#1{\color[gray]{#1}}      \expandafter\def\csname LTw\endcsname{\color{white}}      \expandafter\def\csname LTb\endcsname{\color{black}}      \expandafter\def\csname LTa\endcsname{\color{black}}      \expandafter\def\csname LT0\endcsname{\color[rgb]{1,0,0}}      \expandafter\def\csname LT1\endcsname{\color[rgb]{0,1,0}}      \expandafter\def\csname LT2\endcsname{\color[rgb]{0,0,1}}      \expandafter\def\csname LT3\endcsname{\color[rgb]{1,0,1}}      \expandafter\def\csname LT4\endcsname{\color[rgb]{0,1,1}}      \expandafter\def\csname LT5\endcsname{\color[rgb]{1,1,0}}      \expandafter\def\csname LT6\endcsname{\color[rgb]{0,0,0}}      \expandafter\def\csname LT7\endcsname{\color[rgb]{1,0.3,0}}      \expandafter\def\csname LT8\endcsname{\color[rgb]{0.5,0.5,0.5}}    \else
            \def\colorrgb#1{\color{black}}      \def\colorgray#1{\color[gray]{#1}}      \expandafter\def\csname LTw\endcsname{\color{white}}      \expandafter\def\csname LTb\endcsname{\color{black}}      \expandafter\def\csname LTa\endcsname{\color{black}}      \expandafter\def\csname LT0\endcsname{\color{black}}      \expandafter\def\csname LT1\endcsname{\color{black}}      \expandafter\def\csname LT2\endcsname{\color{black}}      \expandafter\def\csname LT3\endcsname{\color{black}}      \expandafter\def\csname LT4\endcsname{\color{black}}      \expandafter\def\csname LT5\endcsname{\color{black}}      \expandafter\def\csname LT6\endcsname{\color{black}}      \expandafter\def\csname LT7\endcsname{\color{black}}      \expandafter\def\csname LT8\endcsname{\color{black}}    \fi
  \fi
  \setlength{\unitlength}{0.0500bp}  \begin{picture}(3600.00,3528.00)    \gplgaddtomacro\gplbacktext{    }    \gplgaddtomacro\gplfronttext{      \csname LTb\endcsname      \put(507,1092){\makebox(0,0){\strut{} 0}}      \put(1697,557){\makebox(0,0){\strut{} 25}}      \put(908,819){\makebox(0,0){\strut{}$t$}}      \put(1893,615){\makebox(0,0){\strut{} 0}}      \put(3083,1150){\makebox(0,0){\strut{} 1}}      \put(2692,819){\makebox(0,0){\strut{}$x$}}      \put(484,1656){\makebox(0,0)[r]{\strut{} 0}}      \put(484,1831){\makebox(0,0)[r]{\strut{} 0.1}}      \put(484,2004){\makebox(0,0)[r]{\strut{} 0.2}}      \put(484,2179){\makebox(0,0)[r]{\strut{} 0.3}}      \put(484,2354){\makebox(0,0)[r]{\strut{} 0.4}}      \put(484,2528){\makebox(0,0)[r]{\strut{} 0.5}}    }    \gplbacktext
    \put(0,0){\includegraphics{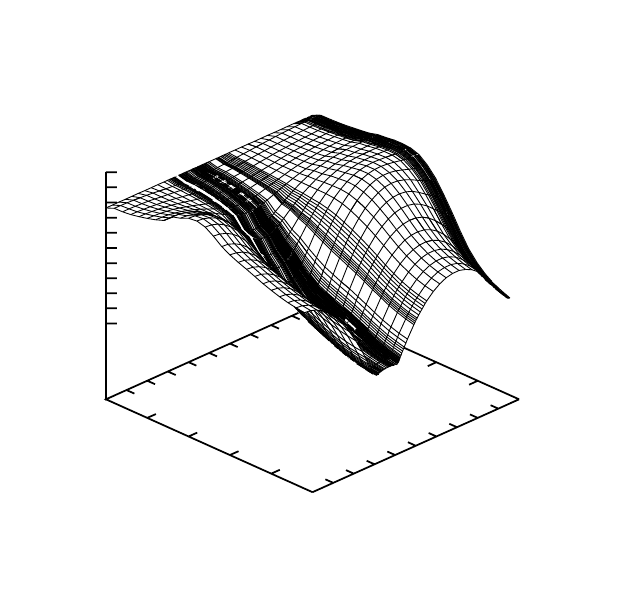}}    \gplfronttext
  \end{picture}\endgroup
 
\end{tabular}
\end{center}
\captionof{figure}{Numerical solution for initial conditions \eqref{IC:sh} with $s=0.4,\epsilon=0.1,\epsilon_2=0.05$ and diffusion coefficient $D_w=1$, $\overline{u}=\frac{3+\sqrt{5}}{2}\approx 2.62,\overline{w}=\frac{3-\sqrt{5}}{2}\approx 0.382$. Left: component $u$. Right: component $w$.}
\label{behaviour-multispike-1}
\end{minipage}\\ $\,$ \\ 
For smaller diffusion coefficient, $D_w=1$, we observe growth of multiple spikes for the same initial conditions, see Figure \ref{behaviour-multispike-1}. We observe a self-amplification of the high-frequency part of the initial perturbation. The short-time behavior is therefore similar to the idea of a 'dominant' eigenvalue in classical Turing type models. However, in this case, we can speak of a 'self amplification of the part of the initial perturbation with sufficiently high wavenumber'.\\
We define,
\begin{equation}
D_{w,k}:=\frac{|A|}{a_{11}(\pi k)^2}=\frac{1}{(\pi k)^2}\frac{-4d_1^2+(a_1-d_1)^2 \kappa^2 + \kappa (a_1-d_1) \sqrt{\kappa_1^2 (a_1-d_1)^2-4 d_1^2}}{2d_1^2}.
\end{equation}
Figure \ref{fig:Dw,1} shows the solution for the corresponding $D_{w,1}$ for model \eqref{eq1-qs}-\eqref{eq2-qs}. Table \ref{tab:D_Spi} shows the number of spikes for further variation of $D_w$ for initial conditions \eqref{IC:sh} with $s=0.4, \epsilon=0.1$, $\epsilon_1=0.05$.\\
\begin{minipage}[H]{40em}
\begin{center}
\begin{tabular}{p{15em}p{15em}}
   \centering
   \makeatletter{}\begingroup
  \makeatletter
  \providecommand\color[2][]{    \GenericError{(gnuplot) \space\space\space\@spaces}{      Package color not loaded in conjunction with
      terminal option `colourtext'    }{See the gnuplot documentation for explanation.    }{Either use 'blacktext' in gnuplot or load the package
      color.sty in LaTeX.}    \renewcommand\color[2][]{}  }  \providecommand\includegraphics[2][]{    \GenericError{(gnuplot) \space\space\space\@spaces}{      Package graphicx or graphics not loaded    }{See the gnuplot documentation for explanation.    }{The gnuplot epslatex terminal needs graphicx.sty or graphics.sty.}    \renewcommand\includegraphics[2][]{}  }  \providecommand\rotatebox[2]{#2}  \@ifundefined{ifGPcolor}{    \newif\ifGPcolor
    \GPcolorfalse
  }{}  \@ifundefined{ifGPblacktext}{    \newif\ifGPblacktext
    \GPblacktexttrue
  }{}    \let\gplgaddtomacro\g@addto@macro
    \gdef\gplbacktext{}  \gdef\gplfronttext{}  \makeatother
  \ifGPblacktext
        \def\colorrgb#1{}    \def\colorgray#1{}  \else
        \ifGPcolor
      \def\colorrgb#1{\color[rgb]{#1}}      \def\colorgray#1{\color[gray]{#1}}      \expandafter\def\csname LTw\endcsname{\color{white}}      \expandafter\def\csname LTb\endcsname{\color{black}}      \expandafter\def\csname LTa\endcsname{\color{black}}      \expandafter\def\csname LT0\endcsname{\color[rgb]{1,0,0}}      \expandafter\def\csname LT1\endcsname{\color[rgb]{0,1,0}}      \expandafter\def\csname LT2\endcsname{\color[rgb]{0,0,1}}      \expandafter\def\csname LT3\endcsname{\color[rgb]{1,0,1}}      \expandafter\def\csname LT4\endcsname{\color[rgb]{0,1,1}}      \expandafter\def\csname LT5\endcsname{\color[rgb]{1,1,0}}      \expandafter\def\csname LT6\endcsname{\color[rgb]{0,0,0}}      \expandafter\def\csname LT7\endcsname{\color[rgb]{1,0.3,0}}      \expandafter\def\csname LT8\endcsname{\color[rgb]{0.5,0.5,0.5}}    \else
            \def\colorrgb#1{\color{black}}      \def\colorgray#1{\color[gray]{#1}}      \expandafter\def\csname LTw\endcsname{\color{white}}      \expandafter\def\csname LTb\endcsname{\color{black}}      \expandafter\def\csname LTa\endcsname{\color{black}}      \expandafter\def\csname LT0\endcsname{\color{black}}      \expandafter\def\csname LT1\endcsname{\color{black}}      \expandafter\def\csname LT2\endcsname{\color{black}}      \expandafter\def\csname LT3\endcsname{\color{black}}      \expandafter\def\csname LT4\endcsname{\color{black}}      \expandafter\def\csname LT5\endcsname{\color{black}}      \expandafter\def\csname LT6\endcsname{\color{black}}      \expandafter\def\csname LT7\endcsname{\color{black}}      \expandafter\def\csname LT8\endcsname{\color{black}}    \fi
  \fi
  \setlength{\unitlength}{0.0500bp}  \begin{picture}(3600.00,3528.00)    \gplgaddtomacro\gplbacktext{    }    \gplgaddtomacro\gplfronttext{      \csname LTb\endcsname      \put(2135,2804){\makebox(0,0)[r]{\strut{}shape of $u_0$}}      \csname LTb\endcsname      \put(507,1092){\makebox(0,0){\strut{} 0}}      \put(1697,557){\makebox(0,0){\strut{} 25}}      \put(908,819){\makebox(0,0){\strut{}$t$}}      \put(1893,615){\makebox(0,0){\strut{} 0}}      \put(3083,1150){\makebox(0,0){\strut{} 1}}      \put(2692,819){\makebox(0,0){\strut{}$x$}}      \put(484,1656){\makebox(0,0)[r]{\strut{} 0}}      \put(484,2004){\makebox(0,0)[r]{\strut{} 100}}      \put(484,2354){\makebox(0,0)[r]{\strut{} 200}}    }    \gplbacktext
    \put(0,0){\includegraphics{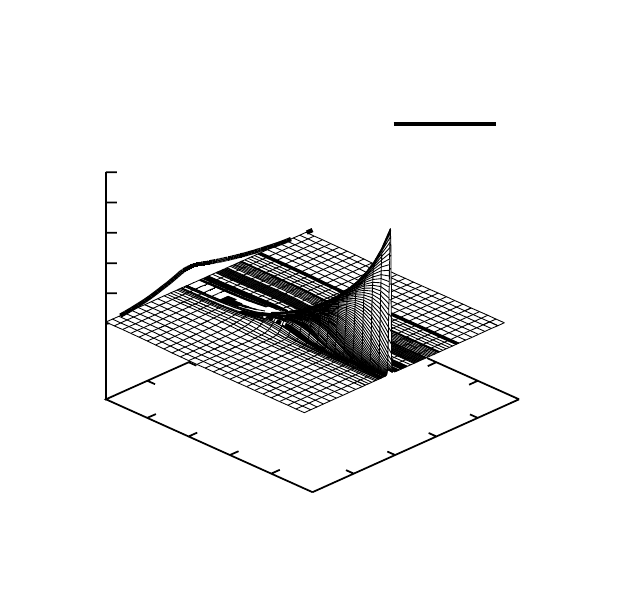}}    \gplfronttext
  \end{picture}\endgroup
 
&
   \centering
   \makeatletter{}\begingroup
  \makeatletter
  \providecommand\color[2][]{    \GenericError{(gnuplot) \space\space\space\@spaces}{      Package color not loaded in conjunction with
      terminal option `colourtext'    }{See the gnuplot documentation for explanation.    }{Either use 'blacktext' in gnuplot or load the package
      color.sty in LaTeX.}    \renewcommand\color[2][]{}  }  \providecommand\includegraphics[2][]{    \GenericError{(gnuplot) \space\space\space\@spaces}{      Package graphicx or graphics not loaded    }{See the gnuplot documentation for explanation.    }{The gnuplot epslatex terminal needs graphicx.sty or graphics.sty.}    \renewcommand\includegraphics[2][]{}  }  \providecommand\rotatebox[2]{#2}  \@ifundefined{ifGPcolor}{    \newif\ifGPcolor
    \GPcolorfalse
  }{}  \@ifundefined{ifGPblacktext}{    \newif\ifGPblacktext
    \GPblacktexttrue
  }{}    \let\gplgaddtomacro\g@addto@macro
    \gdef\gplbacktext{}  \gdef\gplfronttext{}  \makeatother
  \ifGPblacktext
        \def\colorrgb#1{}    \def\colorgray#1{}  \else
        \ifGPcolor
      \def\colorrgb#1{\color[rgb]{#1}}      \def\colorgray#1{\color[gray]{#1}}      \expandafter\def\csname LTw\endcsname{\color{white}}      \expandafter\def\csname LTb\endcsname{\color{black}}      \expandafter\def\csname LTa\endcsname{\color{black}}      \expandafter\def\csname LT0\endcsname{\color[rgb]{1,0,0}}      \expandafter\def\csname LT1\endcsname{\color[rgb]{0,1,0}}      \expandafter\def\csname LT2\endcsname{\color[rgb]{0,0,1}}      \expandafter\def\csname LT3\endcsname{\color[rgb]{1,0,1}}      \expandafter\def\csname LT4\endcsname{\color[rgb]{0,1,1}}      \expandafter\def\csname LT5\endcsname{\color[rgb]{1,1,0}}      \expandafter\def\csname LT6\endcsname{\color[rgb]{0,0,0}}      \expandafter\def\csname LT7\endcsname{\color[rgb]{1,0.3,0}}      \expandafter\def\csname LT8\endcsname{\color[rgb]{0.5,0.5,0.5}}    \else
            \def\colorrgb#1{\color{black}}      \def\colorgray#1{\color[gray]{#1}}      \expandafter\def\csname LTw\endcsname{\color{white}}      \expandafter\def\csname LTb\endcsname{\color{black}}      \expandafter\def\csname LTa\endcsname{\color{black}}      \expandafter\def\csname LT0\endcsname{\color{black}}      \expandafter\def\csname LT1\endcsname{\color{black}}      \expandafter\def\csname LT2\endcsname{\color{black}}      \expandafter\def\csname LT3\endcsname{\color{black}}      \expandafter\def\csname LT4\endcsname{\color{black}}      \expandafter\def\csname LT5\endcsname{\color{black}}      \expandafter\def\csname LT6\endcsname{\color{black}}      \expandafter\def\csname LT7\endcsname{\color{black}}      \expandafter\def\csname LT8\endcsname{\color{black}}    \fi
  \fi
  \setlength{\unitlength}{0.0500bp}  \begin{picture}(3600.00,3528.00)    \gplgaddtomacro\gplbacktext{    }    \gplgaddtomacro\gplfronttext{      \csname LTb\endcsname      \put(507,1092){\makebox(0,0){\strut{} 0}}      \put(1697,557){\makebox(0,0){\strut{} 25}}      \put(908,819){\makebox(0,0){\strut{}$t$}}      \put(1893,615){\makebox(0,0){\strut{} 0}}      \put(3083,1150){\makebox(0,0){\strut{} 1}}      \put(2692,819){\makebox(0,0){\strut{}$x$}}      \put(484,1656){\makebox(0,0)[r]{\strut{} 0}}      \put(484,1874){\makebox(0,0)[r]{\strut{} 0.1}}      \put(484,2092){\makebox(0,0)[r]{\strut{} 0.2}}      \put(484,2310){\makebox(0,0)[r]{\strut{} 0.3}}      \put(484,2528){\makebox(0,0)[r]{\strut{} 0.4}}    }    \gplbacktext
    \put(0,0){\includegraphics{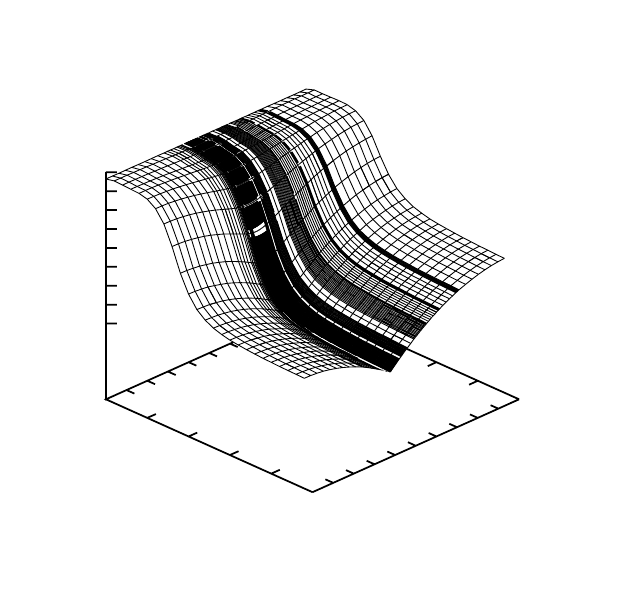}}    \gplfronttext
  \end{picture}\endgroup
 
\end{tabular}
\end{center}
\captionof{figure}{Numerical solution for initial conditions \eqref{IC:sh} with $s=0.4,\epsilon=0.1, \epsilon_1=0.05$ and diffusion coefficient $D_w=5.8541\approx D_{w,1}$, $\overline{u}=\frac{3+\sqrt{5}}{2}\approx 2.62,\overline{w}=\frac{3-\sqrt{5}}{2}\approx 0.382$. Left: component $u$. Right: component $w$.}
\label{fig:Dw,1}
\end{minipage}\\ $\,$ \\ 
\subsection{Evolution of mass}
Our simulations indicate a growth of one or multiple spikes, $u(x) \rightarrow \infty$ for some $x$ as $t \rightarrow \infty$, and decay in all other $x$. We therefore investigate the evolution of the $L^1$-norm of $u$. Figure \ref{fig:m-1s} shows the evolution of mass of the solution shown in figure \ref{fig:1s} for homogeneous spatial mesh size $h=2^{-16}$ and homogeneous temporal mesh size $k=2.5\cdot 10^{-4}$. The convergence order is shown in the appendix, see Figure \ref{fig:L2-sisp}-\ref{fig:L1-sisp}.
Lemma \ref{thm:mass} states that the the mass of the solution, $\left\|u(t)\right\|_{L^1}$ is uniformly bounded. However, an important question when modeling natural phenomena is positivity of mass if there is no extinction. The numerical simulations of the evolution of the mass suggests that it stays strictly positive. Therefore, based on the numerical simulations we have a conjecture that the solutions converge asymptotically to the sum of Diracs. This hypothesis supported by numerical simulations needs however a proof.\\

\begin{minipage}[H]{40em}
\begin{center}
\begin{tabular}{p{15em}p{15em}}
   \centering
   \makeatletter{}\begingroup
  \makeatletter
  \providecommand\color[2][]{    \GenericError{(gnuplot) \space\space\space\@spaces}{      Package color not loaded in conjunction with
      terminal option `colourtext'    }{See the gnuplot documentation for explanation.    }{Either use 'blacktext' in gnuplot or load the package
      color.sty in LaTeX.}    \renewcommand\color[2][]{}  }  \providecommand\includegraphics[2][]{    \GenericError{(gnuplot) \space\space\space\@spaces}{      Package graphicx or graphics not loaded    }{See the gnuplot documentation for explanation.    }{The gnuplot epslatex terminal needs graphicx.sty or graphics.sty.}    \renewcommand\includegraphics[2][]{}  }  \providecommand\rotatebox[2]{#2}  \@ifundefined{ifGPcolor}{    \newif\ifGPcolor
    \GPcolorfalse
  }{}  \@ifundefined{ifGPblacktext}{    \newif\ifGPblacktext
    \GPblacktexttrue
  }{}    \let\gplgaddtomacro\g@addto@macro
    \gdef\gplbacktext{}  \gdef\gplfronttext{}  \makeatother
  \ifGPblacktext
        \def\colorrgb#1{}    \def\colorgray#1{}  \else
        \ifGPcolor
      \def\colorrgb#1{\color[rgb]{#1}}      \def\colorgray#1{\color[gray]{#1}}      \expandafter\def\csname LTw\endcsname{\color{white}}      \expandafter\def\csname LTb\endcsname{\color{black}}      \expandafter\def\csname LTa\endcsname{\color{black}}      \expandafter\def\csname LT0\endcsname{\color[rgb]{1,0,0}}      \expandafter\def\csname LT1\endcsname{\color[rgb]{0,1,0}}      \expandafter\def\csname LT2\endcsname{\color[rgb]{0,0,1}}      \expandafter\def\csname LT3\endcsname{\color[rgb]{1,0,1}}      \expandafter\def\csname LT4\endcsname{\color[rgb]{0,1,1}}      \expandafter\def\csname LT5\endcsname{\color[rgb]{1,1,0}}      \expandafter\def\csname LT6\endcsname{\color[rgb]{0,0,0}}      \expandafter\def\csname LT7\endcsname{\color[rgb]{1,0.3,0}}      \expandafter\def\csname LT8\endcsname{\color[rgb]{0.5,0.5,0.5}}    \else
            \def\colorrgb#1{\color{black}}      \def\colorgray#1{\color[gray]{#1}}      \expandafter\def\csname LTw\endcsname{\color{white}}      \expandafter\def\csname LTb\endcsname{\color{black}}      \expandafter\def\csname LTa\endcsname{\color{black}}      \expandafter\def\csname LT0\endcsname{\color{black}}      \expandafter\def\csname LT1\endcsname{\color{black}}      \expandafter\def\csname LT2\endcsname{\color{black}}      \expandafter\def\csname LT3\endcsname{\color{black}}      \expandafter\def\csname LT4\endcsname{\color{black}}      \expandafter\def\csname LT5\endcsname{\color{black}}      \expandafter\def\csname LT6\endcsname{\color{black}}      \expandafter\def\csname LT7\endcsname{\color{black}}      \expandafter\def\csname LT8\endcsname{\color{black}}    \fi
  \fi
  \setlength{\unitlength}{0.0500bp}  \begin{picture}(3600.00,2520.00)    \gplgaddtomacro\gplbacktext{      \csname LTb\endcsname      \put(858,704){\makebox(0,0)[r]{\strut{} 2.35}}      \csname LTb\endcsname      \put(858,898){\makebox(0,0)[r]{\strut{} 2.4}}      \csname LTb\endcsname      \put(858,1092){\makebox(0,0)[r]{\strut{} 2.45}}      \csname LTb\endcsname      \put(858,1286){\makebox(0,0)[r]{\strut{} 2.5}}      \csname LTb\endcsname      \put(858,1480){\makebox(0,0)[r]{\strut{} 2.55}}      \csname LTb\endcsname      \put(858,1674){\makebox(0,0)[r]{\strut{} 2.6}}      \csname LTb\endcsname      \put(858,1868){\makebox(0,0)[r]{\strut{} 2.65}}      \csname LTb\endcsname      \put(858,2062){\makebox(0,0)[r]{\strut{} 2.7}}      \csname LTb\endcsname      \put(858,2256){\makebox(0,0)[r]{\strut{} 2.75}}      \csname LTb\endcsname      \put(990,484){\makebox(0,0){\strut{} 0}}      \csname LTb\endcsname      \put(1359,484){\makebox(0,0){\strut{} 5}}      \csname LTb\endcsname      \put(1728,484){\makebox(0,0){\strut{} 10}}      \csname LTb\endcsname      \put(2097,484){\makebox(0,0){\strut{} 15}}      \csname LTb\endcsname      \put(2465,484){\makebox(0,0){\strut{} 20}}      \csname LTb\endcsname      \put(2834,484){\makebox(0,0){\strut{} 25}}      \csname LTb\endcsname      \put(3203,484){\makebox(0,0){\strut{} 30}}      \put(2096,154){\makebox(0,0){\strut{}$t$}}    }    \gplgaddtomacro\gplfronttext{      \csname LTb\endcsname      \put(2216,2083){\makebox(0,0)[r]{\strut{}$\left\|u(t)\right\|_{L^1}$}}    }    \gplbacktext
    \put(0,0){\includegraphics{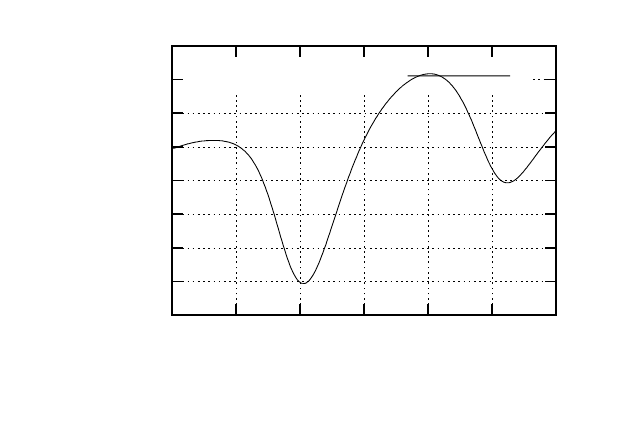}}    \gplfronttext
  \end{picture}\endgroup
 
&
   \centering
   \makeatletter{}\begingroup
  \makeatletter
  \providecommand\color[2][]{    \GenericError{(gnuplot) \space\space\space\@spaces}{      Package color not loaded in conjunction with
      terminal option `colourtext'    }{See the gnuplot documentation for explanation.    }{Either use 'blacktext' in gnuplot or load the package
      color.sty in LaTeX.}    \renewcommand\color[2][]{}  }  \providecommand\includegraphics[2][]{    \GenericError{(gnuplot) \space\space\space\@spaces}{      Package graphicx or graphics not loaded    }{See the gnuplot documentation for explanation.    }{The gnuplot epslatex terminal needs graphicx.sty or graphics.sty.}    \renewcommand\includegraphics[2][]{}  }  \providecommand\rotatebox[2]{#2}  \@ifundefined{ifGPcolor}{    \newif\ifGPcolor
    \GPcolorfalse
  }{}  \@ifundefined{ifGPblacktext}{    \newif\ifGPblacktext
    \GPblacktexttrue
  }{}    \let\gplgaddtomacro\g@addto@macro
    \gdef\gplbacktext{}  \gdef\gplfronttext{}  \makeatother
  \ifGPblacktext
        \def\colorrgb#1{}    \def\colorgray#1{}  \else
        \ifGPcolor
      \def\colorrgb#1{\color[rgb]{#1}}      \def\colorgray#1{\color[gray]{#1}}      \expandafter\def\csname LTw\endcsname{\color{white}}      \expandafter\def\csname LTb\endcsname{\color{black}}      \expandafter\def\csname LTa\endcsname{\color{black}}      \expandafter\def\csname LT0\endcsname{\color[rgb]{1,0,0}}      \expandafter\def\csname LT1\endcsname{\color[rgb]{0,1,0}}      \expandafter\def\csname LT2\endcsname{\color[rgb]{0,0,1}}      \expandafter\def\csname LT3\endcsname{\color[rgb]{1,0,1}}      \expandafter\def\csname LT4\endcsname{\color[rgb]{0,1,1}}      \expandafter\def\csname LT5\endcsname{\color[rgb]{1,1,0}}      \expandafter\def\csname LT6\endcsname{\color[rgb]{0,0,0}}      \expandafter\def\csname LT7\endcsname{\color[rgb]{1,0.3,0}}      \expandafter\def\csname LT8\endcsname{\color[rgb]{0.5,0.5,0.5}}    \else
            \def\colorrgb#1{\color{black}}      \def\colorgray#1{\color[gray]{#1}}      \expandafter\def\csname LTw\endcsname{\color{white}}      \expandafter\def\csname LTb\endcsname{\color{black}}      \expandafter\def\csname LTa\endcsname{\color{black}}      \expandafter\def\csname LT0\endcsname{\color{black}}      \expandafter\def\csname LT1\endcsname{\color{black}}      \expandafter\def\csname LT2\endcsname{\color{black}}      \expandafter\def\csname LT3\endcsname{\color{black}}      \expandafter\def\csname LT4\endcsname{\color{black}}      \expandafter\def\csname LT5\endcsname{\color{black}}      \expandafter\def\csname LT6\endcsname{\color{black}}      \expandafter\def\csname LT7\endcsname{\color{black}}      \expandafter\def\csname LT8\endcsname{\color{black}}    \fi
  \fi
  \setlength{\unitlength}{0.0500bp}  \begin{picture}(3600.00,2520.00)    \gplgaddtomacro\gplbacktext{      \csname LTb\endcsname      \put(858,704){\makebox(0,0)[r]{\strut{} 0.1}}      \csname LTb\endcsname      \put(858,963){\makebox(0,0)[r]{\strut{} 0.15}}      \csname LTb\endcsname      \put(858,1221){\makebox(0,0)[r]{\strut{} 0.2}}      \csname LTb\endcsname      \put(858,1480){\makebox(0,0)[r]{\strut{} 0.25}}      \csname LTb\endcsname      \put(858,1739){\makebox(0,0)[r]{\strut{} 0.3}}      \csname LTb\endcsname      \put(858,1997){\makebox(0,0)[r]{\strut{} 0.35}}      \csname LTb\endcsname      \put(858,2256){\makebox(0,0)[r]{\strut{} 0.4}}      \csname LTb\endcsname      \put(990,484){\makebox(0,0){\strut{} 0}}      \csname LTb\endcsname      \put(1433,484){\makebox(0,0){\strut{} 5}}      \csname LTb\endcsname      \put(1875,484){\makebox(0,0){\strut{} 10}}      \csname LTb\endcsname      \put(2318,484){\makebox(0,0){\strut{} 15}}      \csname LTb\endcsname      \put(2760,484){\makebox(0,0){\strut{} 20}}      \csname LTb\endcsname      \put(3203,484){\makebox(0,0){\strut{} 25}}      \put(2096,154){\makebox(0,0){\strut{}$t$}}    }    \gplgaddtomacro\gplfronttext{      \csname LTb\endcsname      \put(2216,2083){\makebox(0,0)[r]{\strut{}$\left\|w(t)\right\|_{L^1}$}}    }    \gplbacktext
    \put(0,0){\includegraphics{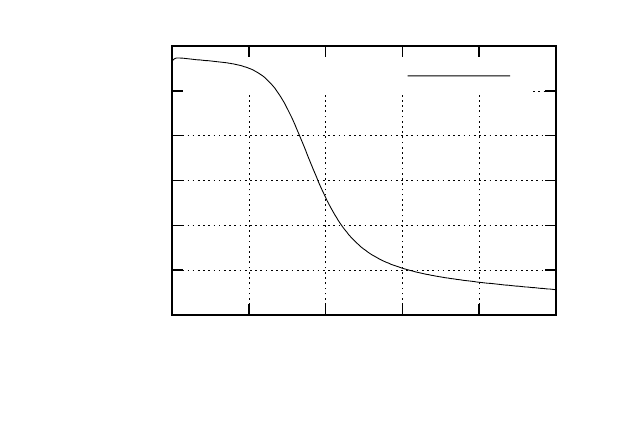}}    \gplfronttext
  \end{picture}\endgroup
 
\end{tabular}
\end{center}
\captionof{figure}{Evolution of the $L^1$-norm of the solution shown in figure \ref{fig:1s}. Left: component $u$. Right: component $w$.}
\label{fig:m-1s}
\end{minipage}\\ $\,$ \\ 

\section{Acknowledgments}

This work was supported by European Research Council Starting Grant No 210680 ``Multiscale mathematical modelling of dynamics of structure formation in cell systems'' and Emmy Noether Programme of German Research Council (DFG).  Steffen H\"arting was partly supported by Ev. Studienwerk Villigst e.V..

\section{Appendix}
\subsection{Derivation of the model}
Assume that component $v$ in model \eqref{eq1}-\eqref{eq3} satisfies the steady state equation
\begin{equation}
 0=\alpha u^2 w + dv - d_b v.
\end{equation}
Solving for $v$ yields
\begin{equation}
 v=\frac{\alpha}{d_b+d} u^2 w. \label{red_v}
\end{equation}
Substituting \eqref{red_v} into \eqref{eq1} and $\eqref{eq3}$ yields
\begin{align}
u_t&=\Big(a\frac{uw}{\sigma + uw} -d_c\Big) u,                             &\text{for}\ x\in [0,1], \; t>0, \label{eq1-qs-dipl}\\
w_t &= \frac{1}{\gamma} w_{xx} - d_g w - \sigma^{-1} d_b u^2 w +\kappa, \; &\text{for}\ x\in (0,1), \; t>0, \label{eq2-qs-dipl}
\end{align}
for $\sigma:=\frac{d_b+d}{\alpha}$.
After rescaling time, $\hat{t}:=d_g t$ yields
\begin{align}
u_{\hat{t}}&=\Big(\frac{a}{d_g}\frac{uw}{\sigma + uw} -\frac{d_c}{d_g}\Big) u,                            & \text{for}\ x\in [0,1], \; \hat{t}>0, \noindent\\
w_{\hat{t}}&= \frac{1}{\gamma d_g} w_{xx} - w - \sigma^{-1} \frac{d_b}{d_g} u^2 w +\frac{\kappa}{d_g}. \; & \text{for}\ x\in (0,1). \; \hat{t}>0 \noindent
\end{align}
Defining $\hat{u}(x,t):=\sqrt{\frac{d_b}{\sigma d_g}} u(x,t)$ and $\hat{w}(x,t):=\sqrt{\frac{d_g}{d_b \sigma}} w(x,t)$, we obtain system \eqref{eq1-qs}-\eqref{eq2-qs}:
\begin{align}
\hat{u}_{\hat{t}}&=\Big(\frac{a}{d_g}\frac{\sigma \hat{u}\hat{w}}{\sigma + \sigma \hat{u}\hat{w}} -\frac{d_c}{d_g}\Big) \hat{u}         &\text{for}\ x\in [0,1], \; \hat{t}>0, \noindent\\
\hat{w}_{\hat{t}}&= \frac{1}{\gamma d_g} \hat{w}_{xx} - \hat{w} - \hat{u}^2 \hat{w} +\frac{\kappa}{\sqrt{d_g d_b \sigma}}        \qquad &\text{for}\ x\in (0,1), \; \hat{t}>0. \noindent
\end{align}
\subsection{Proofs of analytical statements}
\begin{proof}[Proof of Theorem \ref{thm:C2}]
Since $v=\frac{\alpha}{d_b+d} u^2 w$ is the unique root of the right-hand side of \eqref{eq2},
there exists a one-to-one mapping from the set of steady states of \eqref{eq1-qs-dipl}-\eqref{eq2-qs-dipl} into the set of steady states of \eqref{eq1}-\eqref{eq3} by
$(u,w) \rightarrow (u,\frac{\alpha}{d_b+d}u^2 w, w)$.\\
Since model \eqref{eq1-qs}-\eqref{eq2-qs} is a linear rescaling resp. linear substitution of \eqref{eq1-qs-dipl}-\eqref{eq2-qs-dipl}, there exists also a one-to-one mapping between the sets of steady states.\\
\cite{MCKS12}, Theorem 2.6 proves Theorem \ref{thm:C2} for system \eqref{eq1}-\eqref{eq3}. Since we found a one-to-one mapping between the sets of steady states, statements (ii) and (iii) and existence of the steady states in (i) follow from \cite{MCKS12}, Theorem 2.6.\\
It is left to calculate the exact values of the spatially homogeneous steady states.\\
The right-hand-side of \eqref{eq1-qs} has two roots:
\begin{align}
\overline{u}_0 &= 0, \label{P:u=0} \\
\overline{u}_1 &=\frac{d_1}{a_1-d_1} \frac{1}{w}. \label{P:u!=0}
\end{align}
Substituting \ref{P:u=0} into the right-hand side of \eqref{eq2-qs} and setting it equal to zero leads to
\begin{equation}
0 = -w+\kappa_1,\\
\end{equation}
defining $(\overline{u}_0,\overline{w}_0)=(0,\kappa_1)$.\\
Substituting \ref{P:u!=0} into the right-hand side of \eqref{eq2-qs} and setting it equal to zero leads to
\begin{equation}
0 = -w-(\frac{d_1}{a_1-d_1})^2 \frac{1}{w}+\kappa_1,
\end{equation}
with roots $\overline{w}_-$ and $\overline{w}_+$. 
\end{proof}

To prove \ref{thm:DDI}, we use the following lemma from linear algebra, proved in \cite{Baker}, section 2.1.2:
\begin{lemma}\label{lem:LA}
Let a real-valued block-matrix 
\begin{equation}\label{LA-matrix}
A=
\begin{bmatrix}
 A_{11} & A_{12} \\
 A_{21} & A_{22}-D k^2
\end{bmatrix},
\end{equation}
be given with $D=\text{diag}(d_1,...,d_{m})$,$d_i>0$. \\
Let $\lambda_1,...,\lambda_n$ denote the eigenvalues of $A_{12}$ and $\hat{\lambda}_1,...,\hat{\lambda}_{m+n}$ the eigenvalues of $A$.\\
Then there exists an injective mapping $j:\{1,...,n\}\rightarrow \{1,...,n+m\}$, s.t. for all $1 \leq i \leq n$ holds
\begin{equation}
 \lim_{k \rightarrow \infty}\hat{\lambda}_{j(i)} = \lambda_i,
\end{equation}
and the real parts of all other eigenvalues of $A$ converge towards $-\infty$ as $k \rightarrow \infty$.
\end{lemma}
Lemma \ref{lem:LA}, applied to stability of spatially homogeneous steady states of ordinary differential equations coupled to reaction-diffusion equations reads:
\begin{lemma}\label{lem:EV}
Given a system of ordinary/partial-differential equations:
\begin{equation}\label{ODE_RD_sys_general}
\begin{aligned}
\frac{d}{dt}u_{i}&=f_i(u),                     &1\leq i \leq n,\\
\frac{d}{dt}u_{i}&=d_i \Delta u_i + f_i(u), \; &n<i\leq n+m.
\end{aligned}
\end{equation}
Let $\overline{u}$ denote a constant steady state of system \eqref{ODE_RD_sys_general} and $J^O$ denote the Jacobian of the \textbf{ODE subsystem} at $\overline{u}$:
\begin{equation}
 J_{ij}^O=\frac{d}{du_i}f_j(u)|_{u=\overline{u}}, \qquad 1\leq i \leq n.
\end{equation}
If $J^O$ has a positive eigenvalue $\lambda_+$, the operator resulting from a linearization of \eqref{ODE_RD_sys_general} around $\overline{u}$ has infinitely many positive eigenvalues.
\end{lemma}
\begin{proof}
The linearization of the right-hand side of \eqref{ODE_RD_sys_general} at $u=\overline{u}$ is of type, written in matrix form:
\begin{equation}
\begin{bmatrix}
J^O    & A_{12} \\
A_{21} & A_{22}-D \Delta
\end{bmatrix},
\end{equation}
and the corresponding eigenvalue problem in matrix form:
\begin{equation}\label{ev-prob}
\begin{bmatrix}
J^O-\lambda    & A_{12} \\
A_{21} & A_{22}-\lambda-D \Delta
\end{bmatrix}
\begin{bmatrix}
\psi \\ \phi
\end{bmatrix}
= \begin{bmatrix} 0 \\ 0 \end{bmatrix}.
\end{equation}
Assuming $\phi_k$ being the eigenfunction of the Laplace operator associated to the $k$th eigenvalue, the matrix is of type \eqref{LA-matrix}. It follows that there exists a sequence of solutions $(\lambda(k),\phi_k)$ of the eigenvalue problem \eqref{ev-prob} with $\lim_{k \rightarrow \infty} \lambda(k)=\lambda_+$, where $\text{Re}(\lambda_+)>0$. 
\end{proof}
Now, we can prove Lemma \ref{thm:DDI}:
\begin{proof}[Proof of Lemma \ref{thm:DDI}]
The Jacobian of the kinetic system of \eqref{eq1-qs}-\eqref{eq2-qs} at $(\overline{u}_0,\overline{w_0})=(0,\kappa_1)$ reads:
\begin{equation}
J=
\begin{bmatrix}
-d_1 & 0 \\
0 & -1
\end{bmatrix}.
\end{equation}
It follows that $(0,\kappa_1)$ is stable solution of \eqref{eq1-qs}-\eqref{eq2-qs} and its kinetic system.\\
The Jacobian of the kinetic system of \eqref{eq1-qs}-\eqref{eq2-qs} at $(\frac{d_1}{a_1-d_1}\frac{1}{w},w)$ reads
\begin{equation}\label{J:w+-}
J=
\begin{bmatrix}
\frac{(a_1-d_1)d_1}{a_1} & \frac{d_1^2}{a} \frac{1}{w^2} \\
-\frac{2 d_1}{a_1-d_1} & -\big(1+(\frac{d_1}{(a_1-d_1)w})^2\big)
\end{bmatrix}.
\end{equation}
Since $J_{11}=\frac{(a_1-d_1)d_1}{a_1}$ is positive, both $(\overline{u}_-,\overline{w}_-)$ and $(\overline{u}_+,\overline{w}_+)$ are unstable solutions of \eqref{eq1-qs}-\eqref{eq2-qs}, see Lemma \ref{lem:EV}.
To determine stability as steady state of the kinetic system, we calculate the determinant and trace of $J$ from \eqref{J:w+-}:
\begin{align}
|J|&=\frac{d_1}{a_1 (a_1-d_1) w^2} (-(a_1-d_1)^2 w^2+d_1^2),\\
\operatorname{tr}(J)&=\frac{(a_1-d_1)d_1}{a_1}-\big(1+(\frac{d_1}{(a_1-d_1)w})^2\big).
\end{align}
We note $|J|\rightarrow -\frac{d_1(a_1-d_1)}{a_1}$ as $w \rightarrow \infty$.\\
The only roots of the determinant $|J|$ are
\begin{equation}
 w_{\pm} = \pm \frac{d_1}{a_1-d_1}.
\end{equation}
Since $w_{\pm}=\frac{\kappa_1}{2} \pm \sqrt{(\frac{\kappa_1}{2})^2-(\frac{d_1}{a_1-d_1})^2}$ and $\frac{\kappa_1}{2}> \frac{d_1}{a_1-d_1}>0$, it follows
\begin{align}
|J(\overline{u}_+,\overline{w}_+)|&<0,\\
|J(\overline{u}_-,\overline{w}_-)|&>0,
\end{align}
what proves instability of $(u_+,w_+)$, because $|J|=\lambda_1 \lambda_2<0$.\\
The stability of $(u_-,w_-)$:\\
Since $J(\overline{u}_-,\overline{w}_-)>0$, $(\overline{u}_-,\overline{w}_-)$ is unstable if and only if $\operatorname{tr}(J(\overline{u}_-,\overline{w}_-))>0$.\\
$\operatorname{tr}(J(\overline{u}_-,\overline{w}_-))>0$ is equivalent to
\begin{equation*}
\begin{aligned}
 \frac{d_1}{a_1} (a_1-d_1)-1&>(\frac{d_1}{a_1-d_1})^2 \frac{1}{w_-^2} \frac{w_+^2}{w_+^2},\\
 \frac{d_1}{a_1} (a_1-d_1)-1&>(\frac{a_1-d_1}{a_1})^2 w_+^2,\\
 \frac{d_1}{a_1-d_1} \frac{d_1^2}{a_1} - (\frac{d_1}{a_1-d_1})^2 &> (\frac{\kappa_1}{2})^2 + 2 \frac{\kappa_1}{2} \sqrt{ (\frac{\kappa_1}{2})^2 - (\frac{d_1}{a_1-d_1})^2 } + (\frac{\kappa_1}{2})^2 - ( \frac{d_1}{a_1-d_1} )^2,\\
 \frac{d_1}{a_1-d_1} \frac{d_1^2}{2a_1} - (\frac{\kappa_1}{2})^2 &> \frac{\kappa_1}{2} \sqrt{ (\frac{\kappa_1}{2})^2 - (\frac{d_1}{a_1-d_1})^2 }.
\end{aligned}
\end{equation*}
This is not satisfied for $\kappa_1^2 > 2 \frac{d_1^3}{a_1(a_1-d_1)}$. 
We continue assuming that $\frac{d_1}{a_1-d_1} \frac{d_1^2}{2a_1} - (\frac{\kappa_1}{2})^2>0$ and define $x:=\frac{\kappa_1}{2}$ and $y:=\frac{d_1}{a_1-d_1}$.
\begin{equation*}
\begin{aligned}
 y \frac{d_1^2}{2a_1}-x^2 &> x \sqrt{x^2 - y^2},\\
 \frac{d_1^4}{4a_1^2} y^2 - \frac{d_1^2}{a_1} x^2 y &> -x^2 y^2,\\
 (y-\frac{d_1^2}{a_1})x^2 + \frac{d_1^4}{4 a_1^2} y &> 0.
\end{aligned}
\end{equation*}
This is satisfied if and only if 
\begin{equation} \label{x1}
\begin{aligned}
  y&>\frac{d_1^2}{a_1},\\
(\Leftrightarrow a_1 &<\frac{d_1^2}{d_1-1}),
\end{aligned}
\end{equation}
or 
\begin{equation} \label{x2}
\begin{aligned}
x^2&<\frac{d_1^4}{4a_1^2}y (y-\frac{d_1^2}{a_1})^{-1},\\
(\Leftrightarrow \kappa_1^2 &< \frac{d_1^4}{a_1} \frac{1}{a_1-d_1(a_1-d_1)}).
\end{aligned}
\end{equation}
Negation yields the result.
\end{proof}
\begin{proof}[Proof of Lemma \ref{thm:mass}]
Adding a multiple of \eqref{eq1-qs} and \eqref{eq2-qs} and integrating over $\Omega$ leads to
\begin{equation}\label{m:u+w}
\begin{aligned}
\frac{d}{dt} \int \frac{1}{a_1}u+w dx &=\int \left( (\frac{u^2 w}{1+uw}-\frac{d_1}{a_1} u) - w - u^2 w + \kappa_1 \right) dx, \\
 &\leq \int\left(-\frac{d_1}{a_1} u-w+\kappa_1\right)dx,\\
 &\leq -\min(d_1,1)\int \left(\frac{1}{a_1}u+w\right)dx+\kappa_1 \mu(\Omega).
\end{aligned}
\end{equation}
This leads to
\begin{equation}
\limsup_{t \rightarrow \infty}\left( \frac{1}{a_1}\int u dx + \int w dx \right) \leq \frac{\kappa_1}{\min(d_1,1)} \mu(\Omega).
\end{equation}
Additionally, it immediately follows by integrating \eqref{eq2-qs} over $\Omega$:
\begin{equation}\label{m:w}
\frac{d}{dt} \int w dx \leq -\int w dx+\kappa \mu(\Omega).
\end{equation}
From \eqref{m:w} follows 
\begin{equation}
\limsup_{t \rightarrow \infty} \int w dx \leq \kappa_1 \mu(\Omega).
\end{equation}
Since $w\geq 0$, it follows from \eqref{m:u+w}
\begin{equation}
\limsup_{t \rightarrow \infty} \int u dx \leq \frac{a_1}{\min(d_1,1)}\kappa_1 \mu(\Omega).
\end{equation}
\end{proof}
\begin{proof}[Proof of Lemma \ref{thm:bounded}]
 The proof is analogues to the proof of Lemma \ref{thm:mass}, without integrating over $\Omega$. 
\end{proof}
\begin{proof}[Proof of Lemma \ref{le:ampl}]
 The eigenvector associated to the eigenvalue $\lambda_+$ of a 2x2 matrix $(a_{ij}+\delta_{i2}\delta_{j2}D_w k^2)$ is $v_k:=\begin{bmatrix} 1 \\ \frac{a_{21}}{\lambda_+-a_{22}+D_w k^2} \end{bmatrix}$.
 It follows that $v_k \phi_k$ is the eigenvector of $J$ associated to $\lambda_+(k)$.
\end{proof}
\subsection{Additional figures}
In this section, we show numerically obtained solutions which were referred to in the previous sections. Additionally, we show for convenience the explicit formula for the "perturbation" function $p$, defined by \eqref{ICdef1}-\eqref{ICdef5}:
\begin{equation}\label{spline}
p(x)=\left\{\begin{array}{ll} 
              \frac{4 (-1+s-\epsilon)}{(s-\epsilon)(-2s+2s^2-\epsilon)}x^2 -1     ,                                                              & x \in [0,s-\epsilon),\\ 
              \frac{2(1+2\epsilon)x^2-4(s+\epsilon)x+2s^2+2s\epsilon-2s^2\epsilon-\epsilon^2}{\epsilon (-2s+2s^2-\epsilon)},                     & x \in [s-\epsilon, s+\epsilon],\\
              \frac{(2 s + 4 s^2 - 2 s^3 + 3 \epsilon + 3 s \epsilon - 2 s^2 \epsilon + \epsilon^2 - 8 x (s + \epsilon) + 4 x^2 (s + \epsilon))}{(-2s+2s^2-\epsilon)(-1+s+\epsilon)},
 & x \in (s+\epsilon,1].\\
            \end{array}\right.
\end{equation}
\begin{minipage}[t]{40em}
\begin{center}
\begin{tabular}{p{15em}p{15em}}
   \centering
   \includegraphics[width=15em]{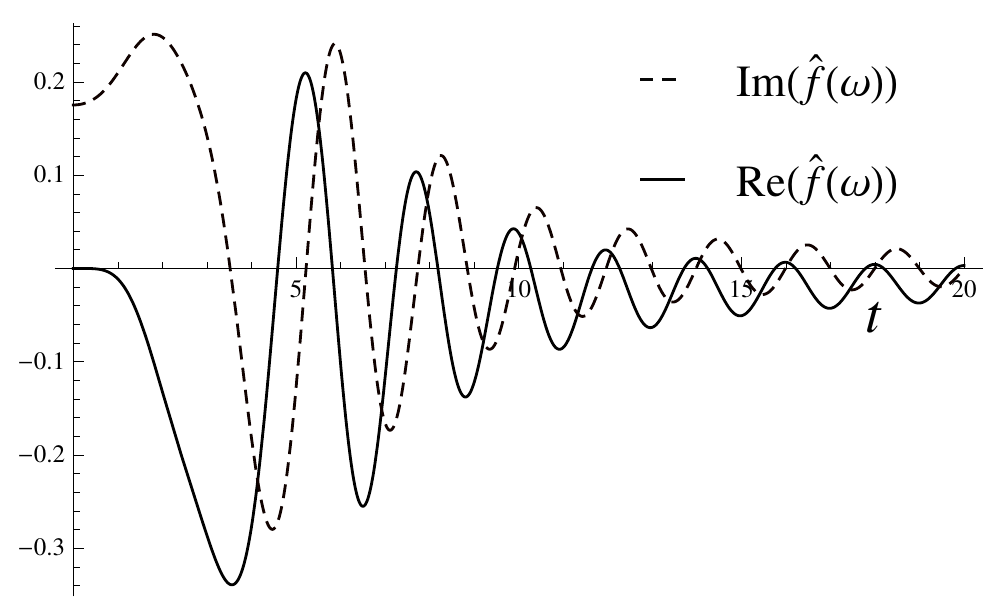}
&
   \centering
   \makeatletter{}\begingroup
  \makeatletter
  \providecommand\color[2][]{    \GenericError{(gnuplot) \space\space\space\@spaces}{      Package color not loaded in conjunction with
      terminal option `colourtext'    }{See the gnuplot documentation for explanation.    }{Either use 'blacktext' in gnuplot or load the package
      color.sty in LaTeX.}    \renewcommand\color[2][]{}  }  \providecommand\includegraphics[2][]{    \GenericError{(gnuplot) \space\space\space\@spaces}{      Package graphicx or graphics not loaded    }{See the gnuplot documentation for explanation.    }{The gnuplot epslatex terminal needs graphicx.sty or graphics.sty.}    \renewcommand\includegraphics[2][]{}  }  \providecommand\rotatebox[2]{#2}  \@ifundefined{ifGPcolor}{    \newif\ifGPcolor
    \GPcolorfalse
  }{}  \@ifundefined{ifGPblacktext}{    \newif\ifGPblacktext
    \GPblacktexttrue
  }{}    \let\gplgaddtomacro\g@addto@macro
    \gdef\gplbacktext{}  \gdef\gplfronttext{}  \makeatother
  \ifGPblacktext
        \def\colorrgb#1{}    \def\colorgray#1{}  \else
        \ifGPcolor
      \def\colorrgb#1{\color[rgb]{#1}}      \def\colorgray#1{\color[gray]{#1}}      \expandafter\def\csname LTw\endcsname{\color{white}}      \expandafter\def\csname LTb\endcsname{\color{black}}      \expandafter\def\csname LTa\endcsname{\color{black}}      \expandafter\def\csname LT0\endcsname{\color[rgb]{1,0,0}}      \expandafter\def\csname LT1\endcsname{\color[rgb]{0,1,0}}      \expandafter\def\csname LT2\endcsname{\color[rgb]{0,0,1}}      \expandafter\def\csname LT3\endcsname{\color[rgb]{1,0,1}}      \expandafter\def\csname LT4\endcsname{\color[rgb]{0,1,1}}      \expandafter\def\csname LT5\endcsname{\color[rgb]{1,1,0}}      \expandafter\def\csname LT6\endcsname{\color[rgb]{0,0,0}}      \expandafter\def\csname LT7\endcsname{\color[rgb]{1,0.3,0}}      \expandafter\def\csname LT8\endcsname{\color[rgb]{0.5,0.5,0.5}}    \else
            \def\colorrgb#1{\color{black}}      \def\colorgray#1{\color[gray]{#1}}      \expandafter\def\csname LTw\endcsname{\color{white}}      \expandafter\def\csname LTb\endcsname{\color{black}}      \expandafter\def\csname LTa\endcsname{\color{black}}      \expandafter\def\csname LT0\endcsname{\color{black}}      \expandafter\def\csname LT1\endcsname{\color{black}}      \expandafter\def\csname LT2\endcsname{\color{black}}      \expandafter\def\csname LT3\endcsname{\color{black}}      \expandafter\def\csname LT4\endcsname{\color{black}}      \expandafter\def\csname LT5\endcsname{\color{black}}      \expandafter\def\csname LT6\endcsname{\color{black}}      \expandafter\def\csname LT7\endcsname{\color{black}}      \expandafter\def\csname LT8\endcsname{\color{black}}    \fi
  \fi
  \setlength{\unitlength}{0.0500bp}  \begin{picture}(3600.00,2520.00)    \gplgaddtomacro\gplfronttext{      \csname LTb\endcsname      \put(902,440){\makebox(0,0)[r]{\strut{} 0.2}}      \put(902,667){\makebox(0,0)[r]{\strut{} 0.25}}      \put(902,894){\makebox(0,0)[r]{\strut{} 0.3}}      \put(902,1121){\makebox(0,0)[r]{\strut{} 0.35}}      \put(902,1348){\makebox(0,0)[r]{\strut{} 0.4}}      \put(902,1575){\makebox(0,0)[r]{\strut{} 0.45}}      \put(902,1802){\makebox(0,0)[r]{\strut{} 0.5}}      \put(902,2029){\makebox(0,0)[r]{\strut{} 0.55}}      \put(902,2256){\makebox(0,0)[r]{\strut{} 0.6}}      \put(1034,220){\makebox(0,0){\strut{} 0}}      \put(1347,220){\makebox(0,0){\strut{} 2}}      \put(1660,220){\makebox(0,0){\strut{} 4}}      \put(1973,220){\makebox(0,0){\strut{} 6}}      \put(2287,220){\makebox(0,0){\strut{} 8}}      \put(2600,220){\makebox(0,0){\strut{} 10}}      \put(2913,220){\makebox(0,0){\strut{} 12}}      \put(3226,220){\makebox(0,0){\strut{} 14}}      \put(2287,0){\makebox(0,0){\strut{} $t$}}    }    \gplgaddtomacro\gplfronttext{      \csname LTb\endcsname      \put(2239,1053){\makebox(0,0)[r]{\strut{}$x_0$}}      \csname LTb\endcsname      \put(2239,833){\makebox(0,0)[r]{\strut{}$x_1$}}      \csname LTb\endcsname      \put(2239,613){\makebox(0,0)[r]{\strut{}$x_2$}}    }    \gplbacktext
    \put(0,0){\includegraphics{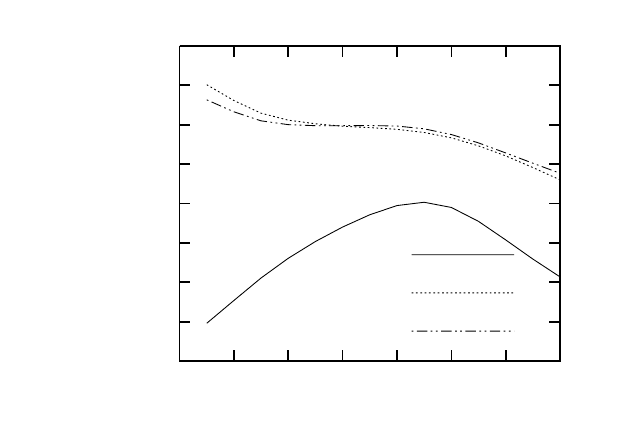}}    \gplfronttext
  \end{picture}\endgroup
 
\end{tabular}
\end{center}
\captionof{figure}{Left: Finite Fourier Transform $\hat{f}(\omega)=\int_0^1 \cos(4\pi x^2) e^{-i \pi \omega x}dx$. Right: Order $\frac{\log(\frac{u(t,x_i)-\overline{u}_-}{u_0(x_i)-\overline{u}_-})}{t}$ of the growth of the perturbation $-\epsilon \cos(4\pi x^2)$ of $\overline{u}_-$ at $x_0=0.250092$, $x_1=\frac{1}{2}$, $x_2=0.866028 \approx \frac{\sqrt{3}}{2}$. See fig. \ref{fig:cosxx} for the numerically obtained solution.}
\label{fig:ff-cosxx}
\end{minipage}\\
\begin{minipage}[t]{40em}
\begin{center}
\begin{tabular}{p{15em}p{15em}}
   \centering
   \makeatletter{}\begingroup
  \makeatletter
  \providecommand\color[2][]{    \GenericError{(gnuplot) \space\space\space\@spaces}{      Package color not loaded in conjunction with
      terminal option `colourtext'    }{See the gnuplot documentation for explanation.    }{Either use 'blacktext' in gnuplot or load the package
      color.sty in LaTeX.}    \renewcommand\color[2][]{}  }  \providecommand\includegraphics[2][]{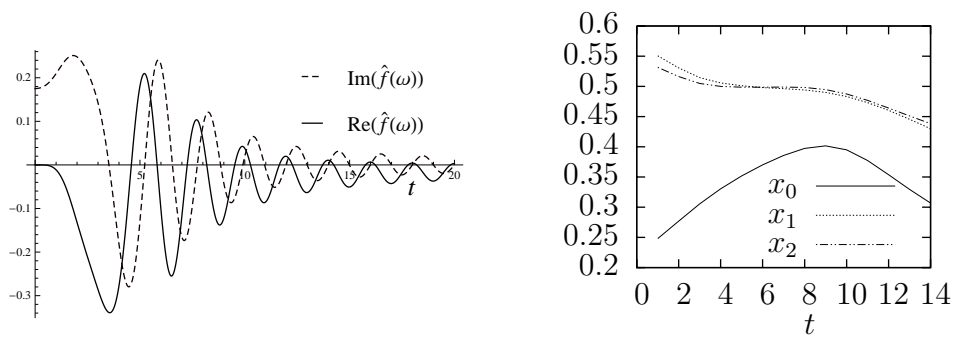}  \providecommand\rotatebox[2]{#2}  \@ifundefined{ifGPcolor}{    \newif\ifGPcolor
    \GPcolorfalse
  }{}  \@ifundefined{ifGPblacktext}{    \newif\ifGPblacktext
    \GPblacktexttrue
  }{}    \let\gplgaddtomacro\g@addto@macro
    \gdef\gplbacktext{}  \gdef\gplfronttext{}  \makeatother
  \ifGPblacktext
        \def\colorrgb#1{}    \def\colorgray#1{}  \else
        \ifGPcolor
      \def\colorrgb#1{\color[rgb]{#1}}      \def\colorgray#1{\color[gray]{#1}}      \expandafter\def\csname LTw\endcsname{\color{white}}      \expandafter\def\csname LTb\endcsname{\color{black}}      \expandafter\def\csname LTa\endcsname{\color{black}}      \expandafter\def\csname LT0\endcsname{\color[rgb]{1,0,0}}      \expandafter\def\csname LT1\endcsname{\color[rgb]{0,1,0}}      \expandafter\def\csname LT2\endcsname{\color[rgb]{0,0,1}}      \expandafter\def\csname LT3\endcsname{\color[rgb]{1,0,1}}      \expandafter\def\csname LT4\endcsname{\color[rgb]{0,1,1}}      \expandafter\def\csname LT5\endcsname{\color[rgb]{1,1,0}}      \expandafter\def\csname LT6\endcsname{\color[rgb]{0,0,0}}      \expandafter\def\csname LT7\endcsname{\color[rgb]{1,0.3,0}}      \expandafter\def\csname LT8\endcsname{\color[rgb]{0.5,0.5,0.5}}    \else
            \def\colorrgb#1{\color{black}}      \def\colorgray#1{\color[gray]{#1}}      \expandafter\def\csname LTw\endcsname{\color{white}}      \expandafter\def\csname LTb\endcsname{\color{black}}      \expandafter\def\csname LTa\endcsname{\color{black}}      \expandafter\def\csname LT0\endcsname{\color{black}}      \expandafter\def\csname LT1\endcsname{\color{black}}      \expandafter\def\csname LT2\endcsname{\color{black}}      \expandafter\def\csname LT3\endcsname{\color{black}}      \expandafter\def\csname LT4\endcsname{\color{black}}      \expandafter\def\csname LT5\endcsname{\color{black}}      \expandafter\def\csname LT6\endcsname{\color{black}}      \expandafter\def\csname LT7\endcsname{\color{black}}      \expandafter\def\csname LT8\endcsname{\color{black}}    \fi
  \fi
  \setlength{\unitlength}{0.0500bp}  \begin{picture}(3600.00,2520.00)    \gplgaddtomacro\gplfronttext{      \csname LTb\endcsname      \put(902,440){\makebox(0,0)[r]{\strut{} 0.25}}      \put(902,803){\makebox(0,0)[r]{\strut{} 0.3}}      \put(902,1166){\makebox(0,0)[r]{\strut{} 0.35}}      \put(902,1530){\makebox(0,0)[r]{\strut{} 0.4}}      \put(902,1893){\makebox(0,0)[r]{\strut{} 0.45}}      \put(902,2256){\makebox(0,0)[r]{\strut{} 0.5}}      \put(1034,220){\makebox(0,0){\strut{} 0}}      \put(1347,220){\makebox(0,0){\strut{} 2}}      \put(1660,220){\makebox(0,0){\strut{} 4}}      \put(1973,220){\makebox(0,0){\strut{} 6}}      \put(2287,220){\makebox(0,0){\strut{} 8}}      \put(2600,220){\makebox(0,0){\strut{} 10}}      \put(2913,220){\makebox(0,0){\strut{} 12}}      \put(3226,220){\makebox(0,0){\strut{} 14}}      \put(2287,0){\makebox(0,0){\strut{} $t$}}    }    \gplgaddtomacro\gplfronttext{      \csname LTb\endcsname      \put(2239,833){\makebox(0,0)[r]{\strut{}\small{$x_1$}}}      \csname LTb\endcsname      \put(2239,613){\makebox(0,0)[r]{\strut{}\small{$x_2$}}}    }    \gplbacktext
    \put(0,0){\includegraphics{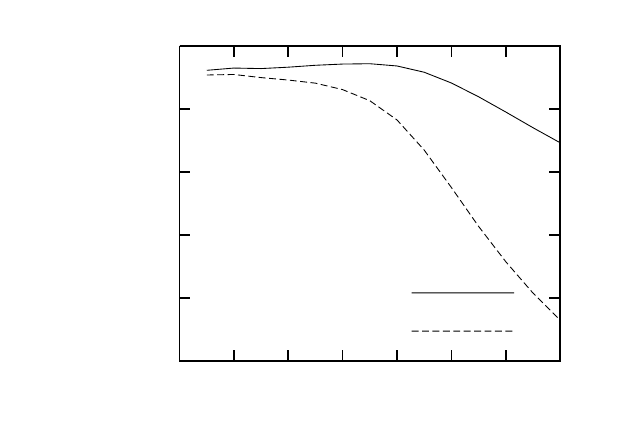}}    \gplfronttext
  \end{picture}\endgroup
 
\end{tabular}
\end{center}
\captionof{figure}{Order $\frac{\log(\frac{u(t,x_i)-\overline{u}_-}{u_0(x_i)-\overline{u}_-})}{t}$ of the growth of perturbation $-\epsilon \cos(4\pi x)$ of $\overline{u_-}$ at $x_1=0.250092$ and $x_2=\frac{1}{2}$. See fig. \ref{fig:cos} for the numerically obtained solution.}
\label{fig:ff-cosx}
\end{minipage}\\
\begin{minipage}[t]{40em}
\begin{center}
\begin{tabular}{p{15em}p{15em}}
   \centering
   \makeatletter{}\begingroup
  \makeatletter
  \providecommand\color[2][]{    \GenericError{(gnuplot) \space\space\space\@spaces}{      Package color not loaded in conjunction with
      terminal option `colourtext'    }{See the gnuplot documentation for explanation.    }{Either use 'blacktext' in gnuplot or load the package
      color.sty in LaTeX.}    \renewcommand\color[2][]{}  }  \providecommand\includegraphics[2][]{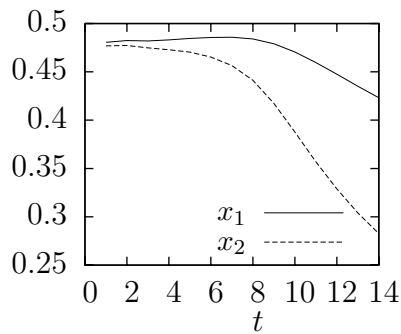}  \providecommand\rotatebox[2]{#2}  \@ifundefined{ifGPcolor}{    \newif\ifGPcolor
    \GPcolorfalse
  }{}  \@ifundefined{ifGPblacktext}{    \newif\ifGPblacktext
    \GPblacktexttrue
  }{}    \let\gplgaddtomacro\g@addto@macro
    \gdef\gplbacktext{}  \gdef\gplfronttext{}  \makeatother
  \ifGPblacktext
        \def\colorrgb#1{}    \def\colorgray#1{}  \else
        \ifGPcolor
      \def\colorrgb#1{\color[rgb]{#1}}      \def\colorgray#1{\color[gray]{#1}}      \expandafter\def\csname LTw\endcsname{\color{white}}      \expandafter\def\csname LTb\endcsname{\color{black}}      \expandafter\def\csname LTa\endcsname{\color{black}}      \expandafter\def\csname LT0\endcsname{\color[rgb]{1,0,0}}      \expandafter\def\csname LT1\endcsname{\color[rgb]{0,1,0}}      \expandafter\def\csname LT2\endcsname{\color[rgb]{0,0,1}}      \expandafter\def\csname LT3\endcsname{\color[rgb]{1,0,1}}      \expandafter\def\csname LT4\endcsname{\color[rgb]{0,1,1}}      \expandafter\def\csname LT5\endcsname{\color[rgb]{1,1,0}}      \expandafter\def\csname LT6\endcsname{\color[rgb]{0,0,0}}      \expandafter\def\csname LT7\endcsname{\color[rgb]{1,0.3,0}}      \expandafter\def\csname LT8\endcsname{\color[rgb]{0.5,0.5,0.5}}    \else
            \def\colorrgb#1{\color{black}}      \def\colorgray#1{\color[gray]{#1}}      \expandafter\def\csname LTw\endcsname{\color{white}}      \expandafter\def\csname LTb\endcsname{\color{black}}      \expandafter\def\csname LTa\endcsname{\color{black}}      \expandafter\def\csname LT0\endcsname{\color{black}}      \expandafter\def\csname LT1\endcsname{\color{black}}      \expandafter\def\csname LT2\endcsname{\color{black}}      \expandafter\def\csname LT3\endcsname{\color{black}}      \expandafter\def\csname LT4\endcsname{\color{black}}      \expandafter\def\csname LT5\endcsname{\color{black}}      \expandafter\def\csname LT6\endcsname{\color{black}}      \expandafter\def\csname LT7\endcsname{\color{black}}      \expandafter\def\csname LT8\endcsname{\color{black}}    \fi
  \fi
  \setlength{\unitlength}{0.0500bp}  \begin{picture}(3600.00,3528.00)    \gplgaddtomacro\gplbacktext{    }    \gplgaddtomacro\gplfronttext{      \csname LTb\endcsname      \put(2135,2804){\makebox(0,0)[r]{\strut{}shape of $u_0$}}      \csname LTb\endcsname      \put(507,1092){\makebox(0,0){\strut{} 0}}      \put(1697,557){\makebox(0,0){\strut{} 25}}      \put(908,819){\makebox(0,0){\strut{}$t$}}      \put(1893,615){\makebox(0,0){\strut{} 0}}      \put(3083,1150){\makebox(0,0){\strut{} 1}}      \put(2692,819){\makebox(0,0){\strut{}$x$}}      \put(484,1656){\makebox(0,0)[r]{\strut{} 0}}      \put(484,1850){\makebox(0,0)[r]{\strut{} 100}}      \put(484,2043){\makebox(0,0)[r]{\strut{} 200}}      \put(484,2237){\makebox(0,0)[r]{\strut{} 300}}      \put(484,2431){\makebox(0,0)[r]{\strut{} 400}}    }    \gplbacktext
    \put(0,0){\includegraphics{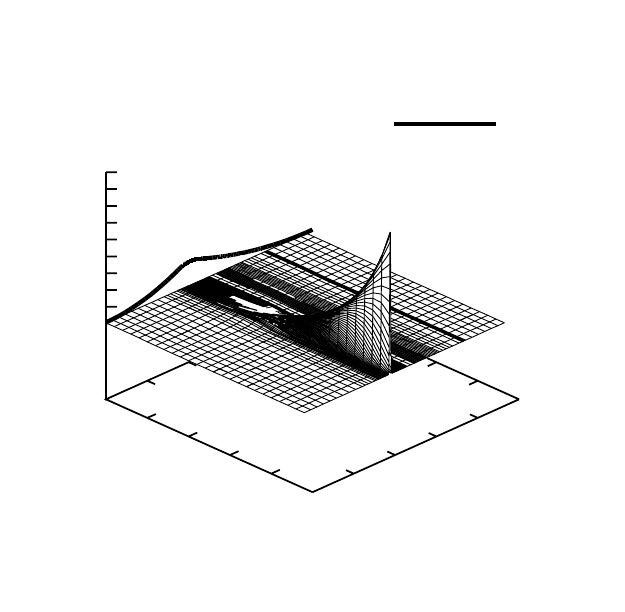}}    \gplfronttext
  \end{picture}\endgroup
 
&
   \centering
   \makeatletter{}\begingroup
  \makeatletter
  \providecommand\color[2][]{    \GenericError{(gnuplot) \space\space\space\@spaces}{      Package color not loaded in conjunction with
      terminal option `colourtext'    }{See the gnuplot documentation for explanation.    }{Either use 'blacktext' in gnuplot or load the package
      color.sty in LaTeX.}    \renewcommand\color[2][]{}  }  \providecommand\includegraphics[2][]{    \GenericError{(gnuplot) \space\space\space\@spaces}{      Package graphicx or graphics not loaded    }{See the gnuplot documentation for explanation.    }{The gnuplot epslatex terminal needs graphicx.sty or graphics.sty.}    \renewcommand\includegraphics[2][]{}  }  \providecommand\rotatebox[2]{#2}  \@ifundefined{ifGPcolor}{    \newif\ifGPcolor
    \GPcolorfalse
  }{}  \@ifundefined{ifGPblacktext}{    \newif\ifGPblacktext
    \GPblacktexttrue
  }{}    \let\gplgaddtomacro\g@addto@macro
    \gdef\gplbacktext{}  \gdef\gplfronttext{}  \makeatother
  \ifGPblacktext
        \def\colorrgb#1{}    \def\colorgray#1{}  \else
        \ifGPcolor
      \def\colorrgb#1{\color[rgb]{#1}}      \def\colorgray#1{\color[gray]{#1}}      \expandafter\def\csname LTw\endcsname{\color{white}}      \expandafter\def\csname LTb\endcsname{\color{black}}      \expandafter\def\csname LTa\endcsname{\color{black}}      \expandafter\def\csname LT0\endcsname{\color[rgb]{1,0,0}}      \expandafter\def\csname LT1\endcsname{\color[rgb]{0,1,0}}      \expandafter\def\csname LT2\endcsname{\color[rgb]{0,0,1}}      \expandafter\def\csname LT3\endcsname{\color[rgb]{1,0,1}}      \expandafter\def\csname LT4\endcsname{\color[rgb]{0,1,1}}      \expandafter\def\csname LT5\endcsname{\color[rgb]{1,1,0}}      \expandafter\def\csname LT6\endcsname{\color[rgb]{0,0,0}}      \expandafter\def\csname LT7\endcsname{\color[rgb]{1,0.3,0}}      \expandafter\def\csname LT8\endcsname{\color[rgb]{0.5,0.5,0.5}}    \else
            \def\colorrgb#1{\color{black}}      \def\colorgray#1{\color[gray]{#1}}      \expandafter\def\csname LTw\endcsname{\color{white}}      \expandafter\def\csname LTb\endcsname{\color{black}}      \expandafter\def\csname LTa\endcsname{\color{black}}      \expandafter\def\csname LT0\endcsname{\color{black}}      \expandafter\def\csname LT1\endcsname{\color{black}}      \expandafter\def\csname LT2\endcsname{\color{black}}      \expandafter\def\csname LT3\endcsname{\color{black}}      \expandafter\def\csname LT4\endcsname{\color{black}}      \expandafter\def\csname LT5\endcsname{\color{black}}      \expandafter\def\csname LT6\endcsname{\color{black}}      \expandafter\def\csname LT7\endcsname{\color{black}}      \expandafter\def\csname LT8\endcsname{\color{black}}    \fi
  \fi
  \setlength{\unitlength}{0.0500bp}  \begin{picture}(3600.00,3528.00)    \gplgaddtomacro\gplbacktext{    }    \gplgaddtomacro\gplfronttext{      \csname LTb\endcsname      \put(507,1092){\makebox(0,0){\strut{} 0}}      \put(1697,557){\makebox(0,0){\strut{} 25}}      \put(908,819){\makebox(0,0){\strut{}$t$}}      \put(1893,615){\makebox(0,0){\strut{} 0}}      \put(3083,1150){\makebox(0,0){\strut{} 1}}      \put(2692,819){\makebox(0,0){\strut{}$x$}}      \put(484,1656){\makebox(0,0)[r]{\strut{} 0}}      \put(484,1905){\makebox(0,0)[r]{\strut{} 0.2}}      \put(484,2154){\makebox(0,0)[r]{\strut{} 0.4}}      \put(484,2404){\makebox(0,0)[r]{\strut{} 0.6}}    }    \gplbacktext
    \put(0,0){\includegraphics{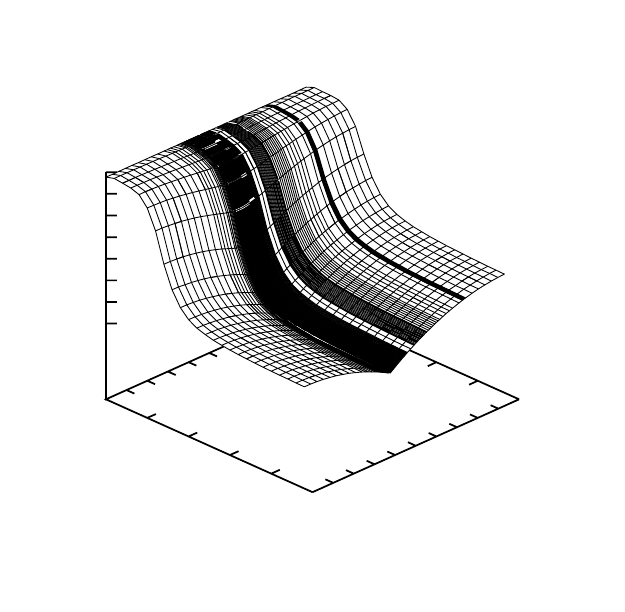}}    \gplfronttext
  \end{picture}\endgroup
 
\end{tabular}
\end{center}
\captionof{figure}{Numerically obtained solution for initial conditions \eqref{IC:sh} with $\epsilon=0.05$, $\epsilon_1=0.1$, $s=0.4$, $\overline{u}=2.215$, $\overline{w}=0.677123$ and parameters
$a_1=2.5,d_1=1.5,\kappa_1=4$ and $D_w=5.8541$,. Left: component $u$. Right: component $w$.}
\label{fig:np}
\end{minipage}\\

\begin{minipage}[t]{40em}
\begin{center}
\begin{tabular}{p{15em}p{15em}}
   \centering
   \makeatletter{}\begingroup
  \makeatletter
  \providecommand\color[2][]{    \GenericError{(gnuplot) \space\space\space\@spaces}{      Package color not loaded in conjunction with
      terminal option `colourtext'    }{See the gnuplot documentation for explanation.    }{Either use 'blacktext' in gnuplot or load the package
      color.sty in LaTeX.}    \renewcommand\color[2][]{}  }  \providecommand\includegraphics[2][]{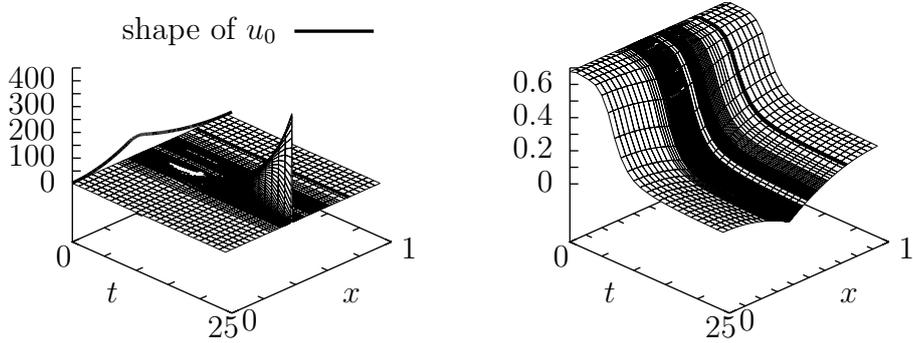}  \providecommand\rotatebox[2]{#2}  \@ifundefined{ifGPcolor}{    \newif\ifGPcolor
    \GPcolorfalse
  }{}  \@ifundefined{ifGPblacktext}{    \newif\ifGPblacktext
    \GPblacktexttrue
  }{}    \let\gplgaddtomacro\g@addto@macro
    \gdef\gplbacktext{}  \gdef\gplfronttext{}  \makeatother
  \ifGPblacktext
        \def\colorrgb#1{}    \def\colorgray#1{}  \else
        \ifGPcolor
      \def\colorrgb#1{\color[rgb]{#1}}      \def\colorgray#1{\color[gray]{#1}}      \expandafter\def\csname LTw\endcsname{\color{white}}      \expandafter\def\csname LTb\endcsname{\color{black}}      \expandafter\def\csname LTa\endcsname{\color{black}}      \expandafter\def\csname LT0\endcsname{\color[rgb]{1,0,0}}      \expandafter\def\csname LT1\endcsname{\color[rgb]{0,1,0}}      \expandafter\def\csname LT2\endcsname{\color[rgb]{0,0,1}}      \expandafter\def\csname LT3\endcsname{\color[rgb]{1,0,1}}      \expandafter\def\csname LT4\endcsname{\color[rgb]{0,1,1}}      \expandafter\def\csname LT5\endcsname{\color[rgb]{1,1,0}}      \expandafter\def\csname LT6\endcsname{\color[rgb]{0,0,0}}      \expandafter\def\csname LT7\endcsname{\color[rgb]{1,0.3,0}}      \expandafter\def\csname LT8\endcsname{\color[rgb]{0.5,0.5,0.5}}    \else
            \def\colorrgb#1{\color{black}}      \def\colorgray#1{\color[gray]{#1}}      \expandafter\def\csname LTw\endcsname{\color{white}}      \expandafter\def\csname LTb\endcsname{\color{black}}      \expandafter\def\csname LTa\endcsname{\color{black}}      \expandafter\def\csname LT0\endcsname{\color{black}}      \expandafter\def\csname LT1\endcsname{\color{black}}      \expandafter\def\csname LT2\endcsname{\color{black}}      \expandafter\def\csname LT3\endcsname{\color{black}}      \expandafter\def\csname LT4\endcsname{\color{black}}      \expandafter\def\csname LT5\endcsname{\color{black}}      \expandafter\def\csname LT6\endcsname{\color{black}}      \expandafter\def\csname LT7\endcsname{\color{black}}      \expandafter\def\csname LT8\endcsname{\color{black}}    \fi
  \fi
  \setlength{\unitlength}{0.0500bp}  \begin{picture}(3600.00,3528.00)    \gplgaddtomacro\gplbacktext{    }    \gplgaddtomacro\gplfronttext{      \csname LTb\endcsname      \put(507,1092){\makebox(0,0){\strut{} 0}}      \put(1697,557){\makebox(0,0){\strut{} 25}}      \put(908,819){\makebox(0,0){\strut{}$t$}}      \put(1893,615){\makebox(0,0){\strut{} 0}}      \put(3083,1150){\makebox(0,0){\strut{} 1}}      \put(2692,819){\makebox(0,0){\strut{}$x$}}      \put(484,1656){\makebox(0,0)[r]{\strut{} 0}}      \put(484,2004){\makebox(0,0)[r]{\strut{} 0.1}}      \put(484,2354){\makebox(0,0)[r]{\strut{} 0.2}}    }    \gplbacktext
    \put(0,0){\includegraphics{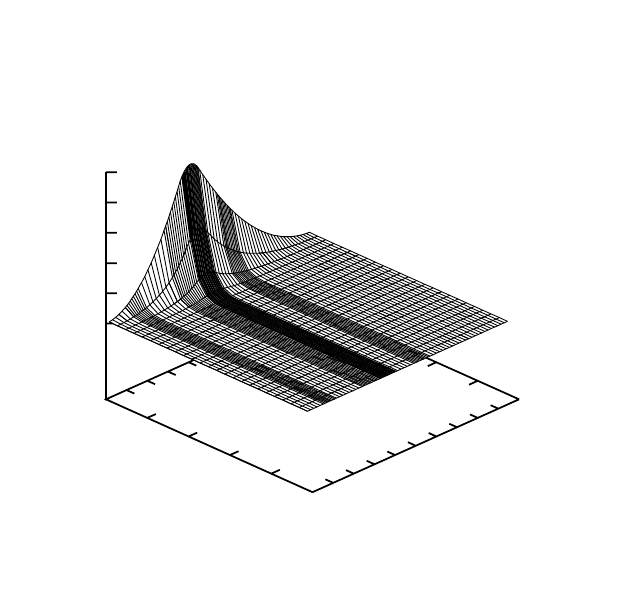}}    \gplfronttext
  \end{picture}\endgroup
 
&
   \centering
   \makeatletter{}\begingroup
  \makeatletter
  \providecommand\color[2][]{    \GenericError{(gnuplot) \space\space\space\@spaces}{      Package color not loaded in conjunction with
      terminal option `colourtext'    }{See the gnuplot documentation for explanation.    }{Either use 'blacktext' in gnuplot or load the package
      color.sty in LaTeX.}    \renewcommand\color[2][]{}  }  \providecommand\includegraphics[2][]{    \GenericError{(gnuplot) \space\space\space\@spaces}{      Package graphicx or graphics not loaded    }{See the gnuplot documentation for explanation.    }{The gnuplot epslatex terminal needs graphicx.sty or graphics.sty.}    \renewcommand\includegraphics[2][]{}  }  \providecommand\rotatebox[2]{#2}  \@ifundefined{ifGPcolor}{    \newif\ifGPcolor
    \GPcolorfalse
  }{}  \@ifundefined{ifGPblacktext}{    \newif\ifGPblacktext
    \GPblacktexttrue
  }{}    \let\gplgaddtomacro\g@addto@macro
    \gdef\gplbacktext{}  \gdef\gplfronttext{}  \makeatother
  \ifGPblacktext
        \def\colorrgb#1{}    \def\colorgray#1{}  \else
        \ifGPcolor
      \def\colorrgb#1{\color[rgb]{#1}}      \def\colorgray#1{\color[gray]{#1}}      \expandafter\def\csname LTw\endcsname{\color{white}}      \expandafter\def\csname LTb\endcsname{\color{black}}      \expandafter\def\csname LTa\endcsname{\color{black}}      \expandafter\def\csname LT0\endcsname{\color[rgb]{1,0,0}}      \expandafter\def\csname LT1\endcsname{\color[rgb]{0,1,0}}      \expandafter\def\csname LT2\endcsname{\color[rgb]{0,0,1}}      \expandafter\def\csname LT3\endcsname{\color[rgb]{1,0,1}}      \expandafter\def\csname LT4\endcsname{\color[rgb]{0,1,1}}      \expandafter\def\csname LT5\endcsname{\color[rgb]{1,1,0}}      \expandafter\def\csname LT6\endcsname{\color[rgb]{0,0,0}}      \expandafter\def\csname LT7\endcsname{\color[rgb]{1,0.3,0}}      \expandafter\def\csname LT8\endcsname{\color[rgb]{0.5,0.5,0.5}}    \else
            \def\colorrgb#1{\color{black}}      \def\colorgray#1{\color[gray]{#1}}      \expandafter\def\csname LTw\endcsname{\color{white}}      \expandafter\def\csname LTb\endcsname{\color{black}}      \expandafter\def\csname LTa\endcsname{\color{black}}      \expandafter\def\csname LT0\endcsname{\color{black}}      \expandafter\def\csname LT1\endcsname{\color{black}}      \expandafter\def\csname LT2\endcsname{\color{black}}      \expandafter\def\csname LT3\endcsname{\color{black}}      \expandafter\def\csname LT4\endcsname{\color{black}}      \expandafter\def\csname LT5\endcsname{\color{black}}      \expandafter\def\csname LT6\endcsname{\color{black}}      \expandafter\def\csname LT7\endcsname{\color{black}}      \expandafter\def\csname LT8\endcsname{\color{black}}    \fi
  \fi
  \setlength{\unitlength}{0.0500bp}  \begin{picture}(3600.00,3528.00)    \gplgaddtomacro\gplbacktext{    }    \gplgaddtomacro\gplfronttext{      \csname LTb\endcsname      \put(507,1092){\makebox(0,0){\strut{} 0}}      \put(1697,557){\makebox(0,0){\strut{} 25}}      \put(908,819){\makebox(0,0){\strut{}$t$}}      \put(1893,615){\makebox(0,0){\strut{} 0}}      \put(3083,1150){\makebox(0,0){\strut{} 1}}      \put(2692,819){\makebox(0,0){\strut{}$x$}}      \put(484,1656){\makebox(0,0)[r]{\strut{} 2.8}}      \put(484,1946){\makebox(0,0)[r]{\strut{} 3.2}}      \put(484,2237){\makebox(0,0)[r]{\strut{} 3.6}}      \put(484,2528){\makebox(0,0)[r]{\strut{} 4}}    }    \gplbacktext
    \put(0,0){\includegraphics{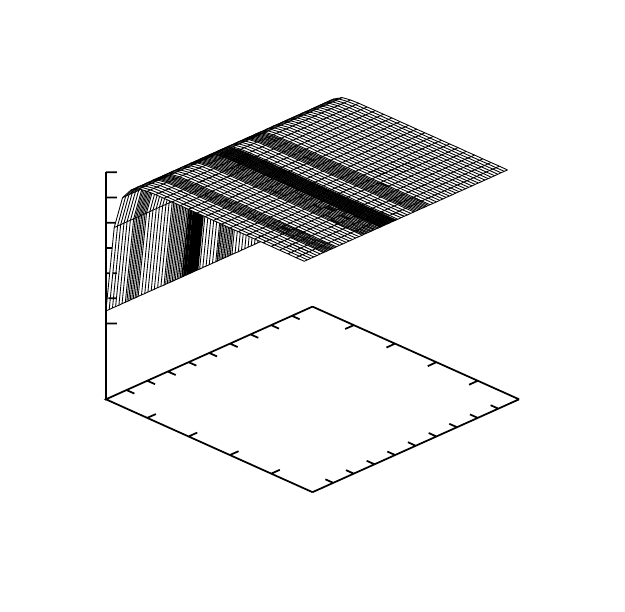}}    \gplfronttext
  \end{picture}\endgroup
 
\end{tabular}
\end{center}
\captionof{figure}{Numerical solution for initial conditions very close to the stable steady state $(\overline{u}_0,\overline{w}_0)$, parameters $a_1=2.5,d_1=1.5,\kappa_1=4$. Left: component $u$. Right: component $w$.}
\label{fig:trivstab}
\end{minipage}\\
\subsection{Mesh asymptotic}
In this section, we investigate the asymptotic behavior of the error due to numerical approximation. Since we do not know the true solution, we investigate the asymptotic behavior of the difference of the solution $(u,w)$ and a calculated "reference solution" $(u_{\text{ref}},w_{\text{ref}})$. The reference solution is the numerical solution on a much finer mesh in time and space.\\
First, we show this error for the approximation of the configuration in the introductory part for large diffusion coefficient. In that case, only a single spike arises close to the position where the initial condition has a maximum.
In figure \ref{fig:L2-sisp}, the error in $L^2$ norm and the corresponding order of the error reduction under mesh refinement is plotted for equidistant mesh.\\
We observe the expected order $O(h^2)$ of error reduction for piecewise linear approximation, see e.g. \cite{ELS2000}.\\
\begin{minipage}[t]{40em}
\begin{tabular}{p{18em}p{22em}}
      \makeatletter{}\begingroup
  \makeatletter
  \providecommand\color[2][]{    \GenericError{(gnuplot) \space\space\space\@spaces}{      Package color not loaded in conjunction with
      terminal option `colourtext'    }{See the gnuplot documentation for explanation.    }{Either use 'blacktext' in gnuplot or load the package
      color.sty in LaTeX.}    \renewcommand\color[2][]{}  }  \providecommand\includegraphics[2][]{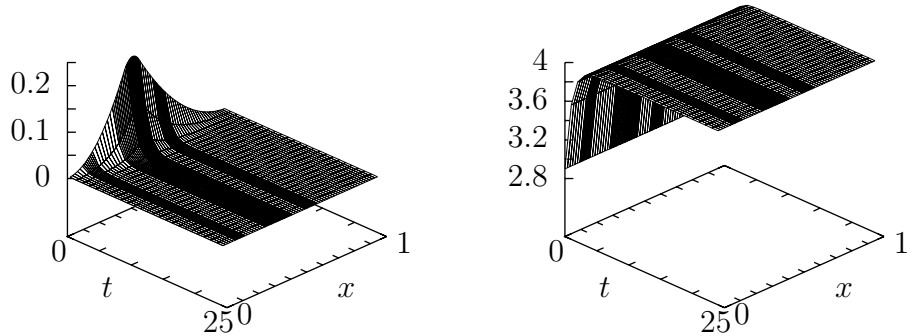}  \providecommand\rotatebox[2]{#2}  \@ifundefined{ifGPcolor}{    \newif\ifGPcolor
    \GPcolorfalse
  }{}  \@ifundefined{ifGPblacktext}{    \newif\ifGPblacktext
    \GPblacktexttrue
  }{}    \let\gplgaddtomacro\g@addto@macro
    \gdef\gplbacktext{}  \gdef\gplfronttext{}  \makeatother
  \ifGPblacktext
        \def\colorrgb#1{}    \def\colorgray#1{}  \else
        \ifGPcolor
      \def\colorrgb#1{\color[rgb]{#1}}      \def\colorgray#1{\color[gray]{#1}}      \expandafter\def\csname LTw\endcsname{\color{white}}      \expandafter\def\csname LTb\endcsname{\color{black}}      \expandafter\def\csname LTa\endcsname{\color{black}}      \expandafter\def\csname LT0\endcsname{\color[rgb]{1,0,0}}      \expandafter\def\csname LT1\endcsname{\color[rgb]{0,1,0}}      \expandafter\def\csname LT2\endcsname{\color[rgb]{0,0,1}}      \expandafter\def\csname LT3\endcsname{\color[rgb]{1,0,1}}      \expandafter\def\csname LT4\endcsname{\color[rgb]{0,1,1}}      \expandafter\def\csname LT5\endcsname{\color[rgb]{1,1,0}}      \expandafter\def\csname LT6\endcsname{\color[rgb]{0,0,0}}      \expandafter\def\csname LT7\endcsname{\color[rgb]{1,0.3,0}}      \expandafter\def\csname LT8\endcsname{\color[rgb]{0.5,0.5,0.5}}    \else
            \def\colorrgb#1{\color{black}}      \def\colorgray#1{\color[gray]{#1}}      \expandafter\def\csname LTw\endcsname{\color{white}}      \expandafter\def\csname LTb\endcsname{\color{black}}      \expandafter\def\csname LTa\endcsname{\color{black}}      \expandafter\def\csname LT0\endcsname{\color{black}}      \expandafter\def\csname LT1\endcsname{\color{black}}      \expandafter\def\csname LT2\endcsname{\color{black}}      \expandafter\def\csname LT3\endcsname{\color{black}}      \expandafter\def\csname LT4\endcsname{\color{black}}      \expandafter\def\csname LT5\endcsname{\color{black}}      \expandafter\def\csname LT6\endcsname{\color{black}}      \expandafter\def\csname LT7\endcsname{\color{black}}      \expandafter\def\csname LT8\endcsname{\color{black}}    \fi
  \fi
  \setlength{\unitlength}{0.0500bp}  \begin{picture}(5040.00,3528.00)    \gplgaddtomacro\gplfronttext{      \csname LTb\endcsname      \put(1122,704){\makebox(0,0)[r]{\strut{} 1e-07}}      \put(1122,944){\makebox(0,0)[r]{\strut{} 1e-06}}      \put(1122,1185){\makebox(0,0)[r]{\strut{} 1e-05}}      \put(1122,1425){\makebox(0,0)[r]{\strut{} 0.0001}}      \put(1122,1665){\makebox(0,0)[r]{\strut{} 0.001}}      \put(1122,1906){\makebox(0,0)[r]{\strut{} 0.01}}      \put(1122,2146){\makebox(0,0)[r]{\strut{} 0.1}}      \put(1122,2386){\makebox(0,0)[r]{\strut{} 1}}      \put(1122,2627){\makebox(0,0)[r]{\strut{} 10}}      \put(1122,2867){\makebox(0,0)[r]{\strut{} 100}}      \put(1254,484){\makebox(0,0){\strut{} 0}}      \put(1932,484){\makebox(0,0){\strut{} 5}}      \put(2610,484){\makebox(0,0){\strut{} 10}}      \put(3287,484){\makebox(0,0){\strut{} 15}}      \put(3965,484){\makebox(0,0){\strut{} 20}}      \put(4643,484){\makebox(0,0){\strut{} 25}}      \put(2948,154){\makebox(0,0){\strut{}$t$}}      \put(2948,3197){\makebox(0,0){\strut{}$\|u-u_{\text{ref}}\|_{L^2}$}}    }    \gplgaddtomacro\gplfronttext{      \csname LTb\endcsname      \put(3888,1683){\makebox(0,0)[r]{\strut{}$2^{-6}$}}      \csname LTb\endcsname      \put(3888,1463){\makebox(0,0)[r]{\strut{}$2^{-7}$}}      \csname LTb\endcsname      \put(3888,1243){\makebox(0,0)[r]{\strut{}$2^{-8}$}}      \csname LTb\endcsname      \put(3888,1023){\makebox(0,0)[r]{\strut{}$2^{-9}$}}      \csname LTb\endcsname      \put(3888,803){\makebox(0,0)[r]{\strut{}$2^{-10}$}}    }    \gplbacktext
    \put(0,0){\includegraphics{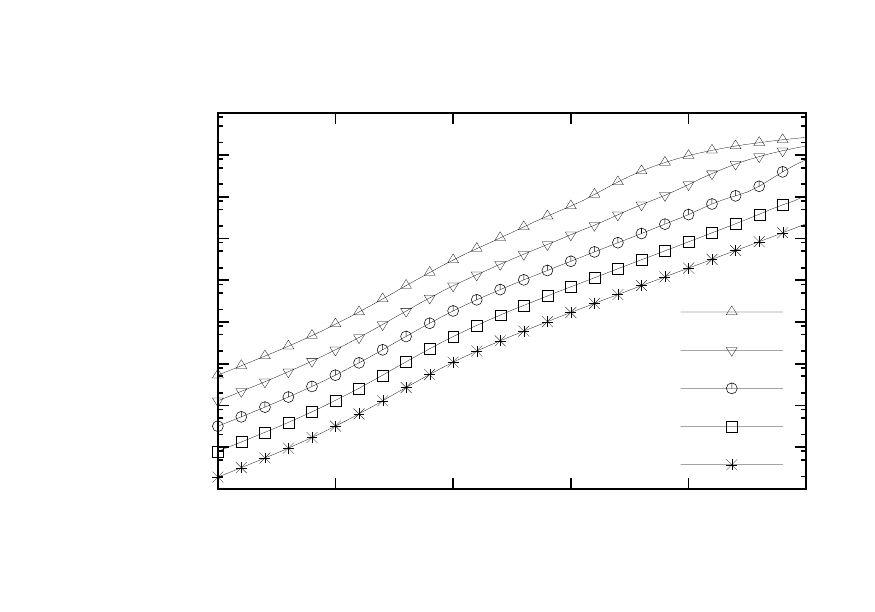}}    \gplfronttext
  \end{picture}\endgroup
  &  \makeatletter{}\begingroup
  \makeatletter
  \providecommand\color[2][]{    \GenericError{(gnuplot) \space\space\space\@spaces}{      Package color not loaded in conjunction with
      terminal option `colourtext'    }{See the gnuplot documentation for explanation.    }{Either use 'blacktext' in gnuplot or load the package
      color.sty in LaTeX.}    \renewcommand\color[2][]{}  }  \providecommand\includegraphics[2][]{    \GenericError{(gnuplot) \space\space\space\@spaces}{      Package graphicx or graphics not loaded    }{See the gnuplot documentation for explanation.    }{The gnuplot epslatex terminal needs graphicx.sty or graphics.sty.}    \renewcommand\includegraphics[2][]{}  }  \providecommand\rotatebox[2]{#2}  \@ifundefined{ifGPcolor}{    \newif\ifGPcolor
    \GPcolorfalse
  }{}  \@ifundefined{ifGPblacktext}{    \newif\ifGPblacktext
    \GPblacktexttrue
  }{}    \let\gplgaddtomacro\g@addto@macro
    \gdef\gplbacktext{}  \gdef\gplfronttext{}  \makeatother
  \ifGPblacktext
        \def\colorrgb#1{}    \def\colorgray#1{}  \else
        \ifGPcolor
      \def\colorrgb#1{\color[rgb]{#1}}      \def\colorgray#1{\color[gray]{#1}}      \expandafter\def\csname LTw\endcsname{\color{white}}      \expandafter\def\csname LTb\endcsname{\color{black}}      \expandafter\def\csname LTa\endcsname{\color{black}}      \expandafter\def\csname LT0\endcsname{\color[rgb]{1,0,0}}      \expandafter\def\csname LT1\endcsname{\color[rgb]{0,1,0}}      \expandafter\def\csname LT2\endcsname{\color[rgb]{0,0,1}}      \expandafter\def\csname LT3\endcsname{\color[rgb]{1,0,1}}      \expandafter\def\csname LT4\endcsname{\color[rgb]{0,1,1}}      \expandafter\def\csname LT5\endcsname{\color[rgb]{1,1,0}}      \expandafter\def\csname LT6\endcsname{\color[rgb]{0,0,0}}      \expandafter\def\csname LT7\endcsname{\color[rgb]{1,0.3,0}}      \expandafter\def\csname LT8\endcsname{\color[rgb]{0.5,0.5,0.5}}    \else
            \def\colorrgb#1{\color{black}}      \def\colorgray#1{\color[gray]{#1}}      \expandafter\def\csname LTw\endcsname{\color{white}}      \expandafter\def\csname LTb\endcsname{\color{black}}      \expandafter\def\csname LTa\endcsname{\color{black}}      \expandafter\def\csname LT0\endcsname{\color{black}}      \expandafter\def\csname LT1\endcsname{\color{black}}      \expandafter\def\csname LT2\endcsname{\color{black}}      \expandafter\def\csname LT3\endcsname{\color{black}}      \expandafter\def\csname LT4\endcsname{\color{black}}      \expandafter\def\csname LT5\endcsname{\color{black}}      \expandafter\def\csname LT6\endcsname{\color{black}}      \expandafter\def\csname LT7\endcsname{\color{black}}      \expandafter\def\csname LT8\endcsname{\color{black}}    \fi
  \fi
  \setlength{\unitlength}{0.0500bp}  \begin{picture}(5040.00,3528.00)    \gplgaddtomacro\gplfronttext{      \csname LTb\endcsname      \put(1122,704){\makebox(0,0)[r]{\strut{} 1e-09}}      \put(1122,1013){\makebox(0,0)[r]{\strut{} 1e-08}}      \put(1122,1322){\makebox(0,0)[r]{\strut{} 1e-07}}      \put(1122,1631){\makebox(0,0)[r]{\strut{} 1e-06}}      \put(1122,1940){\makebox(0,0)[r]{\strut{} 1e-05}}      \put(1122,2249){\makebox(0,0)[r]{\strut{} 0.0001}}      \put(1122,2558){\makebox(0,0)[r]{\strut{} 0.001}}      \put(1122,2867){\makebox(0,0)[r]{\strut{} 0.01}}      \put(1254,484){\makebox(0,0){\strut{} 0}}      \put(1932,484){\makebox(0,0){\strut{} 5}}      \put(2610,484){\makebox(0,0){\strut{} 10}}      \put(3287,484){\makebox(0,0){\strut{} 15}}      \put(3965,484){\makebox(0,0){\strut{} 20}}      \put(4643,484){\makebox(0,0){\strut{} 25}}      \put(2948,154){\makebox(0,0){\strut{}$t$}}      \put(2948,3197){\makebox(0,0){\strut{}$\|w-w_{\text{ref}}\|_{L^2}$}}    }    \gplgaddtomacro\gplfronttext{    }    \gplbacktext
    \put(0,0){\includegraphics{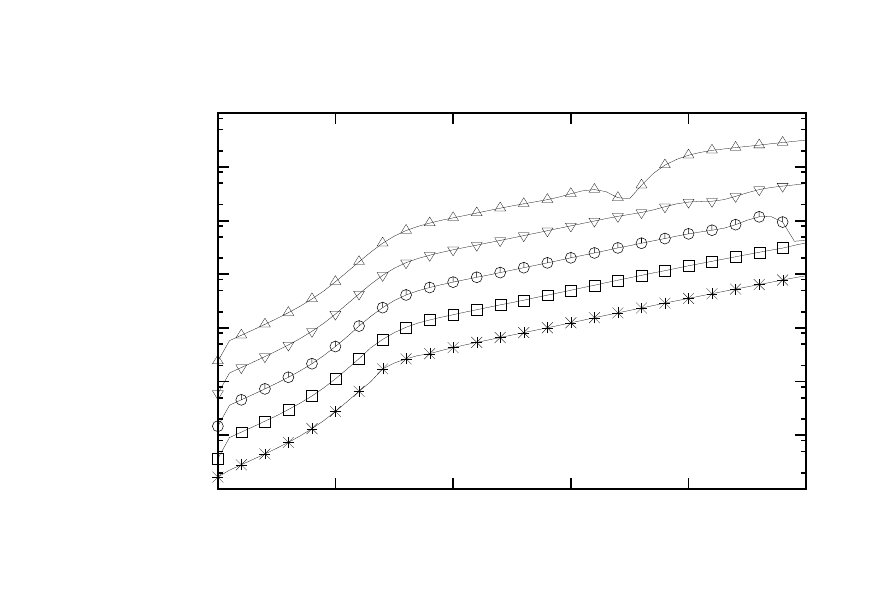}}    \gplfronttext
  \end{picture}\endgroup
 \\
   \makeatletter{}\begingroup
  \makeatletter
  \providecommand\color[2][]{    \GenericError{(gnuplot) \space\space\space\@spaces}{      Package color not loaded in conjunction with
      terminal option `colourtext'    }{See the gnuplot documentation for explanation.    }{Either use 'blacktext' in gnuplot or load the package
      color.sty in LaTeX.}    \renewcommand\color[2][]{}  }  \providecommand\includegraphics[2][]{    \GenericError{(gnuplot) \space\space\space\@spaces}{      Package graphicx or graphics not loaded    }{See the gnuplot documentation for explanation.    }{The gnuplot epslatex terminal needs graphicx.sty or graphics.sty.}    \renewcommand\includegraphics[2][]{}  }  \providecommand\rotatebox[2]{#2}  \@ifundefined{ifGPcolor}{    \newif\ifGPcolor
    \GPcolorfalse
  }{}  \@ifundefined{ifGPblacktext}{    \newif\ifGPblacktext
    \GPblacktexttrue
  }{}    \let\gplgaddtomacro\g@addto@macro
    \gdef\gplbacktext{}  \gdef\gplfronttext{}  \makeatother
  \ifGPblacktext
        \def\colorrgb#1{}    \def\colorgray#1{}  \else
        \ifGPcolor
      \def\colorrgb#1{\color[rgb]{#1}}      \def\colorgray#1{\color[gray]{#1}}      \expandafter\def\csname LTw\endcsname{\color{white}}      \expandafter\def\csname LTb\endcsname{\color{black}}      \expandafter\def\csname LTa\endcsname{\color{black}}      \expandafter\def\csname LT0\endcsname{\color[rgb]{1,0,0}}      \expandafter\def\csname LT1\endcsname{\color[rgb]{0,1,0}}      \expandafter\def\csname LT2\endcsname{\color[rgb]{0,0,1}}      \expandafter\def\csname LT3\endcsname{\color[rgb]{1,0,1}}      \expandafter\def\csname LT4\endcsname{\color[rgb]{0,1,1}}      \expandafter\def\csname LT5\endcsname{\color[rgb]{1,1,0}}      \expandafter\def\csname LT6\endcsname{\color[rgb]{0,0,0}}      \expandafter\def\csname LT7\endcsname{\color[rgb]{1,0.3,0}}      \expandafter\def\csname LT8\endcsname{\color[rgb]{0.5,0.5,0.5}}    \else
            \def\colorrgb#1{\color{black}}      \def\colorgray#1{\color[gray]{#1}}      \expandafter\def\csname LTw\endcsname{\color{white}}      \expandafter\def\csname LTb\endcsname{\color{black}}      \expandafter\def\csname LTa\endcsname{\color{black}}      \expandafter\def\csname LT0\endcsname{\color{black}}      \expandafter\def\csname LT1\endcsname{\color{black}}      \expandafter\def\csname LT2\endcsname{\color{black}}      \expandafter\def\csname LT3\endcsname{\color{black}}      \expandafter\def\csname LT4\endcsname{\color{black}}      \expandafter\def\csname LT5\endcsname{\color{black}}      \expandafter\def\csname LT6\endcsname{\color{black}}      \expandafter\def\csname LT7\endcsname{\color{black}}      \expandafter\def\csname LT8\endcsname{\color{black}}    \fi
  \fi
  \setlength{\unitlength}{0.0500bp}  \begin{picture}(5040.00,3528.00)    \gplgaddtomacro\gplfronttext{      \csname LTb\endcsname      \put(726,704){\makebox(0,0)[r]{\strut{} 0.5}}      \put(726,1137){\makebox(0,0)[r]{\strut{} 1}}      \put(726,1569){\makebox(0,0)[r]{\strut{} 1.5}}      \put(726,2002){\makebox(0,0)[r]{\strut{} 2}}      \put(726,2434){\makebox(0,0)[r]{\strut{} 2.5}}      \put(726,2867){\makebox(0,0)[r]{\strut{} 3}}      \put(858,484){\makebox(0,0){\strut{} 0}}      \put(1615,484){\makebox(0,0){\strut{} 5}}      \put(2372,484){\makebox(0,0){\strut{} 10}}      \put(3129,484){\makebox(0,0){\strut{} 15}}      \put(3886,484){\makebox(0,0){\strut{} 20}}      \put(4643,484){\makebox(0,0){\strut{} 25}}      \put(2750,154){\makebox(0,0){\strut{}$t$}}      \put(2750,3197){\makebox(0,0){\strut{}$\operatorname{Log}(\frac{\|u_{h}-u_{\text{ref}}\|_{L^2}}{\|u_{h/2}-u_{\text{ref}}\|_{L^2}})/\operatorname{Log}(2)$}}    }    \gplgaddtomacro\gplfronttext{      \csname LTb\endcsname      \put(1518,1537){\makebox(0,0)[r]{\strut{}$2^{-7}$}}      \csname LTb\endcsname      \put(1518,1317){\makebox(0,0)[r]{\strut{}$2^{-8}$}}      \csname LTb\endcsname      \put(1518,1097){\makebox(0,0)[r]{\strut{}$2^{-9}$}}      \csname LTb\endcsname      \put(1518,877){\makebox(0,0)[r]{\strut{}$2^{-10}$}}    }    \gplbacktext
    \put(0,0){\includegraphics{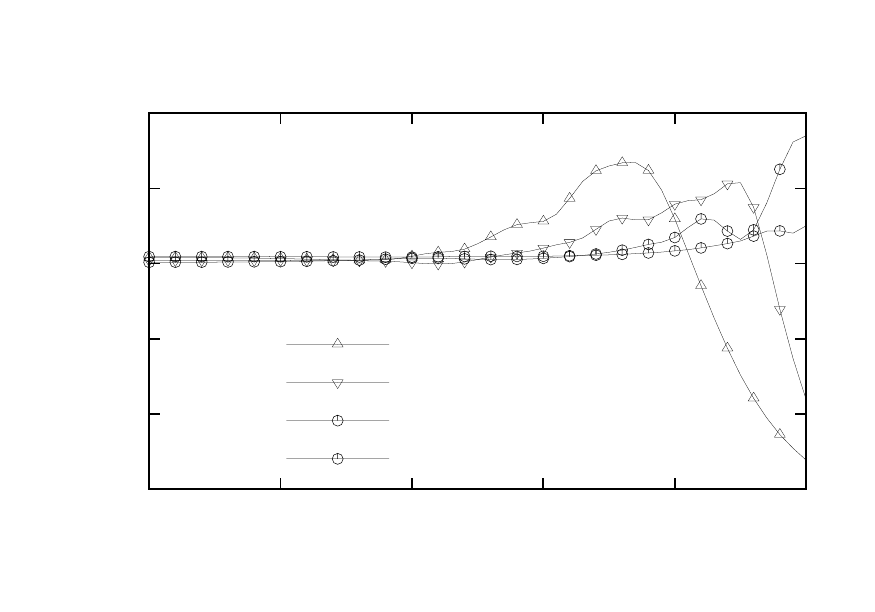}}    \gplfronttext
  \end{picture}\endgroup
   &  \makeatletter{}\begingroup
  \makeatletter
  \providecommand\color[2][]{    \GenericError{(gnuplot) \space\space\space\@spaces}{      Package color not loaded in conjunction with
      terminal option `colourtext'    }{See the gnuplot documentation for explanation.    }{Either use 'blacktext' in gnuplot or load the package
      color.sty in LaTeX.}    \renewcommand\color[2][]{}  }  \providecommand\includegraphics[2][]{    \GenericError{(gnuplot) \space\space\space\@spaces}{      Package graphicx or graphics not loaded    }{See the gnuplot documentation for explanation.    }{The gnuplot epslatex terminal needs graphicx.sty or graphics.sty.}    \renewcommand\includegraphics[2][]{}  }  \providecommand\rotatebox[2]{#2}  \@ifundefined{ifGPcolor}{    \newif\ifGPcolor
    \GPcolorfalse
  }{}  \@ifundefined{ifGPblacktext}{    \newif\ifGPblacktext
    \GPblacktexttrue
  }{}    \let\gplgaddtomacro\g@addto@macro
    \gdef\gplbacktext{}  \gdef\gplfronttext{}  \makeatother
  \ifGPblacktext
        \def\colorrgb#1{}    \def\colorgray#1{}  \else
        \ifGPcolor
      \def\colorrgb#1{\color[rgb]{#1}}      \def\colorgray#1{\color[gray]{#1}}      \expandafter\def\csname LTw\endcsname{\color{white}}      \expandafter\def\csname LTb\endcsname{\color{black}}      \expandafter\def\csname LTa\endcsname{\color{black}}      \expandafter\def\csname LT0\endcsname{\color[rgb]{1,0,0}}      \expandafter\def\csname LT1\endcsname{\color[rgb]{0,1,0}}      \expandafter\def\csname LT2\endcsname{\color[rgb]{0,0,1}}      \expandafter\def\csname LT3\endcsname{\color[rgb]{1,0,1}}      \expandafter\def\csname LT4\endcsname{\color[rgb]{0,1,1}}      \expandafter\def\csname LT5\endcsname{\color[rgb]{1,1,0}}      \expandafter\def\csname LT6\endcsname{\color[rgb]{0,0,0}}      \expandafter\def\csname LT7\endcsname{\color[rgb]{1,0.3,0}}      \expandafter\def\csname LT8\endcsname{\color[rgb]{0.5,0.5,0.5}}    \else
            \def\colorrgb#1{\color{black}}      \def\colorgray#1{\color[gray]{#1}}      \expandafter\def\csname LTw\endcsname{\color{white}}      \expandafter\def\csname LTb\endcsname{\color{black}}      \expandafter\def\csname LTa\endcsname{\color{black}}      \expandafter\def\csname LT0\endcsname{\color{black}}      \expandafter\def\csname LT1\endcsname{\color{black}}      \expandafter\def\csname LT2\endcsname{\color{black}}      \expandafter\def\csname LT3\endcsname{\color{black}}      \expandafter\def\csname LT4\endcsname{\color{black}}      \expandafter\def\csname LT5\endcsname{\color{black}}      \expandafter\def\csname LT6\endcsname{\color{black}}      \expandafter\def\csname LT7\endcsname{\color{black}}      \expandafter\def\csname LT8\endcsname{\color{black}}    \fi
  \fi
  \setlength{\unitlength}{0.0500bp}  \begin{picture}(5040.00,3528.00)    \gplgaddtomacro\gplfronttext{      \csname LTb\endcsname      \put(726,704){\makebox(0,0)[r]{\strut{} 0}}      \put(726,1013){\makebox(0,0)[r]{\strut{} 0.5}}      \put(726,1322){\makebox(0,0)[r]{\strut{} 1}}      \put(726,1631){\makebox(0,0)[r]{\strut{} 1.5}}      \put(726,1940){\makebox(0,0)[r]{\strut{} 2}}      \put(726,2249){\makebox(0,0)[r]{\strut{} 2.5}}      \put(726,2558){\makebox(0,0)[r]{\strut{} 3}}      \put(726,2867){\makebox(0,0)[r]{\strut{} 3.5}}      \put(858,484){\makebox(0,0){\strut{} 0}}      \put(1615,484){\makebox(0,0){\strut{} 5}}      \put(2372,484){\makebox(0,0){\strut{} 10}}      \put(3129,484){\makebox(0,0){\strut{} 15}}      \put(3886,484){\makebox(0,0){\strut{} 20}}      \put(4643,484){\makebox(0,0){\strut{} 25}}      \put(2750,154){\makebox(0,0){\strut{}$t$}}      \put(2750,3197){\makebox(0,0){\strut{}$\operatorname{Log}(\frac{\|w_{h}-w_{\text{ref}}\|_{L^2}}{\|u_{h/2}-u_{\text{ref}}\|_{L^2}})/\operatorname{Log}(2)$}}    }    \gplgaddtomacro\gplfronttext{    }    \gplbacktext
    \put(0,0){\includegraphics{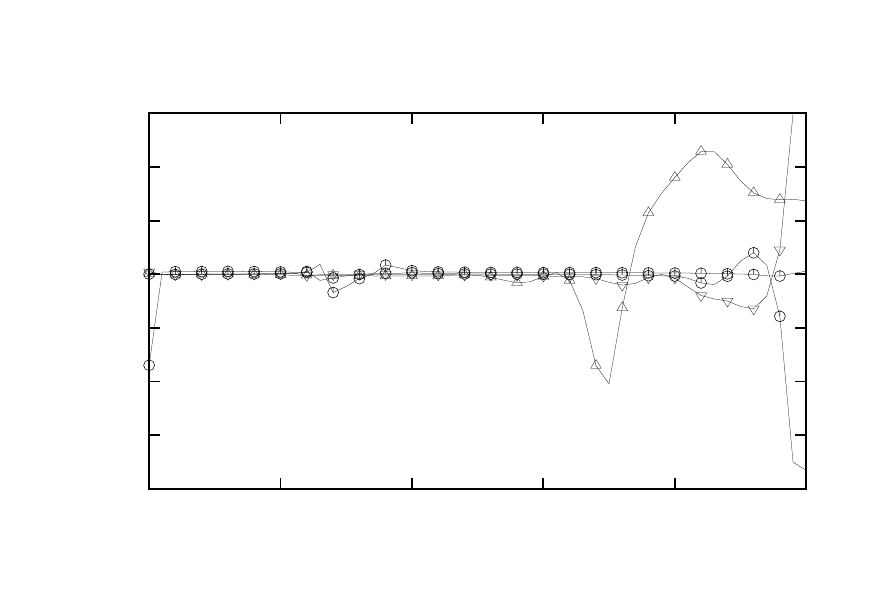}}    \gplfronttext
  \end{picture}\endgroup
 
\end{tabular}
\captionof{figure}{Upper row: Plot of the evolution of the $L^2$-error for a configuration shown in fig. \ref{fig:1s} and its $L^1$ norm shown in fig. \ref{fig:m-1s} in the sense of a reference solution. 
                   Lower row: Plot of the evolution of the order of error reduction. The reference solution was obtained on a mesh with spatial mesh size $h=2^{-13}$ and temporal mesh size $k=0.01$.}
\label{fig:L2-sisp}
\end{minipage}\\
\begin{minipage}[t]{40em}
\begin{tabular}{p{18em}p{22em}}
      \makeatletter{}\begingroup
  \makeatletter
  \providecommand\color[2][]{    \GenericError{(gnuplot) \space\space\space\@spaces}{      Package color not loaded in conjunction with
      terminal option `colourtext'    }{See the gnuplot documentation for explanation.    }{Either use 'blacktext' in gnuplot or load the package
      color.sty in LaTeX.}    \renewcommand\color[2][]{}  }  \providecommand\includegraphics[2][]{    \GenericError{(gnuplot) \space\space\space\@spaces}{      Package graphicx or graphics not loaded    }{See the gnuplot documentation for explanation.    }{The gnuplot epslatex terminal needs graphicx.sty or graphics.sty.}    \renewcommand\includegraphics[2][]{}  }  \providecommand\rotatebox[2]{#2}  \@ifundefined{ifGPcolor}{    \newif\ifGPcolor
    \GPcolorfalse
  }{}  \@ifundefined{ifGPblacktext}{    \newif\ifGPblacktext
    \GPblacktexttrue
  }{}    \let\gplgaddtomacro\g@addto@macro
    \gdef\gplbacktext{}  \gdef\gplfronttext{}  \makeatother
  \ifGPblacktext
        \def\colorrgb#1{}    \def\colorgray#1{}  \else
        \ifGPcolor
      \def\colorrgb#1{\color[rgb]{#1}}      \def\colorgray#1{\color[gray]{#1}}      \expandafter\def\csname LTw\endcsname{\color{white}}      \expandafter\def\csname LTb\endcsname{\color{black}}      \expandafter\def\csname LTa\endcsname{\color{black}}      \expandafter\def\csname LT0\endcsname{\color[rgb]{1,0,0}}      \expandafter\def\csname LT1\endcsname{\color[rgb]{0,1,0}}      \expandafter\def\csname LT2\endcsname{\color[rgb]{0,0,1}}      \expandafter\def\csname LT3\endcsname{\color[rgb]{1,0,1}}      \expandafter\def\csname LT4\endcsname{\color[rgb]{0,1,1}}      \expandafter\def\csname LT5\endcsname{\color[rgb]{1,1,0}}      \expandafter\def\csname LT6\endcsname{\color[rgb]{0,0,0}}      \expandafter\def\csname LT7\endcsname{\color[rgb]{1,0.3,0}}      \expandafter\def\csname LT8\endcsname{\color[rgb]{0.5,0.5,0.5}}    \else
            \def\colorrgb#1{\color{black}}      \def\colorgray#1{\color[gray]{#1}}      \expandafter\def\csname LTw\endcsname{\color{white}}      \expandafter\def\csname LTb\endcsname{\color{black}}      \expandafter\def\csname LTa\endcsname{\color{black}}      \expandafter\def\csname LT0\endcsname{\color{black}}      \expandafter\def\csname LT1\endcsname{\color{black}}      \expandafter\def\csname LT2\endcsname{\color{black}}      \expandafter\def\csname LT3\endcsname{\color{black}}      \expandafter\def\csname LT4\endcsname{\color{black}}      \expandafter\def\csname LT5\endcsname{\color{black}}      \expandafter\def\csname LT6\endcsname{\color{black}}      \expandafter\def\csname LT7\endcsname{\color{black}}      \expandafter\def\csname LT8\endcsname{\color{black}}    \fi
  \fi
  \setlength{\unitlength}{0.0500bp}  \begin{picture}(5040.00,3528.00)    \gplgaddtomacro\gplfronttext{      \csname LTb\endcsname      \put(1122,704){\makebox(0,0)[r]{\strut{} 1e-08}}      \put(1122,944){\makebox(0,0)[r]{\strut{} 1e-07}}      \put(1122,1185){\makebox(0,0)[r]{\strut{} 1e-06}}      \put(1122,1425){\makebox(0,0)[r]{\strut{} 1e-05}}      \put(1122,1665){\makebox(0,0)[r]{\strut{} 0.0001}}      \put(1122,1906){\makebox(0,0)[r]{\strut{} 0.001}}      \put(1122,2146){\makebox(0,0)[r]{\strut{} 0.01}}      \put(1122,2386){\makebox(0,0)[r]{\strut{} 0.1}}      \put(1122,2627){\makebox(0,0)[r]{\strut{} 1}}      \put(1122,2867){\makebox(0,0)[r]{\strut{} 10}}      \put(1254,484){\makebox(0,0){\strut{} 0}}      \put(1932,484){\makebox(0,0){\strut{} 5}}      \put(2610,484){\makebox(0,0){\strut{} 10}}      \put(3287,484){\makebox(0,0){\strut{} 15}}      \put(3965,484){\makebox(0,0){\strut{} 20}}      \put(4643,484){\makebox(0,0){\strut{} 25}}      \put(2948,154){\makebox(0,0){\strut{}$t$}}      \put(2948,3197){\makebox(0,0){\strut{}$\|u-u_{\text{ref}}\|_{L^1}$}}    }    \gplgaddtomacro\gplfronttext{      \csname LTb\endcsname      \put(3756,1727){\makebox(0,0)[r]{\strut{}$2^{-6}$}}      \csname LTb\endcsname      \put(3756,1507){\makebox(0,0)[r]{\strut{}$2^{-7}$}}      \csname LTb\endcsname      \put(3756,1287){\makebox(0,0)[r]{\strut{}$2^{-8}$}}      \csname LTb\endcsname      \put(3756,1067){\makebox(0,0)[r]{\strut{}$2^{-9}$}}      \csname LTb\endcsname      \put(3756,847){\makebox(0,0)[r]{\strut{}$2^{-10}$}}    }    \gplbacktext
    \put(0,0){\includegraphics{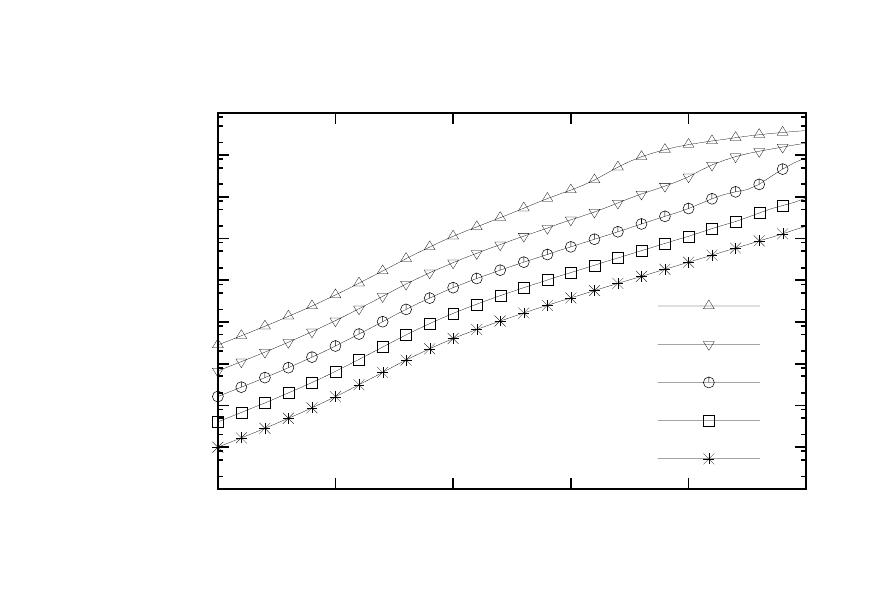}}    \gplfronttext
  \end{picture}\endgroup
  &  \makeatletter{}\begingroup
  \makeatletter
  \providecommand\color[2][]{    \GenericError{(gnuplot) \space\space\space\@spaces}{      Package color not loaded in conjunction with
      terminal option `colourtext'    }{See the gnuplot documentation for explanation.    }{Either use 'blacktext' in gnuplot or load the package
      color.sty in LaTeX.}    \renewcommand\color[2][]{}  }  \providecommand\includegraphics[2][]{    \GenericError{(gnuplot) \space\space\space\@spaces}{      Package graphicx or graphics not loaded    }{See the gnuplot documentation for explanation.    }{The gnuplot epslatex terminal needs graphicx.sty or graphics.sty.}    \renewcommand\includegraphics[2][]{}  }  \providecommand\rotatebox[2]{#2}  \@ifundefined{ifGPcolor}{    \newif\ifGPcolor
    \GPcolorfalse
  }{}  \@ifundefined{ifGPblacktext}{    \newif\ifGPblacktext
    \GPblacktexttrue
  }{}    \let\gplgaddtomacro\g@addto@macro
    \gdef\gplbacktext{}  \gdef\gplfronttext{}  \makeatother
  \ifGPblacktext
        \def\colorrgb#1{}    \def\colorgray#1{}  \else
        \ifGPcolor
      \def\colorrgb#1{\color[rgb]{#1}}      \def\colorgray#1{\color[gray]{#1}}      \expandafter\def\csname LTw\endcsname{\color{white}}      \expandafter\def\csname LTb\endcsname{\color{black}}      \expandafter\def\csname LTa\endcsname{\color{black}}      \expandafter\def\csname LT0\endcsname{\color[rgb]{1,0,0}}      \expandafter\def\csname LT1\endcsname{\color[rgb]{0,1,0}}      \expandafter\def\csname LT2\endcsname{\color[rgb]{0,0,1}}      \expandafter\def\csname LT3\endcsname{\color[rgb]{1,0,1}}      \expandafter\def\csname LT4\endcsname{\color[rgb]{0,1,1}}      \expandafter\def\csname LT5\endcsname{\color[rgb]{1,1,0}}      \expandafter\def\csname LT6\endcsname{\color[rgb]{0,0,0}}      \expandafter\def\csname LT7\endcsname{\color[rgb]{1,0.3,0}}      \expandafter\def\csname LT8\endcsname{\color[rgb]{0.5,0.5,0.5}}    \else
            \def\colorrgb#1{\color{black}}      \def\colorgray#1{\color[gray]{#1}}      \expandafter\def\csname LTw\endcsname{\color{white}}      \expandafter\def\csname LTb\endcsname{\color{black}}      \expandafter\def\csname LTa\endcsname{\color{black}}      \expandafter\def\csname LT0\endcsname{\color{black}}      \expandafter\def\csname LT1\endcsname{\color{black}}      \expandafter\def\csname LT2\endcsname{\color{black}}      \expandafter\def\csname LT3\endcsname{\color{black}}      \expandafter\def\csname LT4\endcsname{\color{black}}      \expandafter\def\csname LT5\endcsname{\color{black}}      \expandafter\def\csname LT6\endcsname{\color{black}}      \expandafter\def\csname LT7\endcsname{\color{black}}      \expandafter\def\csname LT8\endcsname{\color{black}}    \fi
  \fi
  \setlength{\unitlength}{0.0500bp}  \begin{picture}(5040.00,3528.00)    \gplgaddtomacro\gplfronttext{      \csname LTb\endcsname      \put(1122,704){\makebox(0,0)[r]{\strut{} 1e-09}}      \put(1122,1013){\makebox(0,0)[r]{\strut{} 1e-08}}      \put(1122,1322){\makebox(0,0)[r]{\strut{} 1e-07}}      \put(1122,1631){\makebox(0,0)[r]{\strut{} 1e-06}}      \put(1122,1940){\makebox(0,0)[r]{\strut{} 1e-05}}      \put(1122,2249){\makebox(0,0)[r]{\strut{} 0.0001}}      \put(1122,2558){\makebox(0,0)[r]{\strut{} 0.001}}      \put(1122,2867){\makebox(0,0)[r]{\strut{} 0.01}}      \put(1254,484){\makebox(0,0){\strut{} 0}}      \put(1932,484){\makebox(0,0){\strut{} 5}}      \put(2610,484){\makebox(0,0){\strut{} 10}}      \put(3287,484){\makebox(0,0){\strut{} 15}}      \put(3965,484){\makebox(0,0){\strut{} 20}}      \put(4643,484){\makebox(0,0){\strut{} 25}}      \put(2948,154){\makebox(0,0){\strut{}$t$}}      \put(2948,3197){\makebox(0,0){\strut{}$\|w-w_{\text{ref}}\|_{L^1}$}}    }    \gplgaddtomacro\gplfronttext{    }    \gplbacktext
    \put(0,0){\includegraphics{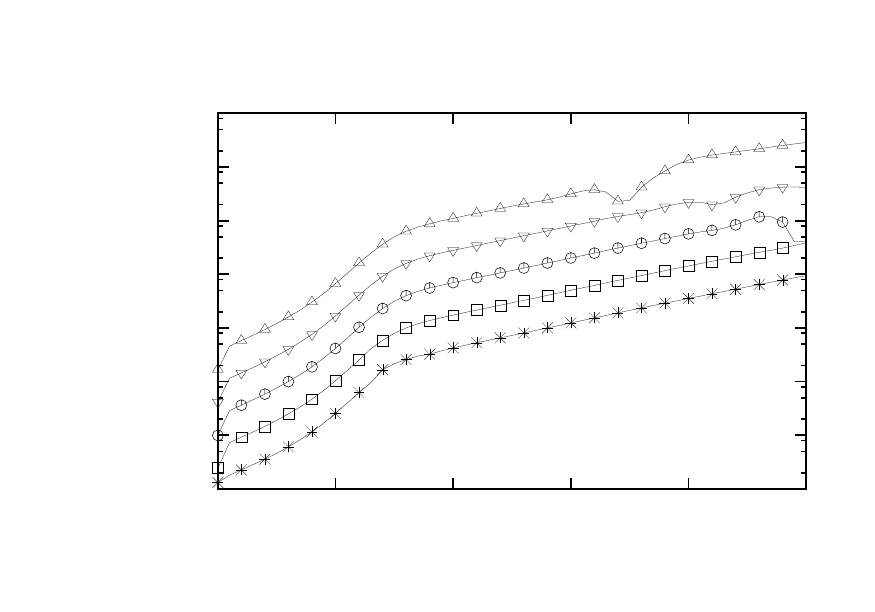}}    \gplfronttext
  \end{picture}\endgroup
 
\end{tabular}
\captionof{figure}{Plot of the evolution of the $L^1$-error for a configuration shown in fig. \ref{fig:1s} and its $L^1$ norm shown in fig. \ref{fig:m-1s} in the sense of a reference solution. 
                   The reference solution was obtained on a mesh with spatial mesh size $h=2^{-13}$ and temporal mesh size $k=0.01$.}
\label{fig:L1-sisp}
\end{minipage}\\

The same observation holds for the same configuration with smaller diffusion coefficient, s.t. growth of more than one spike occurs. The solution is shown in fig. \ref{behaviour-multispike-1}, the error in fig. \ref{fig:L2-multispike}.\\
\begin{minipage}[t]{40em}
\begin{tabular}{p{18em}p{22em}}
      \makeatletter{}\begingroup
  \makeatletter
  \providecommand\color[2][]{    \GenericError{(gnuplot) \space\space\space\@spaces}{      Package color not loaded in conjunction with
      terminal option `colourtext'    }{See the gnuplot documentation for explanation.    }{Either use 'blacktext' in gnuplot or load the package
      color.sty in LaTeX.}    \renewcommand\color[2][]{}  }  \providecommand\includegraphics[2][]{    \GenericError{(gnuplot) \space\space\space\@spaces}{      Package graphicx or graphics not loaded    }{See the gnuplot documentation for explanation.    }{The gnuplot epslatex terminal needs graphicx.sty or graphics.sty.}    \renewcommand\includegraphics[2][]{}  }  \providecommand\rotatebox[2]{#2}  \@ifundefined{ifGPcolor}{    \newif\ifGPcolor
    \GPcolorfalse
  }{}  \@ifundefined{ifGPblacktext}{    \newif\ifGPblacktext
    \GPblacktexttrue
  }{}    \let\gplgaddtomacro\g@addto@macro
    \gdef\gplbacktext{}  \gdef\gplfronttext{}  \makeatother
  \ifGPblacktext
        \def\colorrgb#1{}    \def\colorgray#1{}  \else
        \ifGPcolor
      \def\colorrgb#1{\color[rgb]{#1}}      \def\colorgray#1{\color[gray]{#1}}      \expandafter\def\csname LTw\endcsname{\color{white}}      \expandafter\def\csname LTb\endcsname{\color{black}}      \expandafter\def\csname LTa\endcsname{\color{black}}      \expandafter\def\csname LT0\endcsname{\color[rgb]{1,0,0}}      \expandafter\def\csname LT1\endcsname{\color[rgb]{0,1,0}}      \expandafter\def\csname LT2\endcsname{\color[rgb]{0,0,1}}      \expandafter\def\csname LT3\endcsname{\color[rgb]{1,0,1}}      \expandafter\def\csname LT4\endcsname{\color[rgb]{0,1,1}}      \expandafter\def\csname LT5\endcsname{\color[rgb]{1,1,0}}      \expandafter\def\csname LT6\endcsname{\color[rgb]{0,0,0}}      \expandafter\def\csname LT7\endcsname{\color[rgb]{1,0.3,0}}      \expandafter\def\csname LT8\endcsname{\color[rgb]{0.5,0.5,0.5}}    \else
            \def\colorrgb#1{\color{black}}      \def\colorgray#1{\color[gray]{#1}}      \expandafter\def\csname LTw\endcsname{\color{white}}      \expandafter\def\csname LTb\endcsname{\color{black}}      \expandafter\def\csname LTa\endcsname{\color{black}}      \expandafter\def\csname LT0\endcsname{\color{black}}      \expandafter\def\csname LT1\endcsname{\color{black}}      \expandafter\def\csname LT2\endcsname{\color{black}}      \expandafter\def\csname LT3\endcsname{\color{black}}      \expandafter\def\csname LT4\endcsname{\color{black}}      \expandafter\def\csname LT5\endcsname{\color{black}}      \expandafter\def\csname LT6\endcsname{\color{black}}      \expandafter\def\csname LT7\endcsname{\color{black}}      \expandafter\def\csname LT8\endcsname{\color{black}}    \fi
  \fi
  \setlength{\unitlength}{0.0500bp}  \begin{picture}(5040.00,3528.00)    \gplgaddtomacro\gplfronttext{      \csname LTb\endcsname      \put(1122,704){\makebox(0,0)[r]{\strut{} 1e-08}}      \put(1122,920){\makebox(0,0)[r]{\strut{} 1e-07}}      \put(1122,1137){\makebox(0,0)[r]{\strut{} 1e-06}}      \put(1122,1353){\makebox(0,0)[r]{\strut{} 1e-05}}      \put(1122,1569){\makebox(0,0)[r]{\strut{} 0.0001}}      \put(1122,1786){\makebox(0,0)[r]{\strut{} 0.001}}      \put(1122,2002){\makebox(0,0)[r]{\strut{} 0.01}}      \put(1122,2218){\makebox(0,0)[r]{\strut{} 0.1}}      \put(1122,2434){\makebox(0,0)[r]{\strut{} 1}}      \put(1122,2651){\makebox(0,0)[r]{\strut{} 10}}      \put(1122,2867){\makebox(0,0)[r]{\strut{} 100}}      \put(1254,484){\makebox(0,0){\strut{} 0}}      \put(1932,484){\makebox(0,0){\strut{} 5}}      \put(2610,484){\makebox(0,0){\strut{} 10}}      \put(3287,484){\makebox(0,0){\strut{} 15}}      \put(3965,484){\makebox(0,0){\strut{} 20}}      \put(4643,484){\makebox(0,0){\strut{} 25}}      \put(2948,154){\makebox(0,0){\strut{}$t$}}      \put(2948,3197){\makebox(0,0){\strut{}$\|u_h-u_{\text{ref}}\|_{L^2}$}}    }    \gplgaddtomacro\gplfronttext{      \csname LTb\endcsname      \put(1739,2666){\makebox(0,0)[r]{\strut{}$2^{-10}$}}      \csname LTb\endcsname      \put(1739,2446){\makebox(0,0)[r]{\strut{}$2^{-11}$}}      \csname LTb\endcsname      \put(1739,2226){\makebox(0,0)[r]{\strut{}$2^{-12}$}}      \csname LTb\endcsname      \put(1739,2006){\makebox(0,0)[r]{\strut{}$2^{-13}$}}      \csname LTb\endcsname      \put(1739,1786){\makebox(0,0)[r]{\strut{}$2^{-14}$}}    }    \gplbacktext
    \put(0,0){\includegraphics{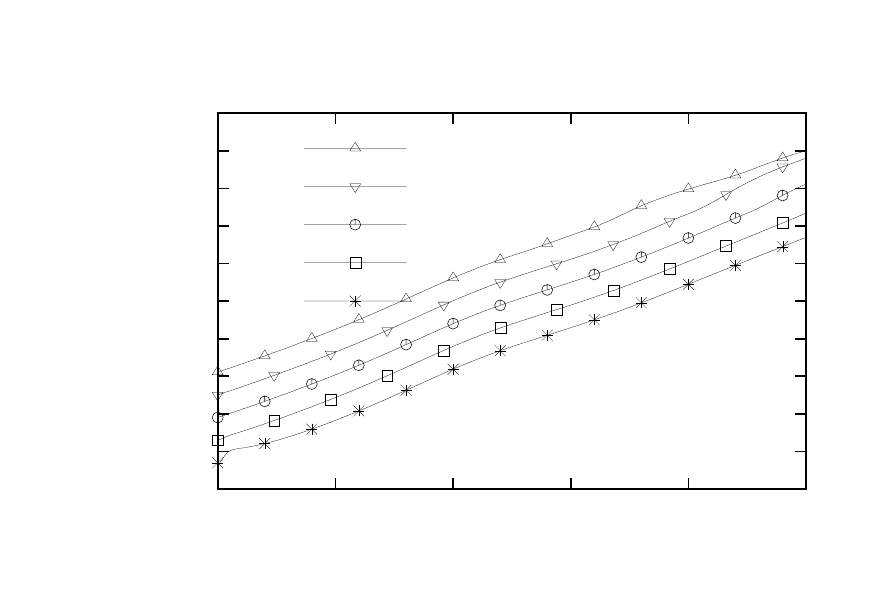}}    \gplfronttext
  \end{picture}\endgroup
 
&
      \makeatletter{}\begingroup
  \makeatletter
  \providecommand\color[2][]{    \GenericError{(gnuplot) \space\space\space\@spaces}{      Package color not loaded in conjunction with
      terminal option `colourtext'    }{See the gnuplot documentation for explanation.    }{Either use 'blacktext' in gnuplot or load the package
      color.sty in LaTeX.}    \renewcommand\color[2][]{}  }  \providecommand\includegraphics[2][]{    \GenericError{(gnuplot) \space\space\space\@spaces}{      Package graphicx or graphics not loaded    }{See the gnuplot documentation for explanation.    }{The gnuplot epslatex terminal needs graphicx.sty or graphics.sty.}    \renewcommand\includegraphics[2][]{}  }  \providecommand\rotatebox[2]{#2}  \@ifundefined{ifGPcolor}{    \newif\ifGPcolor
    \GPcolorfalse
  }{}  \@ifundefined{ifGPblacktext}{    \newif\ifGPblacktext
    \GPblacktexttrue
  }{}    \let\gplgaddtomacro\g@addto@macro
    \gdef\gplbacktext{}  \gdef\gplfronttext{}  \makeatother
  \ifGPblacktext
        \def\colorrgb#1{}    \def\colorgray#1{}  \else
        \ifGPcolor
      \def\colorrgb#1{\color[rgb]{#1}}      \def\colorgray#1{\color[gray]{#1}}      \expandafter\def\csname LTw\endcsname{\color{white}}      \expandafter\def\csname LTb\endcsname{\color{black}}      \expandafter\def\csname LTa\endcsname{\color{black}}      \expandafter\def\csname LT0\endcsname{\color[rgb]{1,0,0}}      \expandafter\def\csname LT1\endcsname{\color[rgb]{0,1,0}}      \expandafter\def\csname LT2\endcsname{\color[rgb]{0,0,1}}      \expandafter\def\csname LT3\endcsname{\color[rgb]{1,0,1}}      \expandafter\def\csname LT4\endcsname{\color[rgb]{0,1,1}}      \expandafter\def\csname LT5\endcsname{\color[rgb]{1,1,0}}      \expandafter\def\csname LT6\endcsname{\color[rgb]{0,0,0}}      \expandafter\def\csname LT7\endcsname{\color[rgb]{1,0.3,0}}      \expandafter\def\csname LT8\endcsname{\color[rgb]{0.5,0.5,0.5}}    \else
            \def\colorrgb#1{\color{black}}      \def\colorgray#1{\color[gray]{#1}}      \expandafter\def\csname LTw\endcsname{\color{white}}      \expandafter\def\csname LTb\endcsname{\color{black}}      \expandafter\def\csname LTa\endcsname{\color{black}}      \expandafter\def\csname LT0\endcsname{\color{black}}      \expandafter\def\csname LT1\endcsname{\color{black}}      \expandafter\def\csname LT2\endcsname{\color{black}}      \expandafter\def\csname LT3\endcsname{\color{black}}      \expandafter\def\csname LT4\endcsname{\color{black}}      \expandafter\def\csname LT5\endcsname{\color{black}}      \expandafter\def\csname LT6\endcsname{\color{black}}      \expandafter\def\csname LT7\endcsname{\color{black}}      \expandafter\def\csname LT8\endcsname{\color{black}}    \fi
  \fi
  \setlength{\unitlength}{0.0500bp}  \begin{picture}(5040.00,3528.00)    \gplgaddtomacro\gplfronttext{      \csname LTb\endcsname      \put(1122,704){\makebox(0,0)[r]{\strut{} 1e-09}}      \put(1122,1013){\makebox(0,0)[r]{\strut{} 1e-08}}      \put(1122,1322){\makebox(0,0)[r]{\strut{} 1e-07}}      \put(1122,1631){\makebox(0,0)[r]{\strut{} 1e-06}}      \put(1122,1940){\makebox(0,0)[r]{\strut{} 1e-05}}      \put(1122,2249){\makebox(0,0)[r]{\strut{} 0.0001}}      \put(1122,2558){\makebox(0,0)[r]{\strut{} 0.001}}      \put(1122,2867){\makebox(0,0)[r]{\strut{} 0.01}}      \put(1254,484){\makebox(0,0){\strut{} 0}}      \put(1932,484){\makebox(0,0){\strut{} 5}}      \put(2610,484){\makebox(0,0){\strut{} 10}}      \put(3287,484){\makebox(0,0){\strut{} 15}}      \put(3965,484){\makebox(0,0){\strut{} 20}}      \put(4643,484){\makebox(0,0){\strut{} 25}}      \put(2948,154){\makebox(0,0){\strut{}$t$}}      \put(2948,3197){\makebox(0,0){\strut{}$\|w-w_{\text{ref}}\|_{L^2}$}}    }    \gplgaddtomacro\gplfronttext{    }    \gplbacktext
    \put(0,0){\includegraphics{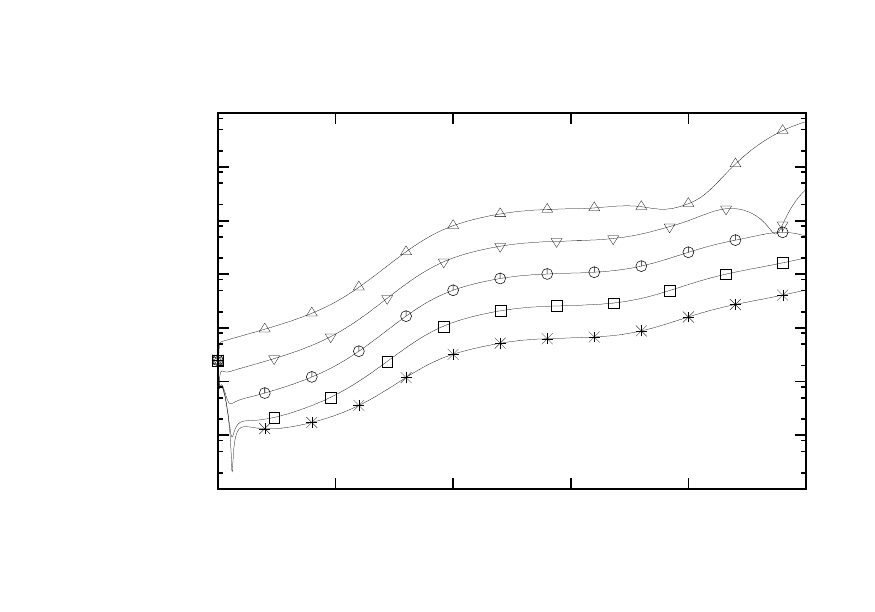}}    \gplfronttext
  \end{picture}\endgroup
 
\end{tabular}
\captionof{figure}{Plot of the evolution of the $L^2$-error for a configuration shown in fig. \ref{behaviour-multispike-1} in the sense of a reference solution. Fig. \ref{behaviour-multispike-1} shows the growth of multiple spikes due to a smaller diffusion coefficient. The reference solution was obtained on a mesh with spatial mesh size $h=2^{-15}$ and temporal mesh size $k=0.00025$.}
\label{fig:L2-multispike}
\end{minipage}\\
\begin{minipage}[t]{40em}
\begin{tabular}{p{18em}p{22em}}
      \makeatletter{}\begingroup
  \makeatletter
  \providecommand\color[2][]{    \GenericError{(gnuplot) \space\space\space\@spaces}{      Package color not loaded in conjunction with
      terminal option `colourtext'    }{See the gnuplot documentation for explanation.    }{Either use 'blacktext' in gnuplot or load the package
      color.sty in LaTeX.}    \renewcommand\color[2][]{}  }  \providecommand\includegraphics[2][]{    \GenericError{(gnuplot) \space\space\space\@spaces}{      Package graphicx or graphics not loaded    }{See the gnuplot documentation for explanation.    }{The gnuplot epslatex terminal needs graphicx.sty or graphics.sty.}    \renewcommand\includegraphics[2][]{}  }  \providecommand\rotatebox[2]{#2}  \@ifundefined{ifGPcolor}{    \newif\ifGPcolor
    \GPcolorfalse
  }{}  \@ifundefined{ifGPblacktext}{    \newif\ifGPblacktext
    \GPblacktexttrue
  }{}    \let\gplgaddtomacro\g@addto@macro
    \gdef\gplbacktext{}  \gdef\gplfronttext{}  \makeatother
  \ifGPblacktext
        \def\colorrgb#1{}    \def\colorgray#1{}  \else
        \ifGPcolor
      \def\colorrgb#1{\color[rgb]{#1}}      \def\colorgray#1{\color[gray]{#1}}      \expandafter\def\csname LTw\endcsname{\color{white}}      \expandafter\def\csname LTb\endcsname{\color{black}}      \expandafter\def\csname LTa\endcsname{\color{black}}      \expandafter\def\csname LT0\endcsname{\color[rgb]{1,0,0}}      \expandafter\def\csname LT1\endcsname{\color[rgb]{0,1,0}}      \expandafter\def\csname LT2\endcsname{\color[rgb]{0,0,1}}      \expandafter\def\csname LT3\endcsname{\color[rgb]{1,0,1}}      \expandafter\def\csname LT4\endcsname{\color[rgb]{0,1,1}}      \expandafter\def\csname LT5\endcsname{\color[rgb]{1,1,0}}      \expandafter\def\csname LT6\endcsname{\color[rgb]{0,0,0}}      \expandafter\def\csname LT7\endcsname{\color[rgb]{1,0.3,0}}      \expandafter\def\csname LT8\endcsname{\color[rgb]{0.5,0.5,0.5}}    \else
            \def\colorrgb#1{\color{black}}      \def\colorgray#1{\color[gray]{#1}}      \expandafter\def\csname LTw\endcsname{\color{white}}      \expandafter\def\csname LTb\endcsname{\color{black}}      \expandafter\def\csname LTa\endcsname{\color{black}}      \expandafter\def\csname LT0\endcsname{\color{black}}      \expandafter\def\csname LT1\endcsname{\color{black}}      \expandafter\def\csname LT2\endcsname{\color{black}}      \expandafter\def\csname LT3\endcsname{\color{black}}      \expandafter\def\csname LT4\endcsname{\color{black}}      \expandafter\def\csname LT5\endcsname{\color{black}}      \expandafter\def\csname LT6\endcsname{\color{black}}      \expandafter\def\csname LT7\endcsname{\color{black}}      \expandafter\def\csname LT8\endcsname{\color{black}}    \fi
  \fi
  \setlength{\unitlength}{0.0500bp}  \begin{picture}(5040.00,3528.00)    \gplgaddtomacro\gplfronttext{      \csname LTb\endcsname      \put(1122,704){\makebox(0,0)[r]{\strut{} 1e-08}}      \put(1122,920){\makebox(0,0)[r]{\strut{} 1e-07}}      \put(1122,1137){\makebox(0,0)[r]{\strut{} 1e-06}}      \put(1122,1353){\makebox(0,0)[r]{\strut{} 1e-05}}      \put(1122,1569){\makebox(0,0)[r]{\strut{} 0.0001}}      \put(1122,1786){\makebox(0,0)[r]{\strut{} 0.001}}      \put(1122,2002){\makebox(0,0)[r]{\strut{} 0.01}}      \put(1122,2218){\makebox(0,0)[r]{\strut{} 0.1}}      \put(1122,2434){\makebox(0,0)[r]{\strut{} 1}}      \put(1122,2651){\makebox(0,0)[r]{\strut{} 10}}      \put(1122,2867){\makebox(0,0)[r]{\strut{} 100}}      \put(1254,484){\makebox(0,0){\strut{} 0}}      \put(1932,484){\makebox(0,0){\strut{} 5}}      \put(2610,484){\makebox(0,0){\strut{} 10}}      \put(3287,484){\makebox(0,0){\strut{} 15}}      \put(3965,484){\makebox(0,0){\strut{} 20}}      \put(4643,484){\makebox(0,0){\strut{} 25}}      \put(2948,154){\makebox(0,0){\strut{}$t$}}      \put(2948,3197){\makebox(0,0){\strut{}$\|u-u_{\text{ref}}\|_{L^2}$}}    }    \gplgaddtomacro\gplfronttext{      \csname LTb\endcsname      \put(1739,2666){\makebox(0,0)[r]{\strut{}$2^{-10}$}}      \csname LTb\endcsname      \put(1739,2446){\makebox(0,0)[r]{\strut{}$2^{-11}$}}      \csname LTb\endcsname      \put(1739,2226){\makebox(0,0)[r]{\strut{}$2^{-12}$}}      \csname LTb\endcsname      \put(1739,2006){\makebox(0,0)[r]{\strut{}$2^{-13}$}}      \csname LTb\endcsname      \put(1739,1786){\makebox(0,0)[r]{\strut{}$2^{-14}$}}    }    \gplbacktext
    \put(0,0){\includegraphics{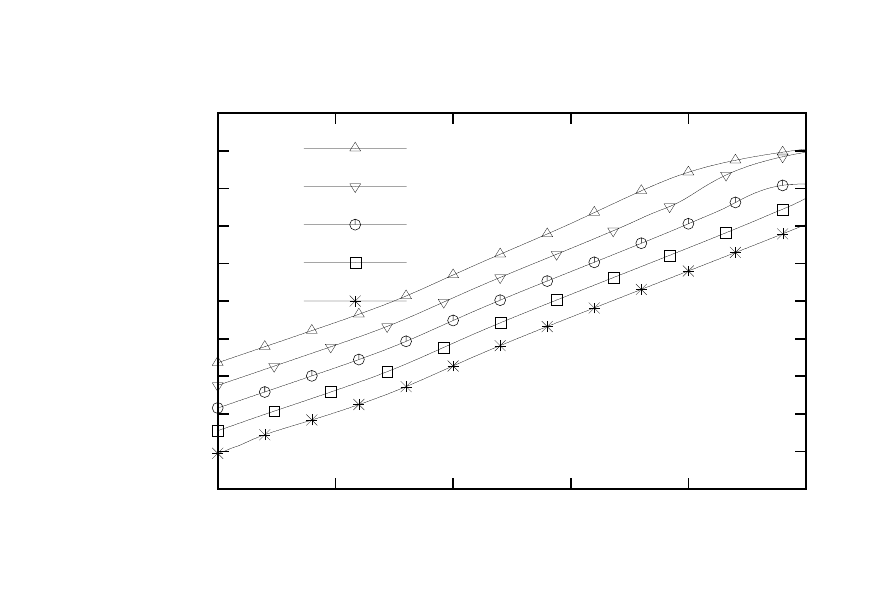}}    \gplfronttext
  \end{picture}\endgroup
 
&
      \makeatletter{}\begingroup
  \makeatletter
  \providecommand\color[2][]{    \GenericError{(gnuplot) \space\space\space\@spaces}{      Package color not loaded in conjunction with
      terminal option `colourtext'    }{See the gnuplot documentation for explanation.    }{Either use 'blacktext' in gnuplot or load the package
      color.sty in LaTeX.}    \renewcommand\color[2][]{}  }  \providecommand\includegraphics[2][]{    \GenericError{(gnuplot) \space\space\space\@spaces}{      Package graphicx or graphics not loaded    }{See the gnuplot documentation for explanation.    }{The gnuplot epslatex terminal needs graphicx.sty or graphics.sty.}    \renewcommand\includegraphics[2][]{}  }  \providecommand\rotatebox[2]{#2}  \@ifundefined{ifGPcolor}{    \newif\ifGPcolor
    \GPcolorfalse
  }{}  \@ifundefined{ifGPblacktext}{    \newif\ifGPblacktext
    \GPblacktexttrue
  }{}    \let\gplgaddtomacro\g@addto@macro
    \gdef\gplbacktext{}  \gdef\gplfronttext{}  \makeatother
  \ifGPblacktext
        \def\colorrgb#1{}    \def\colorgray#1{}  \else
        \ifGPcolor
      \def\colorrgb#1{\color[rgb]{#1}}      \def\colorgray#1{\color[gray]{#1}}      \expandafter\def\csname LTw\endcsname{\color{white}}      \expandafter\def\csname LTb\endcsname{\color{black}}      \expandafter\def\csname LTa\endcsname{\color{black}}      \expandafter\def\csname LT0\endcsname{\color[rgb]{1,0,0}}      \expandafter\def\csname LT1\endcsname{\color[rgb]{0,1,0}}      \expandafter\def\csname LT2\endcsname{\color[rgb]{0,0,1}}      \expandafter\def\csname LT3\endcsname{\color[rgb]{1,0,1}}      \expandafter\def\csname LT4\endcsname{\color[rgb]{0,1,1}}      \expandafter\def\csname LT5\endcsname{\color[rgb]{1,1,0}}      \expandafter\def\csname LT6\endcsname{\color[rgb]{0,0,0}}      \expandafter\def\csname LT7\endcsname{\color[rgb]{1,0.3,0}}      \expandafter\def\csname LT8\endcsname{\color[rgb]{0.5,0.5,0.5}}    \else
            \def\colorrgb#1{\color{black}}      \def\colorgray#1{\color[gray]{#1}}      \expandafter\def\csname LTw\endcsname{\color{white}}      \expandafter\def\csname LTb\endcsname{\color{black}}      \expandafter\def\csname LTa\endcsname{\color{black}}      \expandafter\def\csname LT0\endcsname{\color{black}}      \expandafter\def\csname LT1\endcsname{\color{black}}      \expandafter\def\csname LT2\endcsname{\color{black}}      \expandafter\def\csname LT3\endcsname{\color{black}}      \expandafter\def\csname LT4\endcsname{\color{black}}      \expandafter\def\csname LT5\endcsname{\color{black}}      \expandafter\def\csname LT6\endcsname{\color{black}}      \expandafter\def\csname LT7\endcsname{\color{black}}      \expandafter\def\csname LT8\endcsname{\color{black}}    \fi
  \fi
  \setlength{\unitlength}{0.0500bp}  \begin{picture}(5040.00,3528.00)    \gplgaddtomacro\gplfronttext{      \csname LTb\endcsname      \put(1122,704){\makebox(0,0)[r]{\strut{} 1e-09}}      \put(1122,1013){\makebox(0,0)[r]{\strut{} 1e-08}}      \put(1122,1322){\makebox(0,0)[r]{\strut{} 1e-07}}      \put(1122,1631){\makebox(0,0)[r]{\strut{} 1e-06}}      \put(1122,1940){\makebox(0,0)[r]{\strut{} 1e-05}}      \put(1122,2249){\makebox(0,0)[r]{\strut{} 0.0001}}      \put(1122,2558){\makebox(0,0)[r]{\strut{} 0.001}}      \put(1122,2867){\makebox(0,0)[r]{\strut{} 0.01}}      \put(1254,484){\makebox(0,0){\strut{} 0}}      \put(1932,484){\makebox(0,0){\strut{} 5}}      \put(2610,484){\makebox(0,0){\strut{} 10}}      \put(3287,484){\makebox(0,0){\strut{} 15}}      \put(3965,484){\makebox(0,0){\strut{} 20}}      \put(4643,484){\makebox(0,0){\strut{} 25}}      \put(2948,154){\makebox(0,0){\strut{}$t$}}      \put(2948,3197){\makebox(0,0){\strut{}$\|w-w_{\text{ref}}\|_{L^2}$}}    }    \gplgaddtomacro\gplfronttext{    }    \gplbacktext
    \put(0,0){\includegraphics{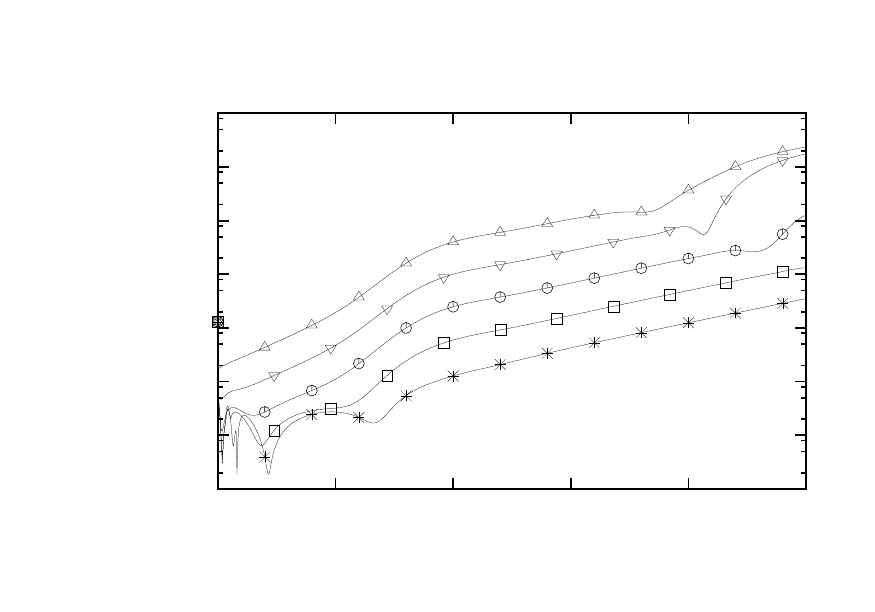}}    \gplfronttext
  \end{picture}\endgroup
 
\end{tabular}
\captionof{figure}{Plot of the evolution of the $L^2$-error for a configuration shown in fig. \ref{fig:cosxx} in the sense of a reference solution. Fig. \ref{fig:cosxx} shows the growth of multiple spikes due multiple maxima of the initial conditions of shape $u_0=\overline{u}+\text{cos}(2 \pi x^2)$. The reference solution was obtained on a mesh with spatial mesh size $h=2^{-16}$ and temporal mesh size $k=0.00025$.}
\label{fig:L2-cosxx}
\end{minipage}\\

\newpage
\begin{table}[ht]
\begin{center}
\begin{tabular}[H]{||l|c|c||}
  \hline \hline
  $s$ & $x_{t=0,\max}$ & $x_{t=25,\max}$ \\
  \hline \hline
  $0.2$  &$0.25$  &$0.2726$ \\
  $0.4$  &$0.417$ &$0.43237$\\
  $0.5$  &$0.5$   &$0.5$ \\
  $0.7$  &$0.66$  &$0.645$ \\
  $0.85$ &$0.792$ &$0.77$ \\
\hline \hline
\end{tabular}
\end{center}
\caption{Position $x_{max}$ of the arising spike ($t=25$) for initial conditions \eqref{IC:sh} with $\epsilon=0.1,\epsilon_1=0.05$, $D_w=6$ and maximum at $x_{t=0,\max}$. The shape of solutions are as in Figure \ref{fig:1s}, differing qualitatively only in the position of spike/sink. We observe that a spike grows close to the position of the maximum of the initial conditions.}
\label{tab:s_pos}
\end{table}

\begin{table}[ht]
\begin{center}
\begin{tabular}[H]{||l|c||}
  \hline \hline
  $D_1$ & spikes \\
  \hline \hline
  $D_{w,1} = 5.8541$ & 1 \\
  $\frac{1}{4}D_{w,1}$  & 2\\
  $\frac{1}{9} D_{w,1}$  & 3 \\
  $\frac{1}{16} D_{w,1}$ & 3 \\
  $\frac{1}{25} D_{w,1}$ & 4 \\
  $\frac{1}{36} D_{w,1}$ & 4 \\
\hline \hline
\end{tabular}
\end{center}
\caption{Number of spikes arising for different diffusion coefficients $D_1$, initial conditions \eqref{IC:sh} with $s=0.4, \epsilon_1=0.05, \epsilon=0.1$.}
\label{tab:D_Spi}
\end{table}

\begin{thebibliography}{99}

\bibitem{Ball}
Ball J.M. {Remarks on blow-up and nonexistence theorems for nonlinear evolution equations}. {\it Quart. J. Math. Oxford}. 1977; {\bf 28}.

\bibitem{deal.ii}
Bangerth W., Hartmann R., Kanschat G.
{deal.II -- a General Purpose Object Oriented Finite Element Library}. {\it ACM Trans. Math. Softw.}. 2007; {\bf 33}(4): 24/1--24/27.

\bibitem{ELS2000}
Estep D., Larson M., Williams R. {Estimating the Error of Numerical Solutions of Systems of Reaction-Diffusion Equations}. {\it Memoirs of the American Mathematical Society}. 2000; {\bf 146}, 696. 

\bibitem{G-M}
Gierer A., Meinhardt H. {A theory of biological pattern formation}. {\it Kybernetik}. 1972; {\bf 12}: 30--39.

\bibitem{H2011}
Haerting S. {Analysis and numerical simulation of the dynamics of pattern formation in a system of degenerated reaction-.diffusion equations}. {\it diploma thesis}. 2011; Fakult\"at f\"ur Mathematik und Informatik, Universit\"at Heidelberg.
 
\bibitem{Henry}
Henry D. {\it Geometric theory of semilinear parabolic equations}. {Springer-Verlag}: New York; 1981.

\bibitem{Keener}
Keener J. {Activators and Inhibitors in Pattern Formation}. {\it Studies in Applied Mathematics}. 1978; {\bf 59}: 1--23.

\bibitem{Baker}
Klika V., Baker R., Headon D., Gaffney E. {The Influence of Receptor-Mediated Interactions on Reaction-Diffusion Mechanisms of Cellular Self-organisation}. {\it Bulletin of Mathematical Biology}. 2012; {\bf 74}(4). 

\bibitem{Epstein}
Lengyel I., Epstein I.R. {A chemical approach to designing Turing patterns in reaction-diffusion systems}. {\it Proc. Natl. Acad. Sci. USA}. 1992; {\bf 89}: 3977--3979.

\bibitem{MC12r}
Marciniak-Czochra A. {\it Reaction-diffusion models of pattern formation in developmental biology}. In 'Mathematics and Life Sciences'  A. Antoniouk, E.V.N. Melnik.  De Gruyter: Germany; 2012: 189--212. 

\bibitem{MCK06} 
Marciniak-Czochra A., Kimmel M. {Dynamics of growth and signaling along linear and surface structures in very early tumors}. {\it Comput. Math. Methods Med.} 2006; {\bf 7}: 189--213.
 
\bibitem{MCK07}
Marciniak-Czochra A., Kimmel M. {Modelling of early lung cancer progression: influence of growth factor production and cooperation between partially transformed cells}. {\it Math. Models Methods Appl. Sci.} 2007; {\bf 17}: 1693--1719. 

\bibitem{MCK08}
Marciniak-Czochra A., Kimmel M. {Reaction-diffusion model of early carcinogenesis: the effects of influx of mutated cells}. {\it Math. Model. Nat. Phenom.} 2008; {\bf 3}: 90--114.

\bibitem{MCKS12} 
Marciniak-Czochra A., Karch G., Suzuki K. {Unstable patterns in reaction-diffusion model of early carcinogenesis}. {\it J.~Math. Pures et Appliquées}, DOI: 10.1016/j.matpur.2012.09.011.

\bibitem{MCKS13}
Marciniak-Czochra A., Karch G., Suzuki K. {Unstable patterns in autocatalytic reaction-diffusion-ODE systems}. {\it Preprint available at} http://arxiv.org/abs/1301.2002.

\bibitem{Murray} 
Murray J.D.   {\it Mathematical biology. II. Spatial models and biomedical applications. Third edition.} Interdisciplinary Applied Mathematics {\bf 18}. Springer-Verlag, New York; 2012.

\bibitem{Nishiura82}
Nishiura Y. {Global structure of bifurcating solutions of some reaction-diffusion systems}. {\it SIAM J. Math. Anal.} 1982; {\bf 13}(4).

\bibitem{Pierre}
Pierre M. {Global existence in reaction-diffusion systems with control of mass: a survey}, {\it Milan J. Math.} 2010; {\bf 78}: 417--455.

 \bibitem{R84}
Rothe F. {\it Global solutions of reaction-diffusion systems}. Lecture Notes in Mathematics, {\bf 1072}. Springer-Verlag: Berlin; 1984.
 
 \bibitem{Smoller}
Smoller J.  {\it Shock waves and reaction-diffusion equations}. Second edition. Grundlehren der Mathematischen Wissenschaften {\bf 258}.
Springer-Verlag: New York; 1984.
 
\bibitem{Turing}
Turing A.M. {The chemical basis of morphogenesis}. {\it Phil. Trans. Roy. Soc. B}. 1952; {\bf 237}: 37--72.

\bibitem{Yanagida}
Mizoguchi, Noriko and Ninomiya, Hirokazu and Yanagida, {Diffusion-Induced Blowup in a Nonlinear Parabolic System}. {\it Journal of Dynamics and Differential Equations.} 1998; {\bf 10}(4).

\bibitem{Epstein2}
Vanag V.K., Yang L., Dolnik M., Zhabotinsky A.M., Epstein I.R. {Oscillatory cluster patterns in a homogeneous chemical system with global feedback}. {Nature.} 2000; {\bf 406}(6794). 

\end{thebibliography}
\end{document}